\documentclass[a4paper, reqno, 11pt, notitlepage]{amsart}
\usepackage{fullpage}
\usepackage[utf8]{inputenc}

\usepackage{latexsym}
\usepackage{amsmath}
\usepackage{amssymb}
\usepackage{amsthm}
\usepackage{amscd}
\usepackage{mathrsfs}
\usepackage[all]{xy}
\usepackage{graphicx}
\usepackage{comment}
\usepackage{arydshln}
\usepackage{mathtools}
\usepackage{dsfont}
\usepackage{adjustbox}
\usepackage{multicol}
\usepackage{quiver}
\usepackage{changepage} 
\usepackage[activate={true,nocompatibility},final,tracking=true,kerning=true,spacing=true]{microtype}
\microtypecontext{spacing=nonfrench}
\usepackage{eucal}
\usepackage{yfonts}
\usepackage{mdframed}

\usepackage[loadshadowlibrary,color=green!40,size=scriptsize, shadow,linecolor=black!40]{todonotes}
\makeatletter

\renewcommand{\tocsection}[3]{%
  \indentlabel{\@ifnotempty{#2}{\bfseries\ignorespaces#1 #2\quad}}\bfseries#3}

\renewcommand{\tocsubsection}[3]{%
  \indentlabel{\@ifnotempty{#2}{\ignorespaces#1 #2\quad}}#3}
  
\renewcommand{\tocsubsubsection}[3]{%
  \indentlabel{\@ifnotempty{#2}{\ignorespaces#1 #2\quad}}#3}

\renewcommand\@dotsep{4.5}
\def\@tocline#1#2#3#4#5#6#7{\relax
  \ifnum #1>\c@tocdepth 
  \else
    \par \addpenalty\@secpenalty\addvspace{#2}%
    \begingroup \hyphenpenalty\@M
    \@ifempty{#4}{%
      \@tempdima\csname r@tocindent\number#1\endcsname\relax
    }{%
      \@tempdima#4\relax
    }%
    \parindent\z@ \leftskip#3\relax \advance\leftskip\@tempdima\relax
    \rightskip\@pnumwidth plus1em \parfillskip-\@pnumwidth
    #5\leavevmode\hskip-\@tempdima{#6}\nobreak
    \leaders\hbox{$\m@th\mkern \@dotsep mu\hbox{ }\mkern \@dotsep mu$}\hfill
    \nobreak
    \hbox to\@pnumwidth{\@tocpagenum{\ifnum#1=1\bfseries\fi#7}}\par
    \nobreak
    \endgroup
  \fi}
\AtBeginDocument{%
  \expandafter\renewcommand\csname r@tocindent0\endcsname{0pt}
}
\def\l@subsection{\@tocline{2}{0pt}{2.5pc}{5pc}{}}

\def\l@subsubsection{\@tocline{3}{0pt}{4.6pc}{7pc}{}}

\makeatother

\makeatletter
\def\subsubsection{\@startsection{paragraph}{4}%
    \z@\z@{-\fontdimen2\font}%
    {\normalfont\bfseries}}
\makeatother

\usepackage[T1]{fontenc}

\usepackage{tikz-cd}
\usepackage{nicefrac}
\usepackage{authblk}
\usepackage{textcomp}
\usepackage{graphicx}
\usepackage{framed}
\usepackage{color}

\usepackage[pagebackref,hypertexnames=false]{hyperref} 
\definecolor{darkblue}{rgb}{0,0,.85}
\definecolor{darkred}{rgb}{0.84,0,0}
\hypersetup{colorlinks = true,
  linkcolor = darkblue,
  urlcolor  = darkblue,
  citecolor = darkred,
  anchorcolor = darkblue,
  bookmarksdepth=3}

\usepackage{accents}
\usepackage[shortlabels]{enumitem}
\usepackage[foot]{amsaddr}
\usepackage{bm}
\usepackage{stmaryrd}
\usepackage{graphics}
\usepackage{mathtools}

\newtheorem{thm}{Theorem}[section]
\newtheorem{prop}[thm]{Proposition}

\newtheorem{thmi}{Theorem}
\newtheorem{conji}{Conjecture}
\makeatletter
\newcommand{\greekconji}[1]{%
  \ifcase#1
    \or \alpha\or \beta\or \gamma\or \delta\or \epsilon\or \zeta\or \eta%
    \or \theta\or \iota\or \kappa\or \lambda\or \mu\or \nu\or \xi\or o%
    \or \pi\or \rho\or \sigma\or \tau\or \upsilon\or \phi\or \chi\or \psi%
    \or \omega%
    \else #1%
  \fi
}

\makeatother

\newtheorem{lem}[thm]{Lemma}

\newtheorem{cor}[thm]{Corollary}
\newtheorem{conj}[thm]{Conjecture}
\newtheorem{hypo}[thm]{Hypothesis}

\theoremstyle{definition}
\newtheorem{construction}[thm]{Construction}

\newtheorem{example}[thm]{Example}

\newtheorem{nota}[thm]{Notation}

\newtheorem{defn}[thm]{Definition}
\newtheorem{eg}[thm]{Example}
\newtheorem{rem}[thm]{Remark}

\makeatletter
\newtheoremstyle{examplestyle}
  {1em}
  {1em}
  {\addtolength{\@totalleftmargin}{1.0em}
   \addtolength{\linewidth}{-1.0em}
   \parshape 1 1.0em \linewidth}
  {}
  {\bfseries}
  {.}
  {.5em}
  {}

\theoremstyle{examplestyle} 
\newtheorem{remi}{Remark}

\DeclareMathOperator{\Spec}{Spec}

\DeclareMathOperator{\Hom}{Hom}

\DeclareMathOperator{\Spf}{Spf}

\newcommand{\colim@}[2]{%
  \vtop{\m@th\ialign{##\cr
    \hfil$#1\operator@font colim$\hfil\cr
    \noalign{\nointerlineskip\kern1.5\ex@}#2\cr
    \noalign{\nointerlineskip\kern-\ex@}\cr}}%
}
\newcommand{\colim}{%
  \mathop{\mathpalette\colim@{}}\nmlimits@
}

\makeatletter
\def\subsubsection{\@startsection{subsubsection}{3}%
  \z@{.5\linespacing\@plus.7\linespacing}{-.5em}%
  {\normalfont\bfseries}}
\makeatother

\newcommand{\mysetminusD}{\hbox{\tikz{\draw[line width=0.6pt,line cap=round] (3pt,0) -- (0,6pt);}}}
\newcommand{\mysetminusT}{\mysetminusD}
\newcommand{\mysetminusS}{\hbox{\tikz{\draw[line width=0.45pt,line cap=round] (2pt,0) -- (0,4pt);}}}
\newcommand{\mysetminusSS}{\hbox{\tikz{\draw[line width=0.4pt,line cap=round] (1.5pt,0) -- (0,3pt);}}}

\newcommand{\mysetminus}{\mathbin{\mathchoice{\mysetminusD}{\mysetminusT}{\mysetminusS}{\mysetminusSS}}}

\newcommand{\mc}[1]{\mathcal{#1}}

\newcommand{\mf}[1]{\mathfrak{#1}}

\newcommand{\wh}{\widehat}
\newcommand{\wt}{\widetilde}

\newcommand{\mb}[1]{\mathbf{#1}}
\newcommand{\Z}{\mathbb{Z}}

\renewcommand{\ll}{\llbracket}
\newcommand{\rr}{\rrbracket}

\renewcommand{\j}{\jmath}

\newcommand{\globalnotationref}{\hyperref[s:global-notation]{global notation and conventions}}

\makeatletter
\renewcommand{\email}[2][]{%
  \ifx\emails\@empty\relax\else{\g@addto@macro\emails{,\space}}\fi%
  \@ifnotempty{#1}{\g@addto@macro\emails{\textrm{(#1)}\space}}%
  \g@addto@macro\emails{#2}%
}
\makeatother

\newcommand{\stacks}[1]{\cite[\href{https://stacks.math.columbia.edu/tag/#1}{Tag~#1}]{StacksProject}}

\newcommand{\cat}[1]{\mathbf{#1}}

\newcommand{\isomorphic}{\xrightarrow{\hspace{0.5mm} \sim \hspace{0.5mm}}}

\DeclareMathAlphabet{\pazocal}{OMS}{zplm}{m}{n}
\newcommand{\RGamma}{\textup{R}\Gamma}
\numberwithin{equation}{section} 

\def\Item(#1){\item[\llap{(}\refstepcounter{enumi}$\bullet$] #1]}

\DeclareMathAccent{\wtilde}{\mathord}{largesymbols}{"65}

\newcommand*\isomto{%
        \xrightarrow{\raisebox{-0.2 em}{\smash{\ensuremath{\sim}}}}%
    }
\newcommand*\isomfrom{%
        \xleftarrow{\raisebox{-0.2 em}{\smash{\ensuremath{\sim}}}}%
    }

\newcommand{\ov}[1]{\overline{#1}}

\newcommand{\Ainf}{\mathrm{A}_\mathrm{inf}}
\newcommand{\Acrys}{\mathrm{A}_\mathrm{crys}}

\newcommand{\defeq}{\vcentcolon=}
\newcommand{\eqdef}{=\vcentcolon}
\newcommand{\be}{\begin{equation*}}
\newcommand{\ee}{\end{equation*}}
\newcommand{\bx}{\begin{equation*}\xymatrix}
\newcommand{\ex}{\end{equation*}}

\usepackage{relsize}
\usepackage[bbgreekl]{mathbbol}
\usepackage{amsfonts}

\DeclareSymbolFontAlphabet{\mathbbl}{bbold}
\newcommand{\Prism}{{\mathlarger{\mathbbl{\Delta}}}}
\newcommand{\prism}{{\mathlarger{\mathbbl{\Delta}}}}

\renewcommand{\inf}{\mathrm{inf}}

\newcommand{\tw}[1]{{\mathsmaller{(#1)}}}
\newcommand{\smallprism}{{{\mathsmaller{\Prism}}}}
\newcommand{\smallan}{{{\mathsmaller{\mr{an}}}}}

\newcommand{\mr}{\mathrm}
\newcommand{\bb}{\mathbb}

\newcommand{\dR}{\mathrm{dR}}

    \makeatletter
    \def\paragraph{\@startsection{paragraph}{4}%
    \z@\z@{-\fontdimen2\font}%
    {\normalfont\bfseries}}
    \makeatother
  \newcommand{\crys}{\mathrm{crys}}

\title{An integral comparison of crystalline and de Rham cohomology}

\author{Abhinandan$^{(1)}$}
\address[1]{\scriptsize IMJ-PRG, Sorbonne Universit\'e,
    4 Place Jussieu, 75252 Paris, France}
\email[1]{\scriptsize abhinandan@imj-prg.fr}

\author{Alex Youcis$^{(2)}$}
\address[2]{Department of Mathematics, University of Toronto, Bahen Centre, Room 6290
40 St. George St., Toronto, ON, M5S 2E4, Canada}
\email[2]{\scriptsize alex.youcis@gmail.com}

\date{\today}

\begin{document}

\begin{abstract}
    Let $\mc{O}_K$ be a mixed characteristic complete DVR with perfect residue field $k$ and fraction field $K$.
    It is a celebrated result of Berthelot and Ogus that for a smooth proper formal scheme $X/\mc{O}_K$ there exists a comparison between the de Rham cohomology groups $\mr{H}^i_\mr{dR}(X/\mc{O}_K)$ and the crystalline cohomology groups $\mr{H}^i_\crys(X_k/W(k))$ of the special fibre, after tensoring with $K$.
    In this article, we use the stacky perspective on prismatic cohomology, due to Drinfeld and Bhatt--Lurie, to give a version of this comparison result with coefficients in a perfect complex of prismatic $F\textrm{-crystals}$ on $X$.
    Our method is of an integral nature and suggests new tools to understand the relationship between torsion in de Rham and crystalline cohomology.
\end{abstract}

\maketitle

\tableofcontents

\newpage

\newpage

\section{Introduction}

Let $k/\bb{F}_p$ be a perfect extension, set $W=W(k)$ and $K_0=\mr{Frac}(W)$, and let $K/K_0$ be a finite totally ramified extension of degree $e$ and with uniformiser $\pi$.
Fix $X$ to be a smooth proper formal $\mc{O}_K\textrm{-scheme}$.
One of the major goals of $p\textrm{-adic}$ Hodge theory is the comparison of various $p\textrm{-adic}$ cohomology theories attached to $X$, and for which one of the initial motivations was the crystalline-de Rham comparison of Berthelot--Ogus.

To explain this, recall that for a smooth proper morphism $f\colon T\to S$ of smooth complex varieties and a sufficiently small disk $D\subseteq S^{\mr{an}}$, the Gauss--Manin connection induces an isomorphism,
\begin{equation}\label{eq:GM-isom}\tag{1}
    \nabla_\mr{GM}\colon \mr{R}^if_\ast\Omega^\bullet_{T^\smallan/S^\smallan}|_D\simeq \mc{O}_D\otimes_\mathbb{C}\mr{H}^i_\mr{dR}(T_{s_0}/\mathbb{C}),
\end{equation}
for any $s_0\in D$.
Analogously, envisioning $\Spf(\mc{O}_K)$ as a small disk around $\Spec(k)$, one anticipates that the de Rham cohomology $\mr{H}^i_\mr{dR}(X/\mc{O}_K)$ depends only on the special fibre $X_k$.
The crystalline cohomology groups $\mr{H}^i_\crys(X_k/W)$ were defined to realise this idea (cf.\ \cite{BerCohcri, BerthelotOgus}).

Using the formalism of crystalline cohomology, the analogue of \eqref{eq:GM-isom} becomes the precise mathematical statement that there is a functorial isomorphism of $\mc{O}_K\textrm{-modules}$,
\begin{equation}\label{eq:crys-dR-comparison-low-ramification}\tag{2}
    \mr{H}^i_\mr{dR}(X/\mc{O}_K) \simeq \mr{H}^i_\crys(X_k/W) \otimes_W \mc{O}_K,\qquad \text{if } e \leqslant p-1.
\end{equation}
Thus, for small $e$, we do indeed see that $\mr{H}^i_\mr{dR}(X/\mc{O}_K)$ depends functorially only on $X_k$.

In their celebrated paper \cite{BerthelotOgusFIsocrystals}, Berthelot and Ogus observed that the analogue of \eqref{eq:crys-dR-comparison-low-ramification} \textit{always holds rationally}, regardless of $e$, i.e., there is a functorial isomorphism of $K\textrm{-vector}$ spaces,
\begin{equation}\label{eq:BO-intro}\tag{3}
    \mr{H}^i_\mr{dR}(X_K/K)=\mr{H}^i_\mr{dR}(X/\mc{O}_K)\otimes_{\mc{O}_K}K\simeq \mr{H}^i_\crys(X_k/W)\otimes_W K.
\end{equation}
Thus, the rational de Rham cohomology $\mr{H}^i_\mr{dR}(X_K/K)$ also depends functorially only on $X_k$.

The Berthelot--Ogus isomorphism has become nearly indispensable in arithmetic geometry.
But, this classical perspective is lacking in two significant ways:
\begin{enumerate}[leftmargin=.4in, label=(\Roman*)]
    \item the classical setup of the Berthelot--Ogus comparison does not allow for ``(integral) coefficients'',\label{item:coefficients}
    
    \item the method of proof does not lend itself to relating the \textit{integral structures} $\mr{H}^i_\mr{dR}(X/\mc{O}_K)$ and $\mr{H}^i_\mr{crys}(X_k/W)$, e.g.\ relating their \emph{torsion subgroups}.\label{item:integral}
\end{enumerate}
Our goal in this article is to explain how the recent advances in integral $p\textrm{-adic}$ Hodge theory, in the form of Bhatt and Scholze's prismatic cohomology (see \cite{BhattScholzePrisms}) and its stacky reinterpretation by Drinfeld and Bhatt--Lurie (see \cite{Drinfeld, BhattLurieAbsolute}), help address these insufficiencies.

For now, we state our analogue of the Berthelot--Ogus isomorphism with coefficients where we will use $\phi$ to denote Frobenius, in several different contexts.

\begin{thmi}\label{thmi:main}
    Let $\mc{E}$ be a perfect complex on $X_\smallprism$.
    Then, for $n\geqslant \lceil \log_p(\tfrac{e}{p-1})\rceil$ there exists a natural isomorphism of $n\textrm{-twisted}$ de Rham and crystalline cohomology groups,\footnote{The notion of $n$-twisted de Rham cohomology $\mr{H}^{i,(n)}_\dR(\mc{E})$ is new to this paper, see Definition \ref{defn:n-twisted-deRham}. Also, the crystalline realisation $\mc{E}_k^\crys$ is as in \cite[Theorem 6.4]{GuoReinecke}, a crystal on $(X_k/W)_\mr{crys}$.}
    \begin{equation*}
        \iota^{\tw{n}}\colon \phi^{\ast n}\mr{H}^i_\mr{crys}(\mc{E}_k^\crys)\otimes_W \mc{O}_K\simeq \mr{H}^{i,\tw{n}}_\mr{dR}(\mc{E}).
    \end{equation*}
    Moreover, a Frobenius structure $\varphi\colon \phi^\ast\mc{E}[\nicefrac{1}{\mc{I}_\smallprism}]\isomto\mc{E}[\nicefrac{1}{\mc{I}_\smallprism}]$ gives rise to a natural commutative diagram of isomorphisms,\footnote{In the following we use the de Rham realisation $\mc{E}^\dR$, a sheaf on the de Rham stack $(X/\mc{O}_K)^\dR$, see \S\ref{ss:twisted-de-Rham}.}
    \begin{equation*}
        \begin{tikzcd}[sep=2em]
            {\phi^{\ast n}\mr{H}^i_\mr{crys}(\mc{E}_k^\crys)\otimes_W K} & {\mr{H}^{i,\tw{n}}_\mr{dR}(\mc{E})\otimes_{\mc{O}_K}K}  \\
            {\mr{H}^i_\crys(X_k/W,\mc{E}_k^\crys)\otimes_W K} & {\mr{H}^i_\mr{dR}(X/\mc{O}_K,\mc{E}^\mr{dR})\otimes_{\mc{O}_K}K.} 
            \arrow["{\iota^{\tw{n}}}", from=1-1, to=1-2]
            \arrow["\sim"', from=1-1, to=1-2]
            \arrow["{\overline{\varphi}^{\tw{n}}_\crys}"', from=1-1, to=2-1]
            \arrow["\wr", from=1-1, to=2-1]
            \arrow["{\overline{\varphi}^{\tw{n}}_\mr{dR}}", from=1-2, to=2-2]
            \arrow["\wr"', from=1-2, to=2-2]
            \arrow["\sim", from=2-1, to=2-2]
        \end{tikzcd}
    \end{equation*}
    Thus, the cohomology groups $\mr{H}^{i,\tw{n}}_\dR(\mc{E})$ and $\mr{H}^i_\dR(X/\mc{O}_K,\mc{E}^\mr{dR})\otimes_{\mc{O}_K}K$ admit a $W\textrm{-descent}$ and a $K_0\textrm{-descent}$, respectively.
    So, these $\mc{O}_K\textrm{-modules}$ and $K\textrm{-spaces}$ only depend on the pair $(X_k, \mc{E}_k)$.
    Moreover, the bottom horizontal isomorphism is independent of $n$.
\end{thmi}

In the rest of this introduction we shall explain Theorem \ref{thmi:main} more precisely, and indicate how it provides new tools for comparing torsion in de Rham and crystalline cohomologies.

\medskip

\paragraph*{A prismatic analogue of Dwork's trick}

The restriction on $e$ appearing in \eqref{eq:crys-dR-comparison-low-ramification} can be motivated via the isomorphism \eqref{eq:GM-isom} and its setup: for $e\geqslant p$ the `radius' of $\Spf(\mc{O}_K)$, intuitively $p^{-\nicefrac{1}{e}}$, is too large to support a convergent connection.

Now, note that the `Frobenius' $\phi(x)=x^p$ on the open $p\textrm{-adic}$ unit disk $\bb{D}$ is contracting: $\phi^n(\bb{B}(r)) \subseteq \bb{B}(r^{p^n})$, where $\bb{B}(t)=\{|x|\leqslant t\}\subseteq \bb{D}$.
So, for a vector bundle $\mc{F}$ on $\bb{D}$, we have $\phi^{\ast n}(\mc{F})|_{\bb{B}(r)} = \phi^{\ast n}(\mc{F}|_{\bb{B}(r^{p^n})})$.
Thus, twisting by Frobenius increases the radius of triviality, an observation known as \emph{Dwork's trick}.

We would like to apply this idea to the `disk' $\Spf(\mc{O}_K)$.
By \eqref{eq:crys-dR-comparison-low-ramification} it would suffice to twist sufficiently many times to reduce from radius $p^{-\nicefrac{1}{e}}$ to radius $p^{-\nicefrac{1}{(p-1)}}$, i.e., twist $a=\lceil\log_p(\tfrac{e}{p-1})\rceil$ number of times (cf.\@ \cite[Section 3]{KatzTravauxDwork}).
And, indeed, inspection of the proof of \eqref{eq:BO-intro} reveals the use of a Dwork's-trick-like isomorphism:
\begin{equation}\label{eq:BO-dwork}\tag{4}
    X_{\pi^{\tilde{e}}=0}^{\tw{a}}\simeq X_k^{\tw{a}}\otimes_k \mc{O}_K/\pi^{\tilde{e}},
\end{equation}
where $\tilde{e} = \lceil\tfrac{e}{p-1}\rceil$ is like the ratio of the convergent radius and the radius of $\Spf(\mc{O}_K)$,
and $(-)^{\tw{a}}$ denotes the relevant Frobenius twist of the scheme.
Then, \eqref{eq:BO-intro} follows as the relative Frobenius maps $X_{\pi^{\tilde{e}}=0} \to X_{\pi^{\tilde{e}}=0}^{\tw{a}}$ and $X_k\to X_k^{(a)}$ induce isogenies on crystalline cohomology.

Several attempts to give precise foundations to these ideas exist, most notably Ogus' introduction of the convergent site (e.g.\ see \cite{Ogus-F-Isocrystals-II}). But, these fail to address \ref{item:integral} and (to a lesser extent) \ref{item:coefficients}, and moreover provide a less direct geometric interpretation of the above picture.

To remedy these issues, and in particular give a more direct geometric interpretation of Dwork's trick, we appeal to the stacky interpretation of prismatic cohomology.
From the morphism $f \colon X \to \Spf(\mc{O}_K)$ one produces a diagram of (formal) stacks,
\begin{equation*}
\begin{tikzcd}[column sep=large, row sep=16pt]
	{(X/\mc{O}_K)^\mr{dR}} & {X^\smallprism} & {(X_k/W)^\mr{crys}} \\
	{\Spf(\mc{O}_K)} & {\mathcal{O}_K^\smallprism} & {\Spf(W),}
	\arrow["{\rho_{X,\mr{dR}}}", from=1-1, to=1-2]
	\arrow["{f^\mr{dR}}", from=1-1, to=2-1]
	\arrow["{f^\smallprism}", from=1-2, to=2-2]
	\arrow["{\rho_{X,\crys}}"', from=1-3, to=1-2]
	\arrow["{f^\crys}", from=1-3, to=2-3]
	\arrow["{\rho_\mr{dR}}"', from=2-1, to=2-2]
	\arrow["{\rho_\mr{crys}}", from=2-3, to=2-2]
\end{tikzcd}
\end{equation*}
For a perfect complex $\mc{E}$ on $X^\smallprism$, i.e., a \emph{perfect complex of prismatic crystals}, we have:
\begin{equation*}
    \rho_\dR^\ast f^\smallprism_\ast\mc{E} \simeq f^\dR_\ast \mc{E}^\dR \eqdef \mr{R}\Gamma_\dR(\mc{E}^\dR),\qquad \rho_\crys^\ast f^\smallprism_\ast \mc{E}\simeq f^\crys_\ast\mc{E}_k^\crys\eqdef \mr{R}\Gamma_\crys(\mc{E}_k^\crys),
\end{equation*}
where we are suppressing derived notation and, by definition, $\mc{E}^\dR = \rho_{X,\dR}^\ast \mc{E}$ and $\mc{E}_k^\crys = \rho_{X,\crys}^\ast\mc{E}$.\footnote{If $\mc{E}$ is a vector bundle, then $\mc{E}^\dR$ is equivalent to a vector bundle with integrable connection on $X$, and $\mc{E}_k^\crys$ to a crystal on $(X_k/W)_\crys$. Moreover, their cohomologies recover the usual de Rham and crystalline cohomologies.}

Unlike $\Spf(\mc{O}_K)$, the stack $\mc{O}_K^\smallprism$ admits a \emph{Frobenius} $F_{\mc{O}_K}$, and so we can form the twisted maps,
\begin{equation*}
    \rho_\dR^{\tw{n}} \defeq F_{\mc{O}_K}^n \circ \rho_\dR,\qquad \rho_\crys^{\tw{n}} \defeq F_{\mc{O}_K}^n \circ \rho_\crys.
\end{equation*}
Then, we define the \emph{$n\textrm{-twisted}$ de Rham and crystalline cohomology complexes},
\begin{equation*}
    \mr{R}\Gamma^{\tw{n}}_\dR(\mc{E}) \defeq \left(\rho_\dR^{\tw{n}}\right)^\ast f^\smallprism_\ast\mc{E},\qquad \mr{R}\Gamma^{\tw{n}}_\crys\left(\mc{E}\right) \defeq \left(\rho_\crys^{\tw{n}}\right) f^\smallprism_\ast\mc{E} \simeq \phi^{\ast n}\mr{R}\Gamma_\crys(\mc{E}_k^\crys),
\end{equation*}
and the \emph{$n\textrm{-twisted}$ de Rham and crystalline cohomology groups} are their cohomology groups.
Subsequently, we view the following as a prismatic incarnation of Dwork's trick.

\begin{thmi}[{see Theorem \ref{thm:twisted-crys-dR-comparison}}]\label{thmi:twisted-dr-crys-comp}
    For any $n\geqslant a$, the following compositions,
    \begin{equation*}
        \Spf(\mc{O}_K) \xrightarrow{\,\,\rho^{\tw{n}}_\dR\,\,} \mc{O}_K^\smallprism \xleftarrow{\,\,\rho^{\tw{n}}_\crys\,\,} \Spf(W) \xleftarrow{\,\,\mr{str.}\,\,} \Spf(\mc{O}_K),
    \end{equation*}
    (where the rightmost map is the natural one) are naturally identified in $\mr{Map}(\Spf(\mc{O}_K),\mc{O}_K^\smallprism)$.
\end{thmi}

\begin{remi}
    When $e=1$, and thus $a=0$, Theorem \ref{thmi:twisted-dr-crys-comp} is due to Bhatt--Lurie when $X=\Spf(\bb{Z}_p)$ (see \cite[Proposition 3.6.6]{BhattLurieAbsolute}), and in general due to Imai--Kato--Youcis (see \cite[Theorem 3.19]{IKY3}).
    So, we really do view the main import of Theorem \ref{thmi:twisted-dr-crys-comp} as an analogue of Dwork's trick.
\end{remi}

From Theorem \ref{thmi:twisted-dr-crys-comp}, the first part of Theorem \ref{thmi:main} easily follows.
For the second part, the main observation is that a Frobenius structure $\varphi$ on $\mc{E}$ induces one on $Rf^\smallprism_\ast\mc{E}$, and that pulling it back along $\rho^{\tw{n}}_\dR$ and $\rho^{\tw{n}}_\crys$, induces isogenies labelled $\ov{\varphi}^{\tw{n}}_\dR$ and $\ov{\varphi}^{\tw{n}}_\crys$, as in loc.\@ cit., respectively.

\medskip

\paragraph*{A Breuil--Kisin reinterpretation}

One can make the analogy with Dwork's trick even stronger, and in doing so connect our work to previous incarnations of Dwork's trick appearing in integral $p\textrm{-adic}$ Hodge theory, e.g.\ those appearing in \cite[Section 5]{Hyodo-Kato}, \cite[Section 4.4]{Tsuji-semistable}, \cite[Section 2]{FaltingsRamified}, \cite[Section 6]{Breuil97} and \cite[Section 1.2]{KisinFCrystal}.

Namely, there exists a Frobenius-equivariant flat \emph{Breuil--Kisin} covering $\rho_\mf{S}\colon \Spf(\mf{S})\to \mc{O}_K^\smallprism$ where $\mf{S}\defeq W\ll u\rr$ is equipped with the Witt vector Frobenius on $W$ and $\phi(u)=u^p$.
We think of $\Spf(\mf{S})$ as a closed-times-open bi-disk, where $\Spf(W)$ has radius $p^{-1}$.
Let $E$ denote the minimal polynomial for $\pi$ relative to $W$.
Then, as $E$ has degree $e$ we indeed see that the closed embedding,
\begin{equation*}
    \begin{tikzcd}[column sep=large]
    	{\mr{Spf}(\mc{O}_K)} & {\Spf(\mf{S}).}
    	\arrow["{(E=0)}", hook, from=1-1, to=1-2]
    \end{tikzcd}
\end{equation*}
acts as a closed subdisk of radius $p^{-\nicefrac{1}{e}}$.
In the prismatic world, the role of an appropriately-sized closed subdisk (i.e., where we expect convergence) is played by a morphism,
\begin{equation*}
    \rho_{\tilde{S}} \colon \Spf(\tilde{S}) \to \mc{O}_K^\smallprism,\qquad \textrm{for } \tilde{S} \defeq \mf{S}\big[\big\{\tfrac{u^{m\tilde{e}}}{m!}\big\}_{m\geqslant 1}\big]^\wedge_p,
\end{equation*}
and where $\tilde{e}=\lceil\tfrac{e}{p-1}\rceil$ is the ratio of the convergent radius and the radius of $\Spf(\mc{O}_K)$.\footnote{The prism $(\tilde{S},(p))$ is a modification (better suited to the ramified situation) of the \emph{Breuil prism}.}

We then have the following diagram, which conveys the full strength of Dwork's trick:

\begin{equation}\label{eq:big-intro-diagram}\tag{5}
    \begin{tikzcd}
    	{\Spf(\mc{O}_K)} && {\mc{O}_K^\smallprism} && {\Spf(W)} \\
    	\\
    	{\Spf(\tilde{S})} && {\Spf(\mf{S})} && {\Spf(\tilde{S}),}
    	\arrow["{\rho^{(n)}_\mr{dR}}"{description}, from=1-1, to=1-3]
    	\arrow["{\mr{str.}}"{description}, curve={height=-20pt}, from=1-1, to=1-5]
    	\arrow["i"{description}, from=1-1, to=3-1]
    	\arrow["{\rho^{(n)}_\crys}"{description}, from=1-5, to=1-3]
    	\arrow["{(u=0)}"{description}, from=1-5, to=3-5]
    	\arrow["{\rho^{(n)}_{\tilde{S}}}"{description}, crossing over, from=3-1, to=1-3]
    	\arrow[from=3-1, to=3-3]
    	\arrow[curve={height=20pt}, equals, from=3-1, to=3-5]
    	\arrow["{\rho^{(n+1)}_\mf{S}}"{description}, from=3-3, to=1-3]
    	\arrow["{\rho^{(n)}_{\tilde{S}}}"{description}, from=3-5, to=1-3]
    	\arrow["z"{description}, shift right=2, curve={height=20pt}, dashed, from=3-5, to=1-5]
    	\arrow[from=3-5, to=3-3]
    \end{tikzcd}
\end{equation}
where $i$ corresponds to the map $\tilde{S} \twoheadrightarrow \mc{O}_K$ sending $u$ to $\pi$, the lower horizontal arrows correspond to the natural inclusion $\mathfrak{S} \hookrightarrow \tilde{S}$, $s$ comes from the natural map $W \to \tilde{S}$ which induces a map in the absolute prismatic site $(\mc{O}_K)_\smallprism$, and $\mr{str.}$ denotes the structure induced from the natural inclusion $W \hookrightarrow \mc{O}_K$.

A more refined version of Dwork's trick in the prismatic setting is the following.
\begin{thmi}[{see Proposition \ref{prop:crys-dR-comparison-using-prisms}}]\label{thmi:big-diagram-commmutes}
    For any uniformiser $\pi$ and any integer $n \geqslant a$, the diagram \eqref{eq:big-intro-diagram} is naturally $2\textrm{-commutative}$.
\end{thmi}

To explain the implications of Theorem \ref{thmi:big-diagram-commmutes}, fix a perfect complex $\mc{E}$ on $X^\smallprism$ and $n \geqslant a$.
For notational simplicity, we further write $\mf{M}^{\tw{n+1}}\defeq \left(\rho^{\tw{n+1}}_\mf{S}\right)^\ast f^\smallprism_\ast\mc{E}$ and $\tilde{M}^{\tw{n}} \defeq \big(\rho^{\tw{n}}_{\tilde{S}}\big)^\ast f^\smallprism_\ast\mc{E}$.
The commutativity of the left square of \eqref{eq:big-intro-diagram} implies that we have isomorphisms,
\begin{equation*}
    \left(\rho^{\tw{n}}_\dR\right)^\ast f_\ast^\smallprism\mc{E}\simeq \mf{M}^{\tw{n+1}}/E\simeq i^\ast \tilde{M}^{\tw{n}}.
\end{equation*}
The second identification here is the analogue of the equality $\phi^{\ast n}(\mc{F}|_{\bb{B}(r)})=\phi^{\ast n}(\mc{F}|_{\bb{B}(r^{p^n})})$ from the classical version of Dwork's trick.
The commutativity of the right square gives us identifications,
\begin{equation*}
    \left(\rho^{\tw{n}}_\crys\right)^\ast f^\smallprism_\ast\mc{E}\simeq \mf{M}^{\tw{n+1}}/u\simeq \tilde{M}^{\tw{n}}/u.
\end{equation*}
On the other hand, the commutativity of the right square with $z$ further gives us that,
\begin{equation}\label{eq:constancy}\tag{6}
    \tilde{M}^{\tw{n}}/u\simeq z^\ast (\rho^{\tw{n}}_\crys)^\ast f^\smallprism_\ast \mc{E}.
\end{equation}
The map $\rho^{\tw{n}}_\crys$ is obtained by applying $(-)^\smallprism$ to the map $\Spec(k)\xrightarrow{\pi=0} \Spf(\mc{O}_K)$ precomposed with the $(n+1)^\text{th}$ power of the Frobenius.
Thus, we view \eqref{eq:constancy} precisely as the expected \emph{constancy} after restricting to the smaller subdisk $\Spf(\tilde{S})$, as suggested by Dwork's trick.

Thus, we can ultimately give a refinement of Theorem \ref{thmi:main}.
Namely, there is: 
\begin{enumerate}[leftmargin=.5in, label=(\alph*)]
    \item a canonical isomorphism of perfect complexes over $\mc{O}_K$:
        \begin{equation}\label{eq:BK-reintepretation}\tag{7}
            \mf{M}^{\tw{n+1}}/E\simeq \mf{M}^{\tw{n+1}}/u\otimes_W \mc{O}_K,
        \end{equation}
    
    \item an identification of the source of \eqref{eq:BK-reintepretation} with $\left(\rho^{\tw{n}}_\dR\right)^\ast f_\ast^\smallprism\mc{E}$,
    
    \item the `constancy' of $\mf{M}^{\tw{n+1}}/u$, with constant value $\left(\rho^{\tw{n}}_\crys\right)^\ast f^\smallprism_\ast \mc{E}$.
\end{enumerate}

\medskip

\paragraph*{Implications for crystalline and de Rham torsion}

The isomorphism \eqref{eq:BO-intro} tells us that the (free) ranks of de Rham and crystalline cohomology must always agree.
A much more subtle question is to determine the relationship between their respective \emph{torsion submodules}.

For clarity, let us write
\begin{equation*}
    \mr{H}^i_\crys(X_k/W)[p^\infty] \simeq \textstyle\bigoplus_{i=1}^r W/p^{a_i},\qquad \mr{H}^i_\mr{dR}(X/\mc{O}_K)[p^\infty] \simeq \textstyle\bigoplus_{j=1}^s\mc{O}_K/\pi^{b_j}.
\end{equation*}
The equality $r=s$ holds (e.g.\ see \cite[Remark 7.8]{CesnaviciusKoshikawa}), but relating $\ell^i_\mr{crys}=\sum_{i=1}^r a_i$ and $\ell^i_\mr{dR}=\sum_{j=1}^sb_j$ seems difficult.
To wit, \cite[Question 7.13]{CesnaviciusKoshikawa} asks if $\ell^i_\mr{dR}\ne e\cdot\ell^i_\crys$ can occur.
We discuss such examples below (see Examples \ref{eg:Li--Petrov} and \ref{eg:BK}) giving an affirmative answer to this question, and we expect that equality does not hold `generically' for large $i$ and $e$. 

But, in fact, Theorem \ref{thmi:main} and the isomorphism \eqref{eq:BK-reintepretation} give us tools to study the precise relationship.
\begin{conji}[{see Conjecture \ref{conj:main-inequality}}]\label{conji:main-inequality} 
    The inequality $\ell^i_\mr{crys}\leqslant \ell^i_\mr{dR}\leqslant e\cdot\ell^i_\mr{crys}$ always holds.
\end{conji}

Conjecture \ref{conji:main-inequality} seems difficult even in the simplest non-trivial cases.
But, in Section \ref{ss:conjectural-framework} we explain how Dwork's trick, specifically the isomorphism \eqref{eq:BK-reintepretation}, allows us to transform Conjecture \ref{conji:main-inequality} into questions about $u^\infty\textrm{-torsion}$ in prismatic cohomology, which possesses finer structures given the extensive study in recent years (e.g.\ \cite{li-liu-uinfty, GabberLi}).

In Appendix \ref{app:torsion-calculations}, we give evidence for these conjectural properties of the $u^\infty\textrm{-torsion}$ in Breuil--Kisin cohomology, by verifying that our main such conjecture (see Conjecture \ref{conj:beta}) holds in a small number of non-trivial cases (see Theorem \ref{thm:conj-beta-small-index-and-ramification}).

\medskip

\paragraph*{Acknowledgements} 

The authors thank Piotr Achinger, Hiroki Kato, Tong Liu, and Takeshi Tsuji for helpful discussions.
We especially thank Naoki Imai for discussing ideas related to Conjecture \ref{conji:main-inequality}, Shizhang Li for patiently answering many questions of ours, and Dat Pham for pointing out a mistake in an earlier version of this paper.
This collaboration started while the second named author was hosted at the University of Tokyo by Naoki Imai, and we thank him for his hospitality. Finally, we thank the anonymous referee for many helpful suggestions and corrections that increased the accuracy and readability of this paper.

The work of the first named author is partially supported by a Simons Collaboration grant on Perfection in Algebra, Geometry and Topology, and partially by JSPS KAKENHI grant numbers 22F22711 and 22KF0094.

\medskip

\phantomsection
\label{s:global-notation}
\noindent\textbf{Global notation and conventions}

Throughout this article we fix the following notation.
\begin{multicols}{2}
    \begin{itemize}[label=$\diamond$,leftmargin=.6cm]
        \item $p$ is a prime, 
        \item $k$ is a perfect field of characteristic $p$, 
        \item $F_X$ (resp.\@ $F_R$) is the absolute Frobenius on an $\bb{F}_p\textrm{-scheme}$ $X$ (resp.\@ $\bb{F}_p\textrm{-algebra}$ $R$),
        \item for a ring $R$ and $R\textrm{-module}$ $M$ we write its length as $\ell_R(M)$,
        \item $W \defeq W(k)$ the ring of Witt vectors, 
        \item $K_0=\mr{Frac}(W)$, \item $K$ a finite totally ramified extension of $K_0$,
        \item $\mc{O}_K$ the ring of integers of $K$,
        \item $\pi$ a uniformiser of $K$,
        \item $E_\pi=E\defeq \mr{min.poly}_{\mc{O}_K/W}(\pi)\in W[u]$,
        \item $e=[K\colon K_0]$,
        \item $\tilde{e}=\lceil \tfrac{e}{p-1}\rceil$,
        \item $a=\lceil \log_p\big(\tfrac{e}{p-1}\big) \rceil=\lceil \log_p(\tilde{e})\rceil$.
        \item $\ov{K}$ is an algebraic closure of $K$ and $C=\widehat{\ov{K}}$, 
        \item $\Ainf=\Ainf(\mc{O}_C)$,
        \item $\Acrys=\Acrys(\mc{O}_C)$,
        \item $\pi^\flat \defeq (\pi, \pi^{1/p}, \cdots)$,
        \item $\theta\colon \Ainf\to \mc{O}_C$ is Fontaine's map,
        \item $\theta^{\tw{\text{-}1}}\defeq \theta\circ \phi^{-1}$,
        \item $\xi=p-[p^\flat]$ and ${\xi}^{\tw{1}}=\phi(\xi)$.
    \end{itemize}
\end{multicols}

We use the phrase \emph{formal stack over $\mathbb{Z}_p$} to mean a stack on the category of $p\textrm{-nilpotent}$ rings, equipped with the fpqc topology.
For a stack $S$ on all rings with the fpqc topology, we denote by $\wh{S}$ its $p\textrm{-adic}$ completion, which means its restriction to the full subcategory of $p\textrm{-nilpotent}$ rings.
A morphism of formal stacks $f \colon X \to Y$ over $\mathbb{Z}_p$ is called a \emph{flat surjection} if for every $p\textrm{-nilpotent}$ ring $R$ and object $y$ of $Y(R)$, there exists a faithfully flat map $R \to R'$ and an object $x'$ of $X(R')$ such that $f(x')\simeq y|_{R'}$ is in $Y(R')$.

\section{Preliminaries on prismatic theory}

In this section, we give a minimal recollection on parts of the prismatic theory we need in the sequel.
We do this mainly as a means to set notation and to refer the reader to more in-depth discussions (e.g.\ \cite{BhattNotes, BhattLurieAbsolute, BhattLuriePrismatization, BhattScholzePrisms, BhattScholzeCrystals, Drinfeld}) for details.

In addition to \globalnotationref, we use the following notation in this section.

\begin{nota}\label{nota:Witt-vector-stuff}
    Let us fix the following notation:
    \begin{multicols}{2}
        \begin{itemize}[leftmargin=.2in]
            \item $X$ is a quasi-syntomic $p\textrm{-adic}$ formal scheme, 
            \item $\bb{W}$ is the $p$-typical Witt vector group scheme, 
            \item $\delta\colon \bb{W}\to \bb{W}$ the usual $\delta$-structure,
            \item $F\colon \bb{W}\to\bb{W}$ is the usual Frobenius lift,
            \item $V\colon \bb{W}\to \bb{W}$ is the usual Verschiebung map,
            \item $[-]\colon \wh{\bb{A}}^1\to \bb{W}$ the Teichm\"{u}ller lift,
            \item $\gamma_0\colon \bb{W}\to \bb{G}_a$ the zeroth component map.
        \end{itemize}
    \end{multicols}
    Finally, we shall freely use the notion of quasi-ideals $d \colon I \to A$ (which we sometimes write as $[I\to A]$ for clarity) and their quotients $\mr{Cone}(d)$ (where we occasionally write $A/I$) as in \cite{DrinfeldRingGroupoids}.
\end{nota}

\subsection{Prismatisation and prismatic ($F\textrm{-)crystals}$}

We begin by recalling the prismatisation of $X$, a formal stack over $\mathbb{Z}_p$, due to Drinfeld and Bhatt--Lurie.

\begin{defn}[Prismatisation, {\cite[Definition 5.1.6]{BhattNotes}}]
    Let us recall the following:
    \begin{enumerate}
        \item For a $p\textrm{-nilpotent}$ ring $R$ a \emph{Cartier--Witt divisor over $R$} is a map of invertible $\bb{W}(R)\textrm{-modules}$ $\alpha\colon I\to \bb{W}(R)$ such that $\gamma_0(\alpha(I))$ is nilpotent and $\delta(\alpha(I))$ generates $\bb{W}(R)$.
        
        \item The \emph{prismatisation} $X^\smallprism$ of $X$ is the formal stack over $\mathbb{Z}_p$ associating to any $p\textrm{-nilpotent}$ ring $R$ the groupoid of pairs $(\alpha \colon I\to\bb{W}(R), s)$ where $\alpha$ is a Cartier--Witt divisor over $R$, and $s \colon \Spec(\bb{W}(R)/I)\to X$ is a morphism of (derived) schemes.\footnote{Technically, $\bb{W}(R)/I = \mr{cone}(\alpha)$ need not be a discrete ring and, in general, it is a $1\textrm{-truncated}$ animated ring. See \cite[Paragraph 1.3.5]{Drinfeld} for an elementary description of the groupoid of maps $\Spec(\bb{W}(R)/I) \to X$.}
    \end{enumerate}
    If $X=\Spf(R)$, then we write $R^\smallprism$ instead of $X^\smallprism$.
    We write $\mc{O}^\smallprism$ for the structure sheaf of $X^\smallprism$.
\end{defn}

The formal stack $X^\smallprism$ over $\bb{Z}_p$ carries a Frobenius lift denoted $F_X$ (or $F_R$ when $X = \Spf(R)$).\footnote{Technically, this notation is overloaded if $X$ is an $\bb{F}_p\textrm{-scheme}$, as it could also mean the absolute Frobenius there. But, we hope the meaning is clear to the reader from the context.}
For a $p\textrm{-nilpotent}$ ring $R$, the map $F_X \colon X^\smallprism(R) \to X^\smallprism(R)$ associates to the pair $(\alpha\colon I\to \bb{W}(R), s)$ the pair $(F^\ast(\alpha)\colon F^\ast(I)\to \bb{W}(R),F^*(s))$, where $F$ is the Frobenius on $\bb{W}$.
Here $F^\ast(s)$ denotes the composition $\Spec(\bb{W}(R)/F^\ast(I)) \xrightarrow{F} \Spec(\bb{W}(R)/I) \xrightarrow{s} X$.

By a \emph{prismatic crystal} on $X$ we mean a quasi-coherent object of the derived $\infty\textrm{-category}$,
\begin{equation*}
    \mb{D}(X^\smallprism) \defeq \lim_{\Spec(R)\to X^\smallprism}\mb{D}(R),
\end{equation*}
where $\Spec(R) \to X^\smallprism$ travels over morphisms to $X^\smallprism$ from the spectra of $p\textrm{-nilpotent}$ rings.
By a \emph{perfect prismatic crystal} we mean an object of the full $\textrm{(}\infty\textrm{-)subcategory}$,
\begin{equation*}
    \cat{Perf}(X^\smallprism)\defeq \lim_{\Spec(R)\to X^\smallprism}\mb{Perf}(R).
\end{equation*}
We define the category $\cat{Vect}(X^\smallprism)$ of \emph{prismatic crystals in vector bundles} analogously. 

We frequently use another more down-to-earth interpretation of prismatic crystals using objects $(A,I)$ of the absolute prismatic site $X_\smallprism$.
We let $\phi_A$ denote the Frobenius lift on a prism $(A,I)$.

\begin{defn}[Absolute prismatic site, \cite{BhattScholzeCrystals}]\label{defn:absolute-prismatic}
    The \textit{absolute prismatic site} $X_{\smallprism}$ of $X$ is opposite to the category\footnote{As is standard, we will conflate $X_\smallprism$ and its opposite category, often writing morphisms in $X_\smallprism$ as if they were in $X_\smallprism^\mr{op}$. Additionally, we will almost always omit the structure map $s$ from the notation, just writing $(A,I)$.} of bounded prisms $(A,I)$ equipped with a morphism $s \colon \Spf(A/I) \to X$ and endowed with the flat topology.\footnote{More precisely, the topology generated by those collections of morphisms $\{(A,I)\to (A_m,I_m)\}_{m\in M}$ such that either $\{\Spf(A_m)\to \Spf(A)\}$ is a Zariski cover or $M=\{1\}$ and $A\to A_1$ is $(p,I)$-adically faithfully flat.}
    If $X=\Spf(R)$, we write $R_\smallprism$ instead of $X_\smallprism$.
    We denote by $\mc{O}_\smallprism$ the sheaf associating $A$ to $(A,I)$, and let $\phi_\smallprism$ denote the endomorphism restricting to $\phi_A$ on each $(A,I)$.
\end{defn}

For any object $(A,I,s)$ of $X_\smallprism$, note that one can construct a Frobenius-equivariant morphism $\rho_{(A,I,s)} \colon \Spf(A) \to X^\smallprism$ (or just $\rho_{(A,I)}$ if $s$ is clear from context) as follows:
Consider a morphism $g \colon \Spec(R)\to\Spf(A)$, where $R$ is a $p\textrm{-nilpotent}$ ring.
As $A$ is a $\delta\textrm{-ring}$, the universal property of Witt vectors gives us a morphism $A \to \bb{W}(R)$. Then, the Cartier--Witt divisor over $R$ corresponding to the composition $\rho_{(A,I)} \circ g$ is the map $\alpha \colon I \otimes_A \bb{W}(R) \to \bb{W}(R)$.
Moreover, as $\mr{cone}(\alpha)$ annihilates $A$ we get an induced map,
\begin{equation*} 
    \Spec(\mr{cone}(\alpha))\xrightarrow{g} \Spf(A/I)\xrightarrow{s} X.
\end{equation*}
Thus, we have produced an element $\rho_{(A,I)}\circ g$ in $X^\smallprism(R)$ as desired.
This map is functorial in the sense that if $(A,I)\to (B,J)$ is a map in $X_\smallprism$, then the composition $\Spf(B)\to \Spf(A)\xrightarrow{\rho_{(A,I)}}X^\smallprism$ is naturally identified with the map $\rho_{(B,J)}$.

\begin{example}\label{eg:prismatization-qrsp}
    Suppose that $R$ is a quasi-regular semi-perfectoid algebra (e.g.\ perfectoid) in the sense of \cite[Definition 4.20]{BMS-THH}.
    Then, from \cite[Section 7.1]{BhattScholzePrisms} there exists an initial object $(\Prism_R,I_R)$ of $R_\smallprism$ (e.g.\ if $R$ is perfectoid, then $\Prism_R=\Ainf(R)$). 
    Additionally, the natural map $\rho_{(\Prism_R,I_R)} \colon \Spf(\Prism_R) \to R^\smallprism$ is an isomorphism (see \cite[Lemma 6.1]{BhattLurieAbsolute}).
\end{example}

We record here two propositions that we will freely use in the sequel.

\begin{prop} Let $(A,I)$ be a prism. Then, the map of groupoids
\begin{equation*}
\Hom(\Spf(A/I),X)\to \mathrm{Map}(\Spf(A),X^\smallprism),\qquad s\mapsto \rho_{(A,I,s)},
\end{equation*}
where we view the source as a discrete groupoid, is fully faithful.
\end{prop}
\begin{proof} For $i \in \{1,2\}$, let $s_i\colon \Spf(A/I)\to X$ be two morphisms. We must show that the set $\Hom(\rho_{(A,I,s_1)},\rho_{(A,I,s_2)})$ is a singleton if $s_1=s_2$ and empty otherwise.

Without loss of generality we may assume that $I=(d)$ for some non-zero divisor $d$ in $A$. Let $R$ be a $p$-nilpotent ring and $g\colon \Spec(R)\to \Spf(A)$ be a morphism. Then, we may describe the set $\Hom(\rho_{(A,I,s_1)}\circ g,\rho_{(A,I,s_2)}\circ g)$ as the set of commutative diagrams
\begin{equation*}
\begin{tikzcd}[cramped]
	{I\otimes_{A,g}\bb{W}(R)} & {\bb{W}(R)} \\
	{I\otimes_{A,g}\bb{W}(R)}
	\arrow["\alpha", from=1-1, to=1-2]
	\arrow["f"', from=1-1, to=2-1]
	\arrow["\alpha"', from=2-1, to=1-2]
\end{tikzcd}
\end{equation*}
where $f$ is an isomorphism of invertible $\bb{W}(R)$-modules, and such that the induced diagram
\begin{equation}\label{eq:map-between-morphisms}
\begin{tikzcd}[cramped]
	{\Spec(\mr{cone}(\alpha))} &&& {\Spf(A/I)} \\
	{\Spec(\mr{cone}(\alpha))} && {\Spf(A/I)} & X
	\arrow["g", from=1-1, to=1-4]
	\arrow["{s_1}", from=1-4, to=2-4]
	\arrow["f"', from=2-1, to=1-1]
	\arrow["g", from=2-1, to=2-3]
	\arrow["{s_2}", from=2-3, to=2-4]
\end{tikzcd}
\end{equation}
commutes. Then, an element of $\Hom(\rho_{(A,I,s_1)},\rho_{(A,I,s_2)})$ is such an $f$ satisfying this property compatibly for all $g$. As $I\otimes_{A,g}\bb{W}(R)$ is freely generated by $d$, taking $g$ to be $g_n\colon \Spec(R_n)\to \Spf(A)$, and where $R_n=A/(p,I)^n$ produces an element $u_n\in \bb{W}(R_n)^\times$ such that $\wt{g}_n(d)(u_n-1)=0$ where, $\wt{g}_n\colon A\to \bb{W}(R_n)$ is the natural $\delta$-equivariant map. As these maps are compatible in $n$ we deduce the existence of some $u\in\bb{W}(A)^\times$ such that $\iota(d)(u-1)=0$ where $\iota\colon A\to\bb{W}(A)$ is the natural $\delta$-equivariant map. But, as $d$ is a non-zero divisor in $A$ we have that $\iota$ is a non-zero divisor in $\bb{W}(A)$, and thus we deduce that $u=1$, and thus that $u_n=1$ for all $n$ and thus that $f=\mathrm{id}$. Thus, if $\Hom(\rho_{(A,I,s_1)},\rho_{(A,I,s_2)})$ is non-empty, it is a singleton. 

Moreover, as we now know that $f=\mr{id}$ in \eqref{eq:map-between-morphisms}, we deduce that for all $g$ one has that $g\circ s_1=g\circ s_2$. Taking $g=g_n$ as above for all $n$, one easily deduces that $s_1=s_2$. Thus, as one may always take $f=\mr{id}$ if $s_1=s_2$, we deduce that $\Hom(\rho_{(A,I,s_1)},\rho_{(A,I,s_2)})$ is a singleton if $s_1=s_2$ and empty otherwise, as desired.
\end{proof}

\begin{prop}[{\cite[Theorem 6.5]{BhattLuriePrismatization}}]\label{prop:stack-site-crystals-agree}
    The morphisms $\rho_{(A,I)}$ induce equivalences
    \begin{equation*}
        \mb{D}(X^\smallprism)\simeq \mb{D}(X_\smallprism)\simeq \lim_{(A,I)\in X_\smallprism}\mb{D}(A),
    \end{equation*}
    which restrict to equivalences on the full subcategories of perfect complexes and vector bundles.
\end{prop}

There is a natural map $u \colon X^\smallprism \to \cat{Cart}\defeq \big[\wh{\bb{A}}^1/\bb{G}_m\big]$ associating to a pair $(\alpha\colon I\to\bb{W}(R),s)$ the generalised Cartier divisor $I\otimes_{\bb{W}(R),\gamma_0}R\to R$, using the interpretation of the target as in \cite[Section 3.1]{BhattLurieAbsolute}.
We let $\mc{I}^\smallprism \subseteq \mc{O}^\smallprism$ denote the pullback under $u$ of the tautological invertible ideal sheaf on $\cat{Cart}$.
Under the equivalence in Proposition \ref{prop:stack-site-crystals-agree}, the pair $\mc{I}^\smallprism \subseteq \mc{O}^\smallprism$ and endomorphism $F_X$ of $\mc{O}^\smallprism$ correspond to the pair $\mc{I}_\smallprism\subseteq \mc{O}_\smallprism$ and endomorphism $\phi_\smallprism$ of $\mc{O}_\smallprism$, where $\mc{I}_\smallprism(A,I)\defeq I$.

\begin{defn}
    A \emph{perfect prismatic $F\textrm{-crystal}$} on $X^\smallprism$ (resp.\@ $X^\smallprism$) is a perfect complex $\mc{E}$ on $X^\smallprism$ (resp.\@ $X_\smallprism$) and an isomorphism $\varphi_\mc{E} \colon F_X^\ast(\mc{E})[\nicefrac{1}{\mc{I}^\smallprism}] \isomto \mc{E}[\nicefrac{1}{\mc{I}^\smallprism}]$ of $\mc{O}^\smallprism[\nicefrac{1}{\mc{I}^\smallprism}]\textrm{-modules}$ (resp.\@ an isomorphism $\varphi_\mc{E} \colon \phi_\smallprism^\ast(\mc{E})[\nicefrac{1}{\mc{I}_\smallprism}] \to \mc{E}[\nicefrac{1}{\mc{I}_\smallprism}]$ of $\mc{O}_\smallprism[\nicefrac{1}{\mc{I}_\smallprism}]\textrm{-modules}$),\footnote{Observe that $\mathcal{I}^\smallprism \subseteq \mc{O}^\smallprism$ is a Cartier divisor, therefore, for any morphism $\Spec(R) \to X^\smallprism$, the open subscheme $\Spec(R)-V(\mc{I}^\smallprism(R)) \subseteq \Spec(R)$ is affine (see \stacks{07ZT}) and we denote by $\mc{O}^\smallprism[\nicefrac{1}{\mc{I}^\smallprism}](R)$ its global sections. The sheaf $\mc{O}_\smallprism[\nicefrac{1}{\mc{I}_\smallprism}]$ admits a similar definition.} called a \emph{Frobenius structure}.
    Denote the category of perfect prismatic $F$-crystals on $X^\smallprism$ (resp.\@ $X_\smallprism$) by $\cat{Perf}^\varphi(X^\smallprism)$ (resp.\@ $\cat{Perf}^\varphi(X_\smallprism)$).
\end{defn}

From Proposition \ref{prop:stack-site-crystals-agree}, it easily follows that there are natural identifications,
\begin{equation}\label{eq:perfect-prismatic-F-crystal-identifications}
    \cat{Perf}^\varphi(X^\smallprism)\simeq \cat{Perf}^\varphi(X_\smallprism)\simeq \lim_{(A,I)\in X_\smallprism}\cat{Perf}^\varphi(A,I),
\end{equation}
where $\cat{Perf}^\varphi(A,I)$ denotes the category of perfect complexes of $A$-modules $M$ equipped with an isomorphism of $A$-modules $\varphi_M\colon (\phi_A^\ast M)[\nicefrac{1}{I}]\isomto M[\nicefrac{1}{I}]$.
Due to these equivalences, we shall often implicitly identify the categories in \eqref{eq:perfect-prismatic-F-crystal-identifications}.

\subsection{Absolute pushforwards and relative prismatic cohomology}\label{ss:absolute-pushforwards-and-relative-prismatic-cohomology}

Let $f \colon X \to Y$ be a morphism of quasi-syntomic $p\textrm{-adic}$ formal schemes, and note that we have the following:
\begin{enumerate}
    \item A morphism of formal stacks over $\Z_p$,
        \begin{equation*}
            f^\smallprism \colon X^\smallprism\to Y^\smallprism, \qquad (\alpha \colon I\to \bb{W}(R), s) \mapsto (\alpha \colon I \to \bb{W}(R), f \circ s).
        \end{equation*}
        
    \item A cocontinuous morphism of ringed sites, 
        \begin{equation*}
            f_\smallprism \colon X_\smallprism \to Y_\smallprism, \qquad ((A,I),s) \mapsto ((A,I), f \circ s).
        \end{equation*}
\end{enumerate}
The preceding morphisms thus naturally induce morphisms of $\infty\textrm{-categories}$,
\begin{equation*}
    \mr{R}f^\smallprism_\ast \colon \cat{D}(X^\smallprism) \to \cat{D}(Y^\smallprism), \quad\text{and}\quad \mr{R}(f_\smallprism)_\ast \colon \cat{D}(X_\smallprism,\mc{O}_\smallprism) \to \cat{D}(Y_\smallprism, \mc{O}_\smallprism),
\end{equation*}
If $f$ is quasi-syntomic, then these two pushforwards match under the equivalence in Proposition \ref{prop:stack-site-crystals-agree},\footnote{Indeed, let $\mc{E}$ be an object of $\mb{D}(X_\smallprism)$ corresponding to $\mc{E}^\smallprism$ in $\mb{D}(X^\smallprism)$. Then, by definition, the objects $\mr{R}(f_\smallprism)_\ast\mc{E}$ matching $\mr{R}f^\smallprism_\ast\mc{E}^\smallprism$ means that the restriction to $(A,I)$ of the former is precisely the pullback of the latter along $\rho_{(A,I)}$ in a natural way. But, in the notation of \cite{BhattLuriePrismatization}, the former is $\mr{R}\Gamma_\smallprism^\mr{site}(X/A,\mc{E})$ and the latter is $\mr{R}\Gamma(\mr{WCart}_{X/A},\mc{E})$. That these two naturally agree is \cite[Theorem 6.4]{BhattLuriePrismatization}, whose proof does not use that $\mc{E}=\mc{O}_\smallprism$.\label{footnote:pushforwards-agree}} and so we shall often implicitly identify them.

Let $(A,I)$ be an object of $Y_\smallprism$ and $\mc{E}$ an object of $\mb{D}(X^\smallprism)$.
Then, we define
\begin{equation*}
    \mr{R}\Gamma(X/(A,I),\mc{E}) \defeq \mr{L}\rho_{(A,I)}^\ast\mr{R}f^\smallprism_\ast\mc{E}.
\end{equation*}
When there is no chance for confusion, we shall shorten this notation to $\mr{R}\Gamma_A(\mc{E})$ and denote its cohomology groups by $\mr{H}^i_A(\mc{E})$.
These cohomology groups may be computed site-theoretically as the cohomology of (the pullback of) $\mc{E}$ on the relative prismatic site $(X/(A,I))_\smallprism$ as in \cite{BhattScholzePrisms}.

This interpretation in terms of relative prismatic cohomology allows one to produce from a Frobenius structure $\varphi_\mc{E}$ on $\mc{E}$ a morphism $F_Y^\ast(\mr{R}f^\smallprism_\ast\mc{E})[\nicefrac{1}{\mc{I}^\smallprism}] \to (\mr{R}f^\smallprism_\ast\mc{E})[\nicefrac{1}{\mc{I}^\smallprism}]$ of $\mc{O}^\smallprism[\nicefrac{1}{\mc{I}^\smallprism}]\textrm{-modules}$ (e.g.\ see \cite[Construction 7.6]{GuoReinecke}).
We abuse notation and denote this map by $\mr{R}f^\smallprism_\ast\varphi_\mc{E}$.

\begin{thm}[{\cite[Corollary 5.16 and Theorem 8.1]{GuoReinecke}}]\label{thm:GL}
    Let $f \colon X\to Y$ be a smooth proper morphism of smooth formal $\mc{O}_K\textrm{-schemes}$.
    Then, $\mr{R}f^\smallprism_\ast\mc{E}$ is a perfect prismatic crystal on $Y^\smallprism$ for any perfect prismatic crystal $\mc{E}$ on $X$.
    Moreover, if $\mc{E}$ is equipped with a Frobenius structure $\varphi_\mc{E}$, then $\mr{R}f^\smallprism_\ast\varphi_\mc{E}$ is a Frobenius structure on $\mr{R}f^\smallprism_\ast\mc{E}$.
\end{thm}

\subsection{The prismatisation of $\mc{O}_K$}\label{subsec:breuil_kisin_prism}

We now record some points of $\mc{O}_K^\smallprism$ that play a central role in our proof of the twisted crystalline-de Rham comparison (see Theorem \ref{thmi:twisted-dr-crys-comp}).

\subsubsection{The ($n$-twisted) Breuil--Kisin prism}\label{ss:BK-prism}

Let $\mf{S} \defeq W\llbracket u\rrbracket$ and recall that $E = E_{\pi}$ in $\mf{S}$ is the Eisenstein polynomial of a uniformiser $\pi$ of $\mc{O}_K$.
We define the \emph{Breuil--Kisin prism} to be the pair $ (\mf{S}, E)$, where $\mf{S}$ is equipped with the Frobenius lift
\begin{equation*}
    \displaystyle \phi_{\mf{S}} \colon \mf{S} \to \mf{S}, \qquad \textstyle\sum_{k \in \mathbb{N}} a_k u^k \mapsto \sum_{k \in \mathbb{N}} \phi_{W}(a_k)u^{pk}.
\end{equation*}
Observe that the map sending $u$ to $\pi$ induces an isomorphism $\mf{S}/E\isomto \mc{O}_K$, and we denote the \emph{inverse} of this map as $\mr{nat} \colon \mc{O}_K \isomto \mf{S}/E$.
Then, we may view the pair $(\mf{S}, E)$ as an object of $(\mc{O}_K)_\smallprism$ with the map $\Spf(\mf{S}/E) \to \Spf(\mc{O}_K)$ induced by $\mr{nat}$.
Additionally, we shorten the notation $\rho_{(\mf{S}, E)}$ to $\rho_\mf{S}$ or $\rho_\pi$.

\begin{prop}\label{prop:flat-cover}
    The map $\rho_{\mf{S}}\colon \Spf(\mf{S})\to\mc{O}_K^\smallprism$ is a flat surjection.
\end{prop}
\begin{proof}
    Note that $\phi_{\mathfrak{S}} \colon \Spf(\mathfrak{S}) \rightarrow \Spf(\mathfrak{S})$ is a flat surjection.
    So, it is enough to show that $\rho_{\mathfrak{S}} \circ \phi_{\mathfrak{S}}$ is a flat surjection.
    But, from the commutativity of triangle (2) in diagram \eqref{eq:big-diagram-commutes} of Proposition \ref{prop:big-diagram-commutes}, it is then sufficient to show that the composition $\Spf(\Ainf) \to \mc{O}_K^\smallprism$ (see Section \ref{subsubsec:ainf_prism} for this map), is a flat surjection.
    By using the identification in Example \ref{eg:prismatization-qrsp}, we are further reduced to showing that the map $\mc{O}_C^\smallprism\to\mc{O}_K^\smallprism$ is a flat surjection.
    The latter follows from \cite[Lemma 6.3]{BhattLuriePrismatization}, as $\mc{O}_K\to\mc{O}_C$ is quasi-syntomic (e.g.\ see \cite[Lemma 1.15]{IKY1}).
\end{proof}

For any $n\geqslant 0$, set $E^{\tw{n}} \defeq E^{\tw{n}}_\pi \defeq \phi_\mf{S}^n(E_\pi)$.
Then, with $\mf{S}$ still equipped with the Frobenius lift $\phi_\mf{S}$, the pair $(\mf{S}, E^{\tw{n}})$ is a prism called the ($n\textit{-twisted}$) \textit{Breuil--Kisin prism}, which coincides with the usual Breuil--Kisin prism when $n=0$.
The $n\textrm{-fold}$ self-composition of the Frobenius $\phi_\mf{S}^n$ induces a map,
\begin{equation*}
        \ov{\phi}^n_{\mf{S}}\colon \mf{S}/E \to \mf{S}/E^{\tw{n}},
    \end{equation*}
and composing this with $\mr{nat}$ endows $(\mf{S},(E^{\tw{n}}))$ with the structure of an object of $(\mc{O}_K)_\smallprism$.
Additionally, we shorten $\rho_{(\mf{S},E^{\tw{n}})}$ to $\rho^{\tw{n}}_\mf{S}$ or $\rho^{\tw{n}}_\pi$, noting the change in index.

It is clear that $\phi^n_\mf{S}\colon\mf{S}\to\mf{S}$ is a flat surjection (on formal spectrum) which induces a morphism $(\mf{S},(E))\to (\mf{S},E^{\tw{n}})$ in $(\mc{O}_K)_\smallprism$.
Using this and the fact that $\rho_\mf{S}$ is Frobenius-equivariant, we have the following identification in $\mr{Map}(\Spf(\mf{S}),\mc{O}_K^\smallprism)$:
\begin{equation}\label{eq:BK-twist-identification}
    \rho^{\tw{n}}_\mf{S}\simeq F_{\mc{O}_K}^n\circ \rho_{\mf{S}}.
\end{equation}

\subsubsection{The product of Breuil--Kisin prisms}\label{sss:mixed-BK-prism}

Breuil--Kisin prisms depend on the choice of $E = E_\pi$, and thus on the choice of the uniformiser $\pi$.
To show the independence of our later constructions of this choice, in this section, we study the product of Breuil--Kisin prisms for different choices of the uniformiser.

Let us fix another uniformiser $\pi'$ of $\mc{O}_K$.
Let $\mf{S}\wh{\otimes}_{\mathbb{Z}_p}\mf{S}$ be shorthand for the $(p,E_\pi,E_{\pi'})\textrm{-adic}$ completion of $\mf{S}\otimes_{\mathbb{Z}_p}\mf{S}$, and set
\begin{equation*}
    J \defeq \ker(\mf{S}\wh{\otimes}_{\mathbb{Z}_p}\mf{S}) \twoheadrightarrow \mc{O}_K,
\end{equation*}
where the surjective map denotes the tensor product of the natural maps $\mathfrak{S} \rightarrow \mathfrak{S}/(E_{\pi}) \xrightarrow{\textrm{nat}^{-1}} \mathcal{O}_K$ and $\mathfrak{S} \rightarrow \mathfrak{S}/(E_{\pi'}) \xrightarrow{\textrm{nat}^{-1}} \mathcal{O}_K$ followed by the multiplication map.
We define the \emph{mixed Breuil--Kisin prism} to be
\begin{equation*}
    \mf{S}_{\pi,\pi'} = (\mf{S}\wh{\otimes}_{\mathbb{Z}_p}\mf{S})\{\tfrac{J}{E_{\pi}}\}^\wedge_{\delta},
\end{equation*}
where the right-hand side denotes the $(p,E_\pi)\textrm{-adic}$ completion of the $\delta\textrm{-algebra}$ obtained by freely adjoining to $(\mf{S}\wh{\otimes}_{\mathbb{Z}_p}\mf{S})$ the elements $\tfrac{j}{E_\pi}$, for $j$ in $J$.\footnote{The ostensible asymmetry in the definition of $\mf{S}_{\pi,\pi'}$ is remedied by observing that it is naturally equal to $(\mf{S}\wh{\otimes}_{\mathbb{Z}_p}\mf{S})\{\tfrac{J}{E_{\pi'}}\}^\wedge_{\delta}$ as $E_\pi/E_{\pi'}$ is a unit in $\mf{S}_{\pi,\pi'}$ (using \cite[Lemma 2.24]{BhattScholzePrisms} and the fact that $E_{\pi'} = 0 \textrm{ mod } E_{\pi}$).}
Let us set $I_{\pi,\pi'}\subseteq \mf{S}_{\pi,\pi'}$ to be the ideal generated by $E_\pi$, or equivalently, by $E_{\pi'}$. 

Let us observe that there are natural maps of prisms $\mf{S} \xrightarrow{p_1} \mf{S}_{\pi,\pi'} \xleftarrow{p_2} \mf{S}$ coming from projection and, by our setup, the following two compositions coincide:
\begin{equation*}
    \mc{O}_K \xrightarrow{\mr{nat}}\mf{S}_\pi/E_\pi \xrightarrow{p_1} \mf{S}_{\pi,\pi'}/I_{\pi,\pi'} \xleftarrow{p_2} \mf{S}/E_{\pi'} \xleftarrow{\mr{nat}} \mc{O}_K.
\end{equation*}
Thus, $(\mf{S}_{\pi,\pi'}, I_{\pi,\pi'})$ has an unambiguous structure as an element of $(\mc{O}_K)_\smallprism$, and we use $\rho_{\pi,\pi'}$ to denote the corresponding map $\Spf(\mf{S}_{\pi,\pi'}) \to \mc{O}_K^\smallprism$.

\begin{prop} 
    The object $(\mf{S}_{\pi,\pi'}, I_{\pi,\pi'})$ is the product of $(\mf{S}, E_\pi)$ and $(\mf{S}, E_{\pi'})$ in $(\mc{O}_K)_\smallprism$.
    Moreover, the following diagram
    \begin{equation*}
    \begin{tikzcd}
    	{\Spf(\mf{S}_{\pi,\pi'})} & {\Spf(\mf{S})} \\
    	{\Spf(\mf{S})} & {\mc{O}_K^\smallprism,}
    	\arrow["{p_1}", from=1-1, to=1-2]
    	\arrow["{p_2}"', from=1-1, to=2-1]
    	\arrow["{\rho_{\pi,\pi'}}"{description}, from=1-1, to=2-2]
    	\arrow["{\rho_\pi}", from=1-2, to=2-2]
    	\arrow["{\rho_{\pi'}}"', from=2-1, to=2-2]
    \end{tikzcd}
    \end{equation*}
    is $2\textrm{-commutative}$, and its outer square is $2\textrm{-Cartesian}$.
\end{prop}
\begin{proof}
    The first claim follows from the same argument given in \cite[Example 3.4]{DLMS}.
    For the second claim, consider morphisms $f,g \colon \mf{S} \rightrightarrows R$ for a $p\textrm{-nilpotent}$ ring $R$, such that the compositions $\rho_\pi \circ f$ and $\rho_{\pi'} \circ g$ are identified in $\mathrm{Map}(\Spec(R), \mc{O}_K^\smallprism)$.
    This gives us an isomorphism of generalised Cartier--Witt divisors $E_\pi\bb{W}(R) \to \bb{W}(R)$ and $E_{\pi'}\bb{W}(R) \to \bb{W}(R)$, such that the induced maps $\mc{O}_K \to \mf{S}/E_\pi \to \bb{W}(R)/E_\pi$ and $\mc{O}_K \to \mf{S}/E_{\pi'} \to \bb{W}(R)/E_{\pi'}$ are identified.
    From the two maps $\mf{S}\to \bb{W}(R)$ induced by $f$ and $g$, we obtain a unique morphism $\mf{S} \wh{\otimes}_{\bb{Z}_p} \mf{S} \to \bb{W}(R)$.
    Our identification of Cartier--Witt divisors guarantees that this map uniquely upgrades to a morphism $\mf{S}_{\pi,\pi'}\to \bb{W}(R)$.
    The latter map gives rise to a unique morphism $h$ given as the composition $\mf{S}_{\pi,\pi'} \to \bb{W}(R) \to R$, and it is easy to see that $h\circ p_1= f$ and $h\circ p_2 = g$, as desired.
\end{proof}

\subsubsection{The (modified) Breuil prism}

We continue with the notation from Section \ref{ss:BK-prism}.
Define
\begin{align*}
    \tilde{S} = \mf{S}\big\{\tfrac{\phi(u^{\tilde{e}})}{p}\big\}_{\delta}^\wedge &= \mf{S}\big[\{\tfrac{u^{m\tilde{e}}}{m!}\}_{m\geqslant 1}\big]^\wedge_p\\ 
            &=\big\{\textstyle\sum_{m \in \mathbb{N}} a_m\tfrac{u^m}{\lfloor \nicefrac{m}{\tilde{e}}\rfloor!}\in K_0\llbracket u\rrbracket: a_m \text{ is in } W \text{ and they converge } p\text{-adically to }0\big\},
\end{align*}
endowed with the natural Frobenius lift  
\begin{equation*}
    \phi_{\tilde{S}}\colon \tilde{S} \to \tilde{S},\qquad \textstyle\sum_{m \in \mathbb{N}} a_m\tfrac{u^m}{\lfloor \nicefrac{m}{\tilde{e}}\rfloor!}\mapsto \sum_{m \in \mathbb{N}} \phi_{W}(a_m)\tfrac{u^{pm}}{\lfloor \nicefrac{m}{\tilde{e}}\rfloor!}.
\end{equation*}
The pair $(\tilde{S}, p)$ is a crystalline prism which we call the \emph{modified Breuil prism}.

There is another more well-known prism closely related to $(\tilde{S}, p)$: the \emph{Breuil prism} $(S, p)$ (hence the name `modified Breuil prism' for $(\tilde{S}, p)$).
The ring $S$ is defined in exactly the same manner as the ring $\tilde{S}$ but with each instance of $\tilde{e}$ replaced by $e$.
There is a natural injective homomorphism $S \to \tilde{S}$ which is an isomorphism if and only if $e=1$ or $p=2$.

Next, observe that
\begin{equation*}
    E^{\tw{1}}=E^p=u^{pe}=0 \textrm{ mod } p\tilde{S},
\end{equation*}
where the first equality follows because $\phi_{\mf{S}}$ is a Frobenius lift, the second because $E$ is an Eisenstein polynomial of degree $e$, and the final equality follows because 
\begin{equation*}
    u^{pe}=pu^{p(e-\tilde{e})}(p-1)! \tfrac{u^{p\tilde{e}}}{p!}\in p\tilde{S}.
\end{equation*}
Thus, if $\mf{S}/E^{\tw{1}}\to \tilde{S}/p$ is the induced map, then we obtain a map $\tilde{r} \colon \mc{O}_K \to \tilde{S}/p$, which is defined to be the following composition:
\begin{equation}\label{eq:RtoStildeR_struct}
    \mc{O}_K\xrightarrow{\mr{nat}} \mf{S}/E \xrightarrow{\ov{\phi}_{\mf{S}}} \mf{S}/E^{\tw{1}} \to \tilde{S}/p.
\end{equation}
For $n\geqslant 0$, we may view $(\tilde{S}, p)$ as an object of $(\mc{O}_K)_\smallprism$ with the $\mc{O}_K\textrm{-structure}$ map $F_{\tilde{S}/p}^n \circ \tilde{r}$.
Denote the corresponding map $\Spf(\tilde{S}) \to \mc{O}_K^\smallprism$ by $\rho_{\tilde{S}}^{\tw{n}}$ or $\tilde{\rho}_\pi^{\tw{n}}$, dropping the superscript when $n=0$.
Then, by the same logic as for Equation \eqref{eq:BK-twist-identification}, we have the following identification in $\mr{Map}(\Spf(\tilde{S}),\mc{O}_K^\smallprism)$:
\begin{equation}
    \rho^{\tw{n}}_{\tilde{S}} \simeq F_{\mc{O}_K}^n\circ \rho_{\tilde{S}}.
\end{equation}

Similar to above, for $n\geqslant 0$, we may view $(S, p)$ as an object of $(\mc{O}_K)_\smallprism$ with the $\mc{O}_K\textrm{-structure}$ map $F_{S/p}^n \circ r$, where $r \colon \mc{O}_K \to S/p$ is defined to be the following composition:
\begin{equation*}
    \mc{O}_K \xrightarrow{\mr{nat}} \mf{S}/E \xrightarrow{\ov{\phi}_{\mf{S}}} \mf{S}/E^{\tw{1}} \to S/p.
\end{equation*}
Then, it is clear that the natural homomorphism of prisms $(S, p) \rightarrow (\tilde{S}, p)$ is a morphism in $(\mc{O}_K)_\smallprism$.

{
\begin{prop}\label{prop:BO-prism-to-O_K}
    There exists a unique arrow $i$ making the following diagram commute:
    \begin{equation*}
    \begin{tikzcd}[sep=2.25em]
    	{W} & {\mf{S}} & {\tilde{S}} \\
    	& \mc{O}_K
    	\arrow[from=1-1, to=1-2]
    	\arrow[from=1-1, to=2-2]
    	\arrow["{(E=0)}"{description}, from=1-2, to=2-2]
    	\arrow[from=1-2, to=1-3]
    	\arrow["i"{description}, dotted, from=1-3, to=2-2]
    \end{tikzcd}
    \end{equation*}
    where the unlabelled arrows are the natural maps.
    Moreover, $i$ is a PD-thickening.
\end{prop}
\begin{proof}
    The existence and uniqueness of the arrow $i$ follow by noting that $\tfrac{\pi^{m\tilde{e}}}{m!}$ belongs to $\mc{O}_K$ for all $m \geqslant 1$ (e.g.\ see \cite[Lemma 3.9]{BerthelotOgusFIsocrystals}).
    Observe that the map $i$ is $W\text{-linear}$ and sends $u$ to $\pi$, hence it is surjective.
    To show that $i$ is a PD-thickening, we need to show that the kernel of $i$ is a PD-ideal.
    So, let $a$ be any element in the kernel of $i$. As $\tilde{S}$ is $p\textrm{-torsion}$ free (see \cite[Lemma 2.43]{BhattScholzePrisms}) it is enough to show that $\tfrac{a^m}{m!}$ belongs to $\tilde{S}$ for all $m \geqslant 0$.
    By \cite[Lemma 2.35]{BhattScholzePrisms}, we may further reduce ourselves to showing that $a^p$ is in $p\tilde{S}$.
    Let us write $a = \sum_{m \geqslant 0} a_m \tfrac{u^m}{\lfloor m/\tilde{e} \rfloor!}$, with $a_m$ in $W$ for each $m \geqslant 0$.
    
    To show that $a^p$ is in $p\tilde{S}$, it is enough to show that $\sum_{m = 0}^{\tilde{e}-1} a_m^p u^{mp}$ is in $p\tilde{S}$ and $a_m^p \tfrac{u^{mp}}{(\lfloor m/\tilde{e} \rfloor!)^p}$ is in $p\tilde{S}$ for each $m \geqslant \tilde{e}$.
    For the latter note that for each $m \geqslant \tilde{e}$, we may write
    \begin{equation*}
        a_m^p \tfrac{u^{mp}}{(\lfloor m/\tilde{e} \rfloor!)^p} = a_m^p \tfrac{u^{mp}}{\lfloor pm/\tilde{e} \rfloor!} \tfrac{\lfloor pm/\tilde{e} \rfloor!}{(\lfloor m/\tilde{e} \rfloor!)^p},
    \end{equation*}
    which is in $p\tilde{S}$ because an elementary computation shows that for any rational number $r \geqslant 1$:
    \begin{equation*}
    v_p(\lfloor pr \rfloor!) \geqslant v_p((p \lfloor r \rfloor)!) = \lfloor r \rfloor + v_p(\lfloor r \rfloor!) \geqslant p v_p(\lfloor r \rfloor!)+1,
    \end{equation*}
    where $v_p$ denotes the normalised $p\textrm{-adic}$ valuation such that $v_p(p) = 1$, and to compute $v_p(\lfloor r \rfloor!)$ one may use Legendre's formula.
    For the former claim note that the composition $\tilde{S} \xrightarrow{i} \mathcal{O}_K \twoheadrightarrow \mathcal{O}_K/p\mathcal{O}_K$ sends $a$ to $\sum_{m = 0}^{\tilde{e}-1} a_m \pi^m \textrm{ mod } p$ which is $0$ because $i(a) = 0$.
    As $\{1, \pi, \ldots, \pi^{e-1}\}$ forms a $W\textrm{-basis}$ of $\mathcal{O}_K$, its modulo $p$ reduction forms a $k\textrm{-basis}$ of $\mathcal{O}_K/p\mathcal{O}_K$.
    But we have that $0 \leqslant m \leqslant \tilde{e}-1 \leqslant e-1$, and therefore, it follows that $a_m = 0 \textrm{ mod } p$.   
    Hence, we conclude that $a^p$ is in $p\tilde{S}$, and thus $i$ is a PD-thickening of $\mathcal{O}_K$.
\end{proof}
}

\subsubsection{The prism \texorpdfstring{$(W, p)$}{(W, p)}}

The pair $(W, p)$, where $W$ is equipped with the natural Witt vector Frobenius $\phi_W$, is a prism.
Via the natural map $q \colon \mc{O}_K \twoheadrightarrow k$, we may view $(W, p)$ as an element of $(\mc{O}_K)_\smallprism$, and we denote the corresponding map $\Spf(W) \to \mc{O}_K^\smallprism$ by $\rho_W$. 

Observe that we may also view $(W, p)$ naturally as an object of $k_\smallprism$, and thus we also obtain a map $\Spf(W)\to k^\smallprism$ and denote it as $\ov{\rho}_W$.
The following proposition is clear by Example \ref{eg:prismatization-qrsp}.
\begin{prop}\label{prop:rho-W-factorization}
    The map $\rho_W\colon \Spf(W)\to \mc{O}_K^\smallprism$ naturally factorises as 
    \begin{equation*}
        \Spf(W)\xrightarrow{\ov{\rho}_W}k^\smallprism\to\mc{O}_K^\smallprism,
    \end{equation*}
    and the first map is an isomorphism.
\end{prop}

Now, for any $n\geqslant 0$, we may also consider $(W, p)$ as an object of $(\mc{O}_K)_\smallprism$ using the $\mc{O}_K$-structure map given by the composition
\begin{equation*}
    \mc{O}_K\xrightarrow{q} k\xrightarrow{F_k^{n}}k,
\end{equation*}
and denote the corresponding map $\Spf(W)\to \mc{O}_K^\smallprism$ by $\rho^{\tw{n}}_W$. 
Then, similarly to Equation \eqref{eq:BK-twist-identification}, we have the following identifications in $\mr{Map}(\Spf(W),\mc{O}_K^\smallprism)$ and $\mr{Map}(\Spf(W),k^\smallprism)$, respectively:
\begin{equation}\label{eq:rho-W-twist-identifications}
    \rho^{\tw{n}}_{W}\simeq F_{\mc{O}_K}^n\circ \rho_{W},\qquad \ov{\rho}^{\tw{n}}_W\simeq F_k^n\circ\ov{\rho}_W.
\end{equation}

\subsubsection{The prism \texorpdfstring{$(\Ainf, \xi^{\tw{1}})$}{(Ainf, xi1)}}\label{subsubsec:ainf_prism}

Endowed with the usual Frobenius lift $\phi$, the pair $(\Ainf, \xi^{\tw{1}})$ is a prism.
We may then consider the following composition:
\begin{equation*}
    \mc{O}_K\to\mc{O}_C\isomto \Ainf/\xi^{\tw{1}},
\end{equation*}
where the last isomorphism is induced by the map $\theta^{\tw{\text{-}1}} = \theta \circ \phi^{-1}$.
This endows $(\Ainf, \xi^{\tw{1}})$ with the structure of an object of $(\mc{O}_K)_\smallprism$.
We denote the corresponding map $\Spf(\Ainf) \to \mc{O}_K^\smallprism$ by $\rho_\inf$. 

Moreover, for any $n\geqslant 0$, we may consider the pair $(\Ainf,\xi^{\tw{n+1}})$, where as per usual $\xi^{\tw{n+1}} = \phi^n(\xi^{\tw{1}}) = \phi^{n+1}(\xi)$.
This is also a prism with the same Frobenius lift $\phi$.
We may endow it with the structure of an object of $(\mc{O}_K)_\smallprism$ via the $\mc{O}_K\textrm{-structure}$ map
\begin{equation*}
    \mc{O}_K\to\mc{O}_C \isomto \Ainf/\xi^{\tw{1}} \xrightarrow{\phi^n} \Ainf/\xi^{\tw{n+1}}.
\end{equation*}
We denote the corresponding map $\Spf(\Ainf) \to \mc{O}_K^\smallprism$ by $\rho^{\tw{n}}_\inf$.
Then, by the same logic as for Equation \eqref{eq:BK-twist-identification}, one has the following identification in $\mr{Map}(\Spf(\Ainf),\mc{O}_K^\smallprism)$:
\begin{equation}
    \rho^{\tw{n}}_\inf\simeq F_{\mc{O}_K}^n\circ \rho_\inf.
\end{equation}

\subsubsection{The prism \texorpdfstring{$(\Acrys, p)$}{(Acrys, p)}}

Endowed with the usual Frobenius lift $\phi$, the pair $(\Acrys, p)$ is a prism.
We may then consider the composition,
\begin{equation*}
    \mc{O}_K \to \mc{O}_C \isomto \Ainf/\xi^{\tw{1}} \to \Acrys/p,
\end{equation*}
where the last map makes sense as $\xi^{\tw{1}} \in p\Acrys$,\footnote{Indeed, $\xi^p=p!\tfrac{\xi^p}{p!}$ belongs to $p\Acrys$, and so ${\xi}^{(1)}=\xi^p=0\mod p\Acrys$.} and we endow $(\Acrys, p)$ with the structure of an object of $(\mc{O}_K)_\smallprism$.
We denote the corresponding map $\Spf(\Acrys)\to \mc{O}_K^\smallprism$ by $\rho_\crys$. 

Moreover, for any $n\geqslant 0$, we may consider the pair $(\Acrys, p)$ as an object of $(\mc{O}_K)_\smallprism$ with the $\mc{O}_K\textrm{-structure}$ map
\begin{equation*}
    \mc{O}_K \to \mc{O}_C\isomto \Ainf/\xi^{\tw{1}} \to \Acrys/p \xrightarrow{\phi^n} \Acrys/p.
\end{equation*}
We denote the corresponding map $\Spf(\Acrys)\to \mc{O}_K^\smallprism$ by $\rho^{\tw{n}}_\crys$.
Then, by the same logic as for \eqref{eq:BK-twist-identification}, one has the following identification in $\mr{Map}(\Spf(\Acrys),\mc{O}_K^\smallprism)$:
\begin{equation}
    \rho^{\tw{n}}_\crys\simeq F_{\mc{O}_K}^n\circ \rho_\crys.
\end{equation}

\subsubsection{Relationship between various prisms}

In this section, our goal is to show the precise relationship between various prisms discussed above.
We do this in the form of the following proposition, whose precise meaning and proof will be explained immediately afterwards.

\begin{prop}\label{prop:big-diagram-commutes}
    Let $n \geqslant 0$ and consider the following diagram:
    \begin{equation}\label{eq:big-diagram-commutes}
    \begin{tikzcd}
    	{\Spf(\Ainf)} &&& {\Spf(\Acrys)} & \\
    	&&& {\Spf(W)} \\
    	& {\mc{O}_K^\smallprism} \\
    	{\Spf(\mf{S})} &&& {\Spf(\tilde{S})} & {\Spf(S)}.
    	\arrow["{\rho^{\tw{n}}_\inf}"{description}, from=1-1, to=3-2]
    	\arrow[""{name=0, anchor=center, inner sep=0}, "{\alpha_\inf}"{description}, from=1-1, to=4-1]
    	\arrow[""{name=1, anchor=center, inner sep=0}, from=1-4, to=1-1]
    	\arrow[from=1-4, to=2-4]
    	\arrow[""{name=2, anchor=center, inner sep=0}, "{\rho_\crys^{\tw{n}}}"{description}, curve={height=12pt}, from=1-4, to=3-2]
    	\arrow[""{name=3, anchor=center, inner sep=0}, "{\alpha_\crys}"{description}, curve={height=-30pt}, from=1-4, to=4-5]
    	\arrow[""{name=4, anchor=center, inner sep=0}, "{\rho_W^{\tw{n+1}}}"{description}, shift right=2, curve={height=-12pt}, from=2-4, to=3-2]
    	\arrow["{\rho^{\tw{n+1}}_\mf{S}}"{description}, from=4-1, to=3-2]
    	\arrow[""{name=5, anchor=center, inner sep=0}, shift left=2, from=4-4, to=2-4]
    	\arrow["{\rho^{\tw{n}}_{\tilde{S}}}"{description}, from=4-4, to=3-2]
    	\arrow[""{name=6, anchor=center, inner sep=0}, from=4-4, to=4-1]
    	\arrow[from=4-4, to=4-5]
    	\arrow[""{name=7, anchor=center, inner sep=0}, from=4-5, to=2-4]
    	\arrow["{(2)}"{description}, shift right, draw=none, from=0, to=3-2]
    	\arrow["{(7)}"{description}, shift left, draw=none, from=3, to=7]
    	\arrow["{(3)}"{description}, shift right, draw=none, from=1, to=3-2]
    	\arrow["{(4)}"{description}, shift right=4, draw=none, from=4, to=2]
    	\arrow["{(6)}"{description}, draw=none, from=5, to=7]
    	\arrow["{(1)}"{description}, draw=none, from=6, to=3-2]
    	\arrow["{(5)}"{description, pos=0.3}, shift left, draw=none, from=5, to=3-2]
    \end{tikzcd}
    \end{equation}
    Triangles $(1)$, $(2)$, $(3)$, $(6)$ and $(7)$ commute for all $n \geqslant 0$, triangle $(5)$ commutes if and only if $n \geqslant a$, and triangle $(4)$ commutes if and only if $n \geqslant \lceil \log_p(e) \rceil$.
\end{prop}

In Proposition \ref{prop:big-diagram-commutes} the unlabelled morphisms are the obvious ones, and the remaining morphisms are explained as follows:
\begin{itemize}
    \item The map $\alpha_{\mr{inf}} \colon \mf{S} \to \Ainf$ is a homomorphism of $W\textrm{-algebras}$, where we set $\alpha_\inf(u)=[\pi^\flat]$,
    
    \item the map $\alpha_{\crys} \colon S \to \Acrys$ is the unique extension of $\alpha_\inf$ as the divided powers $\tfrac{[\pi^\flat]^{me}}{m!}$ belong to $\Acrys$ for all $m\geqslant 0$.
\end{itemize}

\begin{proof}[Proof of Proposition \ref{prop:big-diagram-commutes}]
    The proof of this proposition is largely straightforward, so we only comment on the more non-obvious parts of the claim below.
    In particular, by definitions it is clear that triangles $(1)$, $(3)$, $(6)$ and $(7)$ are commutative.

    Let us first show that the triangle $(2)$ in diagram \eqref{eq:big-diagram-commutes} is commutative.
    By definitions it is enough to establish that $\alpha_\mr{inf}$ is a map in $(\mc{O}_K)_\smallprism$.
    To get this claim, we must first show that $\alpha_\mr{inf}$ maps the ideal $(\phi_\mf{S}(E))$ into the ideal $(\xi^{\tw{1}})$.
    This amounts to showing that $\theta(\alpha_\mr{inf}(E)) = 0$, but this precisely translates to having $E(\pi)=0$.
    To show that the $\mc{O}_K\textrm{-structure}$ maps are compatible, it suffices to observe that 
    \begin{equation*}
        \mc{O}_K\isomto \mf{S}/E \xrightarrow{\ov{\phi}_\mf{S}} \mf{S}/E^{(1)}\xrightarrow{\alpha_\mr{inf}} \Ainf/\xi^{\tw{1}},
    \end{equation*}
    sends $w$ in $W$ to $\phi(w)$ in $\Ainf/\xi^{\tw{1}}$ and sends $\pi$ to $[\pi^\flat]^p$.
    But, this is precisely the same as the map $\mc{O}_K\to \Ainf/\xi^{\tw{1}}$ described in Section \ref{subsubsec:ainf_prism}. This establishes the claim for $n=0$, which immediately implies the claim for arbitrary $n$.
    
    Let us next show that the triangle $(5)$ in diagram \eqref{eq:big-diagram-commutes} commutes if and only if $n\geqslant a$.
    This amounts to the claim that the following morphism,
    \begin{equation*}
        \begin{tikzcd}[sep=2.25em]
            {(W, p,F_{k}^{n+1}\circ q),} & {(\tilde{S}, p, F_{\tilde{S}/p}^n\circ r)}
            \arrow[from=1-1, to=1-2]
        \end{tikzcd}
    \end{equation*}
    where the third entry of each term indicates the $\mc{O}_K\textrm{-structure}$ map, is a morphism in $(\mc{O}_K)_\smallprism$ if and only if $n\geqslant a$ (cf.\@ \cite[Lemma 1.14]{IKY1}).
    It is easy to see that the map above is a morphism of prisms, and it remains to show that the $\mc{O}_K\textrm{-structures}$ agree.
    Diagrammatically, it is equivalent to the claim that the following square:
    \begin{equation*}
        \begin{tikzcd}[sep=large]
        	{\mf{S}/E^{(1)}} & {\tilde{S}/p} & {\tilde{S}/p} & {k} \\
        	{\mf{S}/E} & \mc{O}_K & {\mc{O}_K/\pi} & {k},
        	\arrow[ from=1-1, to=1-2]
        	\arrow["{F^n_{\tilde{S}/p}}", from=1-2, to=1-3]
        	\arrow[from=1-4, to=1-3]
        	\arrow["{\ov{\phi}_{\mf{S}}}", from=2-1, to=1-1]
        	\arrow["\sim"', from=2-2, to=2-1]
        	\arrow[from=2-2, to=2-3]
        	\arrow["\sim", from=2-3, to=2-4]
        	\arrow["{F_k^{n+1}}"', from=2-4, to=1-4]
        \end{tikzcd}
    \end{equation*}
    where all the unlabelled arrows are the natural ones, commutes if and only if $n \geqslant a$.
    Writing $\mc{O}_K = W[\pi]$, it suffices to verify that starting from $\mc{O}_K$ and travelling around the diagram in both directions produces the same result for any $x$ in $W$ and $\pi$.
    It is easy to verify that the composition of the arrows in both directions sends $x$ to $x^{p^{n+1}}$ in $\tilde{S}/p$.
    Moreover, the composition of the arrows along the right-hand side of the diagram sends $\pi$ to $0$, but the composition of the arrows along the left hand side sends $\pi$ to $u^{p^{n+1}}$ in $\tilde{S}/p$.
    Thus, the diagram commutes if and only if $p$ divides $u^{p^{n+1}}$ in $\tilde{S}$.
    This clearly happens if and only if $p^{n+1} \geqslant p\tilde{e}$, which is equivalent to having that $p^n \geqslant \tfrac{e}{p-1}$, or equivalently, $n \geqslant \lceil\log_p(\tfrac{e}{p-1})\rceil = a$, as claimed.

    An argument similar to above shows that the triangle $(4)$ in diagram \eqref{eq:big-diagram-commutes} commutes if and only if $p^{n+1} \geqslant pe$, which is equivalent to having that $n \geqslant \lceil\log_p(e)\rceil$.
    This allows us to conclude.
\end{proof}

\begin{lem}
    The morphism $s \colon \tilde{S} \to W$ defined by the identity on $W$ and sending $u$ to $0$, determines a section to the horizontal arrow in the following diagram from \eqref{eq:big-diagram-commutes}:
    \begin{equation*}
        \begin{tikzcd}[sep=2.25em]
            \Spf(\tilde{S}) && \Spf(W) \\
            & {\mc{O}_K^\smallprism}.
        	\arrow[from=1-1, to=1-3]
        	\arrow["\rho^{\tw{n}}_{\tilde{S}}"{description}, from=1-1, to=2-2]
        	\arrow["s"{description, pos=0.6}, shift left, curve={height=18pt}, from=1-3, to=1-1]
        	\arrow["\rho_W^{\tw{n+1}}"{description}, from=1-3, to=2-2]
        \end{tikzcd}
    \end{equation*}
    The morphism $s$ is a morphism in $(\mc{O}_K)_\smallprism$, i.e., the diagram excluding the horizontal arrow is commutative, but the full diagram commutes if and only if $n\geqslant a$.
\end{lem}
\begin{proof}
    It is clear that $s$ determines a section to the natural map $W \rightarrow \tilde{S}$, and we need to show that $s$ is a morphism in $(\mc{O}_K)_\smallprism$.
    To this end, observe that the map $\mf{S} \to W$ sending $u$ to $0$ does send $(E^{\tw{n}})$ to $(p)$ for any $n\geqslant 0$.
    Indeed, note that $\phi^n(E(u))$ is sent to $\phi^n(E(0))$ in $W$.
    But, as $E$ is an Eisenstein polynomial we must have $E(0)=pv$, for some unit $v$ in $W$.
    So, $s$ is a morphism of prisms.
    That it is a morphism in $(\mc{O}_K)_\smallprism$ follows as the composition $\mathcal{O}_K \xrightarrow{F^n_{\tilde{S}/p} \circ \tilde{r}} \tilde{S}/p\tilde{S} \xrightarrow{s} W/pW = k$ coincides with the composition $\mathcal{O}_K \xrightarrow{F_k^{n+1} \circ q} k$ for all $n\geqslant 0$. The final claim follows from the proof of Proposition \ref{prop:big-diagram-commutes}.
\end{proof}

\section{The twisted crystalline--de Rham comparison}\label{s:crys-de-rham-comp}

In this section we provide a general comparison between the so-called \emph{twisted crystalline realisation} and \emph{twisted de Rham realisation} of a perfect complex of prismatic $F\textrm{-crystals}$ over $\mc{O}_K$.

We continue to use notation from {\globalnotationref} and Notation \ref{nota:Witt-vector-stuff}.

\subsection{A stack-theoretic approach}

We begin by setting up and stating the twisted crystalline-de Rham comparison, and giving a conceptual proof in the spirit of \cite{BerthelotOgus}.

\subsubsection{The twisted crystalline realisation functor}\label{ss:twisted-crystalline-realisation}

We first study several ways to explicitly understand the stack $k^\smallprism$.
The interested reader may consult \cite[Section 1.1.3]{IKY3} for a more general discussion in an unramified and non-twisted setting.

\begin{defn}\label{defn:ntwisted_crys_point}
    The \emph{reduced $n\textrm{-twisted}$ crystalline point} of $\mc{O}_K$ is the morphism of formal stacks over $\bb{Z}_p$,
    \begin{equation*}
        \ov{\rho}^{\tw{n}}_\crys\colon \Spf(W) \longrightarrow k^\smallprism,
    \end{equation*}
    for $n\geqslant 0$, which for a $p$-nilpotent ring $R$ associates to a map $\Spec(R)\to \Spf(W)$ the Cartier--Witt divisor $\bb{W}(R)\xrightarrow{p}\bb{W}(R)$ with structure map,
    \begin{equation}\label{eq:crystalline-point-structure-maps}
        k \isomto W/p \xrightarrow{\phi_W} W/p \longrightarrow \bb{W}(R)/p \xrightarrow{F^n} \bb{W}(R)/p,
    \end{equation}
    where $W/p\to \bb{W}(R)/p$ is the map induced by the unique $\delta\textrm{-ring}$ map $W\to\bb{W}(R)$ lifting $W \to R$.
    The \emph{$n\textrm{-twisted}$ crystalline point of $\mc{O}_K$} is the map $\rho^{\tw{n}}_\crys\colon \Spf(W)\to\mc{O}_K^\smallprism$ obtained as the following composition:
    \begin{equation*}
        \Spf(W)\xrightarrow{\ov{\rho}^{\tw{n}}_\crys}k^\smallprism\to\mc{O}_K^\smallprism,
    \end{equation*}
    where the latter map is the obvious one.
\end{defn}

Let us shorten $\ov{\rho}^{\tw{0}}_\mr{crys}$ to $\ov{\rho}_\mr{crys}$ and $\rho^{\tw{0}}_\mr{crys}$ to $\rho_\mr{crys}$.
Then, by the definition of the Frobenius maps $F_k\colon k^\smallprism\to k^\smallprism$ and $F_{\mc{O}_K}\colon \mc{O}_K^\smallprism\to\mc{O}_K^\smallprism$, we have that for any $n\geqslant 0$.

\begin{equation}\label{eq:n-twist-crys-identities}
    \rho^{\tw{n}}_\crys = F_{\mc{O}_K}^n \circ \rho_\crys,\qquad \ov{\rho}^{\tw{n}}_\crys = F_k^n \circ \rho_\crys
\end{equation}

\begin{rem}\label{rem:appearance-of-phiW}
    The appearance of the $\phi_W$ in \eqref{eq:crystalline-point-structure-maps} may seem artificial.
    To help demystify it, observe that in general for a quasi-syntomic $k\textrm{-scheme}$ $Z$ there is an isomorphism between the formal stack over $W$ given by $(Z/W)^\mr{crys}$ (e.g.\ as defined in \cite[Remark 2.5.12]{BhattNotes}) and the pullback stack $\phi_W^\ast(Z^\smallprism)$ (e.g.\ see \cite[Corollary 2.6.8 \& Construction 3.1.1]{BhattNotes} or \cite[Lemma 1.6]{IKY3}).
    From this perspective, we may view $\ov{\rho}_\crys$ as the composition
    \begin{equation*}
        \Spf(W)\isomto (k/W)^\crys\isomto \phi_W^\ast(k^\smallprism)\xrightarrow{F_k}k^\smallprism.
    \end{equation*}
    This observation is important as the first isomorphism $\Spf(W)\to (k/W)^\crys$ is the one well-suited to crystalline theory (e.g.\ it corresponds to the equivalence between ($F$-)crystals on $(k/W)_\crys$ and ($\phi_W$-)modules over $W$), and is also the reason why a Frobenius twist appears in the crystalline comparison theorem (e.g.\ as in \cite[Theorem 1.8]{BhattScholzePrisms}).
\end{rem}

We define our twisted crystalline realisation functor as follows.

\begin{defn}
    The \emph{$n$-twisted crystalline realisation} is the functor,
    \begin{equation*}
        \bb{D}^{\tw{n}}_\crys\colon \cat{Perf}((\mc{O}_K)_\smallprism) \longrightarrow \cat{Perf}(W),
    \end{equation*}
    obtained as the pullback $\rho^{\tw{n}}_\crys$.
\end{defn}

\begin{rem}
    For any object $\mc{V}$ of $\cat{Perf}((\mc{O}_K)_\smallprism)$, by Remark \ref{rem:appearance-of-phiW} one may also interpret $\bb{D}^{\tw{n}}_\mr{crys}(\mc{V})$ as $(\phi^{\ast n}\mc{V}^\crys)(W\twoheadrightarrow k)$, where $\mc{V}^\crys$ is the perfect complex of crystals on $(k/W)^\crys$ associated to $\mc{V}_k$ via the equivalence in \cite[Theorem 6.4]{GuoReinecke}, $\phi$ is the Frobenius morphism of topoi $(k/W)_\crys\to (k/W)_\crys$, and $W\twoheadrightarrow k\simeq W/p$ is the natural PD-thickening.
\end{rem}

\begin{prop}\label{prop:crys-rho-W-factorization}
    For any $n\geqslant 0$, there is a natural identification of $\ov{\rho}^{\tw{n}}_\crys$ and $\rho^{\tw{n}}_\crys$ with $\ov{\rho}^{\tw{n+1}}_W$ and $\rho^{\tw{n+1}}_W$, respectively.
\end{prop}
\begin{proof}
    By Equations \eqref{eq:rho-W-twist-identifications} and \eqref{eq:n-twist-crys-identities}, our definition of $\rho_\crys$, and Proposition \ref{prop:rho-W-factorization} we are reduced to showing that $\rho_\crys = \rho^{\tw{1}}_W$.
    But, by definition, to a $p\textrm{-nilpotent}$ ring $R$ and a map $\Spec(R) \to \Spf(W)$, the map $\rho^{\tw{1}}_W$ associates the Cartier--Witt divisor $d \colon \bb{W}(R) \otimes_W (p) \to \bb{W}(R)$, where $W\to \bb{W}(R)$ is the map of $\delta\textrm{-rings}$ obtained from $W \to R$, and with the structure map
    \begin{equation*}
        k\xrightarrow{F_k}k \longrightarrow \mr{cone}(d) \longrightarrow \bb{W}(R)/p,
    \end{equation*}
    which clearly agrees with the object of $\mc{O}_K^\smallprism(R)$ associated to $\Spec(R)\to \Spf(W)$ by $\rho_\crys$.
\end{proof}

Combining this observation with Proposition \ref{prop:big-diagram-commutes} and the equivalence in \cite[Proposition 3.3.5]{BhattLurieAbsolute}, we obtain the following identifications.

\begin{cor}\label{cor:twisted-crys-identification} 
    For any $n\geqslant 0$, uniformiser $\pi$ of $\mc{O}_K$, and any object $\mc{V}$ of $\cat{Perf}((\mc{O}_K)_\smallprism)$, there is a natural identification in $\cat{Perf}(W)$:
    \begin{equation}\label{eq:twisted-crys-identification}
    \begin{tikzcd}
        {\bb{D}^{\tw{n}}_\crys(\mc{V})\simeq \phi_W^{\ast\tw{n+1}}\mc{V}(W,p)} & {(\phi_\mf{S}^{\ast\tw{n+1}}(\mc{V}(\mf{S}, E_\pi)))/u}.
        \arrow["\sim", from=1-1, to=1-2]
        \arrow["{\psi_{\crys,\pi}^{\tw{n}}}"', from=1-1, to=1-2]
    \end{tikzcd}
    \end{equation}
\end{cor}

\begin{rem}
    The first isomorphism in \eqref{eq:twisted-crys-identification} is truly canonical, but the second (while natural in $\mc{V}$) implicitly depends on the choice of a uniformiser $\pi$ to define $E_\pi$, and therefore the prism $(\mf{S}, E_\pi)$ in $(\mc{O}_K)_{\prism}$.
    This explains our reasoning for labelling the second map but not the first.
\end{rem}

\subsubsection{The twisted de Rham realisation functor}

We now discuss a de Rham analogue of the material from Section \ref{ss:twisted-crystalline-realisation}.
The interested reader may consult \cite[Section 1.2]{IKY3} for a more general discussion in the non-twisted setting.

\begin{defn}\label{defn:ntwisted_deRham_point}
    The \emph{$n\textrm{-twisted}$ de Rham point of $\mc{O}_K$} is the morphism of formal stacks over $\bb{Z}_p$,
    \begin{equation*}
        \rho_\dR^{\tw{n}}\colon \Spf(\mc{O}_K) \longrightarrow \mc{O}_K^\smallprism,
    \end{equation*}
    for $n\geqslant 0$, which for a $p\textrm{-nilpotent}$ ring $R$ associates to a map $\Spec(R) \to \Spf(\mc{O}_K)$ the Cartier--Witt divisor $\bb{W}(R) \xrightarrow{p} \bb{W}(R)$ with the structure map
    \begin{equation*}
        \mc{O}_K \longrightarrow R=\bb{W}(R)/V\bb{W}(R) \xrightarrow{F} \bb{W}(R)/p \xrightarrow{F^n} \bb{W}(R)/p.
    \end{equation*}
\end{defn}

Let us shorten $\rho^{\tw{0}}_\dR$ to $\rho_\dR$. Then, by the definition of the Frobenius map $F_{\mc{O}_K}\colon \mc{O}_K^\smallprism \to \mc{O}_K^\smallprism$ we have that $\rho^{\tw{n}}_\dR = F_{\mc{O}_K}^n\circ\rho_\dR$, for any $n\geqslant 0$.

\begin{defn}
    The \emph{$n\textrm{-twisted}$ de Rham realisation} is the functor,
    \begin{equation*}
        \bb{D}_\dR^{\tw{n}}\colon \cat{Perf}((\mc{O}_K)_\smallprism) \longrightarrow \cat{Perf}(\mc{O}_K),
    \end{equation*}
    given by pullback along $\rho_\dR^{\tw{n}}\colon \Spf(\mc{O}_K) \rightarrow \mc{O}_K^\smallprism$.
\end{defn}

We again wish to more concretely understand the twisted de Rham realisation using the Breuil--Kisin prism, as in Corollary \ref{cor:twisted-crys-identification}. 
\begin{prop}[{\cite[Proposition 1.15]{IKY3}}]\label{prop:dR-twisted-mfS-identification}
    There is a canonical identification in $\mr{Map}(\Spf(\mc{O}_K),\mc{O}_K^\smallprism)$:
    \begin{equation*}
        \rho_\dR^{\tw{n}} \simeq \rho^{\tw{n+1}}_{\mf{S}} \circ \textup{nat.},
    \end{equation*}
    where $\textup{nat.}$ is the natural map induced by the composition $\mathfrak{S} \twoheadrightarrow \mathfrak{S}/E \simeq \mc{O}_K$.
\end{prop}

\begin{cor}\label{cor:twisted-dR-identification}
    For any $n\geqslant 0$, uniformiser $\pi$ of $\mc{O}_K$, and any object $\mc{V}$ of $\cat{Perf}((\mc{O}_K)_\smallprism)$, the following is a natural identification in $\cat{Perf}(\mc{O}_K)$:
    \begin{equation}\label{eq:twisted-dR-identification}
        \psi_{\dR,\pi}^{\tw{n}}\colon \bb{D}^{\tw{n}}_\dR(\mc{V}) \isomto \phi_\mf{S}^{\ast\tw{n+1}}\mc{V}(\mf{S}, E_\pi)/E_\pi.
    \end{equation}
\end{cor}

\subsubsection{The twisted crystalline-de Rham comparison}

We are now ready to state and prove the twisted crystalline-de Rham comparison.

\begin{thm}[Twisted crystalline-de Rham comparison]\label{thm:twisted-crys-dR-comparison}
    Suppose that $n \geqslant a=\lceil \log_p(\tfrac{e}{p-1})\rceil$.
    \begin{enumerate}[leftmargin=.3in]
        \item[\textup{(1)}] There is an identification between $\rho_{\dR}^{\tw{n}}$ and the composition
            \begin{equation*}
                \Spf(\mc{O}_K) \longrightarrow \Spf(W)\xrightarrow{\rho^{\tw{n}}_\crys}\mc{O}_K^\smallprism,
            \end{equation*}
            as objects of $\mr{Map}(\Spf(\mc{O}_K),\mc{O}_K^\smallprism)$. 
            
        \item[\textup{(2)}] For any object $\mc{V}$ of $\cat{Perf}((\mc{O}_K)_\smallprism)$, there is a canonical isomorphism in $\cat{Perf}(\mc{O}_K)$:
            \begin{equation*}
                \iota^{\tw{n}}\colon \bb{D}_\mr{crys}^{\tw{n}}(\mc{V})\otimes_W\mc{O}_K\isomto \bb{D}^{\tw{n}}_\dR(\mc{V}).
            \end{equation*}
    \end{enumerate}
\end{thm}
\begin{proof}Assertion (2) is an immediate consequence of assertion (1). To prove assertion (1), let $R$ be a $p\textrm{-nilpotent}$ ring and $\alpha$ a morphism $\Spec(R)\to \Spf(\mc{O}_K)$.
    The underlying Cartier--Witt divisors of both $\rho^{\tw{n}}_{\dR}(\alpha)$ and $\rho^{\tw{n}}_{\crys}(\alpha)$ are $\bb W(R)\xrightarrow{p} \bb W(R)$, and their $\mc{O}_K\textrm{-structures}$ are given respectively as follows (see Definition \ref{defn:ntwisted_crys_point} and Definition \ref{defn:ntwisted_deRham_point}):
    \begin{eqnarray}
        \mc{O}_K\xrightarrow{\alpha} R=\bb{W}(R)/V\bb{W}(R)\xrightarrow{F}\bb{W}(R)/p\xrightarrow{F^n}\bb{W}(R)/p, \label{eq:first-structure-map-twisted-crys-dR-comparison}
        \\
        \mc{O}_K\to k\isomto W/p\xrightarrow{\phi_W}W/p \xrightarrow{\bb{W}(\alpha)} \bb{W}(R)/p\xrightarrow{F^n}\bb{W}(R)/p.\label{eq:second-structure-map-twisted-crys-dR-comparison}   
    \end{eqnarray}
    To show that the composition of these respective maps agree, first observe that the following respective compositions of the maps,
    \begin{eqnarray}
        W \longrightarrow \mc{O}_K\xrightarrow{\alpha} R=\bb{W}(R)/V\bb{W}(R)\xrightarrow{F}\bb{W}(R)/p,\label{eq:third-structure-map-twisted-crys-dR-comparison}\\
        W \longrightarrow k\isomto W/p\xrightarrow{\phi_W}W/p \xrightarrow{\bb{W}(\alpha)} \bb{W}(R)/p, \label{eq:fourth-structure-map-twisted-crys-dR-comparison} 
    \end{eqnarray}
    agree by \cite[Theorem 1.19]{IKY3}.
    
    To obtain that the compositions \eqref{eq:first-structure-map-twisted-crys-dR-comparison} and \eqref{eq:second-structure-map-twisted-crys-dR-comparison} agree, it remains to show that the image of a fixed uniformiser $\pi$ of $\mathcal{O}_K$, under \eqref{eq:first-structure-map-twisted-crys-dR-comparison}, is sent to $0$. Recall that we have the modified Breuil prism $(\tilde{S}, p)$ in $(\mathcal{O}_K)_{\smallprism}$, depending on the choice of $\pi$.
    Now, consider the following diagram:
    \begin{equation*}
        \begin{tikzcd}[cramped]
        	& {\tilde{S}} & {\tilde{S}} & {\tilde{S}} \\
        	& {\mathbb{W}(R)} & {\mathbb{W}(R)} & {\mathbb{W}(R)} \\
        	{\mathcal{O}_K} & {R=\mathbb{W}(R)/V\mathbb{W}(R)} & {\mathbb{W}(R)/p} & {\mathbb{W}(R)/p}
        	\arrow["{\phi_{\tilde{S}}}", from=1-2, to=1-3]
        	\arrow["\beta", from=1-2, to=2-2]
        	\arrow["i"', from=1-2, to=3-1]
        	\arrow["{\phi_{\tilde{S}}^n}", from=1-3, to=1-4]
        	\arrow["\beta"', from=1-3, to=2-3]
        	\arrow["\beta", from=1-4, to=2-4]
        	\arrow["F", from=2-2, to=2-3]
        	\arrow["V", from=2-2, to=3-2]
        	\arrow["{F^n}", from=2-3, to=2-4]
        	\arrow[from=2-3, to=3-3]
        	\arrow[from=2-4, to=3-4]
        	\arrow["\alpha", from=3-1, to=3-2]
        	\arrow["F", from=3-2, to=3-3]
        	\arrow["{F^n}", from=3-3, to=3-4]
        \end{tikzcd}
    \end{equation*}
    where $\beta$ is a Frobenius-equivariant map making the left triangle commute, and its existence follows from the universal property of Witt vectors and the fact that $\tilde{S}$ is a $\delta\textrm{-ring}$.
   The commutativity of the diagram follows from the definitions.
    From Proposition \ref{prop:BO-prism-to-O_K} note that $i$ is a PD-thickening and $u$ in $\tilde{S}$ is a lift of $\pi$.
    So, to have that the image of $\pi$ under the composition of the bottom arrows is $0$, it is enough to have that $\phi_{\tilde{S}}^{n+1}(u)$ is in $p\tilde{S}$, and hence $(\beta \circ \phi_{\tilde{S}}^{n+1})(u)$ is in $p\mathbb{W}(R)$. But, this is easily verified for $n \geqslant a$ (see the proof of the commutativity of triangle (5) in Proposition \ref{prop:big-diagram-commutes}), and so the conclusion follows.
\end{proof}

%

\subsection{A non-stacky approach}

We now explain a more down-to-earth approach to the realisation functors from the last subsection, and to Theorem \ref{thm:twisted-crys-dR-comparison}, using Breuil--Kisin theory. We hope that this will make it easier to apply this result in practice.

\subsubsection{The twisted crystalline realisation}

We begin by giving a Breuil--Kisin-theoretic definition of the twisted crystalline realisation.
To distinguish it from the (a priori different) stack-theoretic twisted crystalline realisation, we temporarily use a prime in the notation.

\begin{defn} 
    Let $\pi$ be a uniformiser of $\mc{O}_K$ and $n \geqslant 0$ an integer.
    The \emph{$\pi\textrm{-indexed}$ $n\textrm{-twisted}$ crystalline realisation functor} is defined as follows:
    \begin{equation*}
        \mathbb{D}^{\prime,\tw{n}}_{\mr{crys},\pi}\colon \cat{Perf}((\mc{O}_K)_\smallprism) \longrightarrow \cat{Perf}(W),\qquad \mc{V} \longmapsto \phi_\mf{S}^{\ast\tw{n+1}}\mc{V}(\mf{S}, E_\pi)/u.
    \end{equation*}
\end{defn}

To remove the $\pi\textrm{-dependency}$ of the $n\textrm{-twisted}$ crystalline realisation functors, we consider the following construction.

\begin{construction}\label{const:crys-transition-isom}
   Let $\mc{V}$ be an object of $\cat{Perf}((\mc{O}_K)_\smallprism)$ and $n\geqslant 0$ an integer.
   Then, the following commutative diagram (see Proposition \ref{prop:big-diagram-commutes}):
    \begin{equation*}
        \begin{tikzcd}[column sep=50pt, row sep=35pt]
    	{\Spf(\mf{S})} & {\Spf(W)} & {\Spf(\mf{S})} \\
    	& {\mc{O}_K^\smallprism,}
    	\arrow["{\rho^{\tw{n+1}}_\pi}"{description}, from=1-1, to=2-2]
        \arrow["{(u=0)}"{description}, from=1-2, to=1-1]
        \arrow["{(u=0)}"{description}, from=1-2, to=1-3]
    	\arrow["{\rho^{\tw{n+1}}_{W}}"{description}, from=1-2, to=2-2]
    	\arrow["{\rho^{\tw{n+1}}_{\pi'}}"{description}, from=1-3, to=2-2]
        \end{tikzcd}
    \end{equation*}
    together with Equations \eqref{eq:BK-twist-identification} and \eqref{eq:rho-W-twist-identifications}, yield isomorphisms:
    \begin{equation*}
        \begin{tikzcd}[column sep=16pt]
    	\phi_\mf{S}^{\ast\tw{n+1}}\mc{V}(\mf{S}, E_\pi)/u & \phi_W^{\ast\tw{n+1}}\mc{V}(W, p) & \phi_\mf{S}^{\ast\tw{n+1}}\mc{V}(\mf{S}, E_{\pi'})/u.
    	\arrow["\sim", from=1-1, to=1-2]
            \arrow["\sim"', from=1-3, to=1-2]
        \end{tikzcd}
    \end{equation*}
    The preceding diagram induces an isomorphism,
    \begin{equation*}
        \j_{\crys,\pi,\pi'}^{\tw{n},\mathsmaller{\mc{V}}}\colon \bb{D}^{\prime,\tw{n}}_{\mr{crys},\pi}(\mc{V})\isomto \bb{D}^{\prime,\tw{n}}_{\mr{crys},\pi'}(\mc{V}).
    \end{equation*}
    It is easy to see that the isomorphisms $\j_{\crys,\pi,\pi'}^{\tw{n},\mc{V}}$ are functorial in $\mc{V}$ and satisfy the following cocycle condition:
    \begin{equation*} 
        \j^{\tw{n},\mathsmaller{\mc{V}}}_{\crys,\pi',\pi''}\circ \j_{\crys,\pi,\pi'}^{\tw{n},\mathsmaller{\mc{V}}} = \j^{\tw{n},\mathsmaller{\mc{V}}}_{\crys,\pi,\pi''}.
    \end{equation*}
    Thus, we obtain natural equivalences,
    \begin{equation*}
        \j^{\tw{n}}_{\crys,\pi,\pi'}\colon \bb{D}_{\mr{crys},\pi}^{\prime,\tw{n}} \isomto \bb{D}_{\mr{crys},\pi'}^{\prime,\tw{n}}
    \end{equation*}
    in $\cat{Func}(\cat{Perf}((\mc{O}_K)_\smallprism),\cat{Perf}(W))$ satisfying the cocycle condition.
\end{construction}

\begin{defn}
    The \emph{$n\textrm{-twisted}$ crystalline realisation functor} is defined as
    \begin{equation*}
        \bb{D}_\mr{crys}^{\prime,\tw{n}} \defeq \varprojlim_\pi \bb{D}_{\mr{crys},\pi}^{\prime,\tw{n}} \colon \cat{Perf}((\mc{O}_K)_\smallprism) \longrightarrow \cat{Perf}(W).
    \end{equation*}
\end{defn}

We end this section by comparing $\bb{D}^{\prime,\tw{n}}_\crys$ and $\bb{D}^{\tw{n}}_\crys$.
Namely, via Corollary \ref{cor:twisted-crys-identification}, for any $\pi$ we have a canonical identification
\begin{equation*}
    \psi_{\crys,\pi}^{\tw{n}}\colon \bb{D}^{\tw{n}}_\crys\isomto \bb{D}^{\prime,\tw{n}}_{\crys,\pi}.
\end{equation*}
By combining Propositions \ref{prop:big-diagram-commutes} and \ref{prop:crys-rho-W-factorization}, we have the following commutative diagram of stacks:
\begin{equation*}
    \begin{tikzcd}[column sep=50pt, row sep=30pt]
        {\Spf(\mf{S})} & {\Spf(W)} & {\Spf(\mf{S})} \\
        & {\mc{O}_K^\smallprism}
        \arrow["{\rho_{\pi}^{(n+1)}}"{description}, from=1-1, to=2-2]
        \arrow["{(u=0)}"{description}, from=1-2, to=1-1]
        \arrow["{(u=0)}"{description}, from=1-2, to=1-3]
        \arrow["{\rho^{(n)}_\crys}"{description}, from=1-2, to=2-2]
        \arrow["{\rho_{\pi'}^{(n+1)}}"{description}, from=1-3, to=2-2]
    \end{tikzcd}
\end{equation*}
Thus, it is easy to deduce that $\j^{\tw{n}}_{\crys,\pi,\pi'} \circ \psi_{\crys,\pi}^{\tw{n}}=\psi_{\crys,\pi'}^{\tw{n}}$.
From the preceding discussion we obtain that,
\begin{prop}\label{prop:prime-to-nonprime-crys}
    The following morphism is an equivalence in $\cat{Func}(\cat{Perf}((\mc{O}_K)_\smallprism),\cat{Perf}(W))$:
    \begin{equation*}
        \psi_\crys^{\tw{n}} \defeq (\psi_{\crys,\pi}^{\tw{n}})\colon \bb{D}^{\tw{n}}_\crys \longrightarrow \varprojlim_{\pi} \bb{D}^{\prime,\tw{n}}_{\crys,\pi} = \bb{D}^{\prime,\tw{n}}_\crys.
    \end{equation*}
\end{prop}

\subsubsection{The twisted de Rham realisation} 

We now perform a similar analysis to define the twisted de Rham realisation using the Breuil--Kisin theory, and to compare this to the stack-theoretic definition.

\begin{defn}
    Let $\pi$ be a uniformiser of $\mc{O}_K$ and $n\geqslant 0$ an integer.
    The \emph{$\pi\textrm{-indexed}$ $n\textrm{-twisted}$ de Rham realisation functor} is defined as follows:
    \begin{equation*}
        \mathbb{D}_{\mr{dR},\pi}^{\prime,\tw{n}}\colon \cat{Perf}((\mc{O}_K)_\smallprism) \longrightarrow \cat{Perf}(\mc{O}_K),\qquad \mc{V} \longmapsto \phi_\mf{S}^{\ast\tw{n+1}}\mc{V}(\mf{S}, E_\pi)/ E_\pi.
    \end{equation*}
\end{defn}

Observe that, by the setup as in Section \ref{sss:mixed-BK-prism}, there is a commuting diagram of maps,
\begin{equation}\label{eq:O_K-quotient-product-pi-pi'}
    \begin{tikzcd}[sep=2.25em]
    	{(\mf{S}, E_\pi)} & {(\mf{S}_{\pi,\pi'}, I_{\pi,\pi'})} & {(\mf{S}, E_{\pi'})} \\
    	& {\mc{O}_K},
    	\arrow["p_1",from=1-1, to=1-2]
    	\arrow[from=1-1, to=2-2]
    	\arrow[from=1-2, to=2-2]
    	\arrow["p_2"', from=1-3, to=1-2]
    	\arrow[from=1-3, to=2-2]
    \end{tikzcd}
\end{equation}
where the horizontal arrows are the two projection maps in $(\mc{O}_K)_\smallprism$ and the rest are given by reduction modulo the prismatic ideal.

\begin{construction}\label{const:dR-transition-isom}
    Let $\mc{V}$ be an object of $\cat{Perf}((\mc{O}_K)_\smallprism)$ and $n\geqslant 0$ an integer.
    Applying the crystal property for the top row of the diagram in \eqref{eq:O_K-quotient-product-pi-pi'}, we obtain a commutative diagram of equivalences:
    \begin{equation*}
        \begin{tikzcd}[column sep=16pt]
    	\phi_{\mf{S}}^{\ast\tw{n+1}}\mc{V}(\mf{S}, E_\pi) \otimes_\mf{S} \mf{S}_{\pi,\pi'} & \phi_{\mf{S}_{\pi,\pi'}}^{\ast\tw{n+1}}\mc{V}(\mf{S}_{\pi,\pi'}, I_{\pi,\pi'}) & \phi_{\mf{S}}^{\ast\tw{n+1}}\mc{V}(\mf{S}, E_{\pi'}) \otimes_\mf{S}\mf{S}_{\pi,\pi'}. 
            \arrow["\sim", from=1-1, to=1-2]
            \arrow["\sim"', from=1-3, to=1-2]
        \end{tikzcd}
    \end{equation*}

    From the preceding diagram and commutativity of \eqref{eq:O_K-quotient-product-pi-pi'}, note that by tensoring along $\mf{S}_{\pi,\pi'}\to\mc{O}_K$ we obtain an isomorphism,
    \begin{equation*}
    \j^{\tw{n},\mathsmaller{\mc{V}}}_{\dR,\pi,\pi'}\colon \bb{D}^{\prime,\tw{n}}_{\dR,\pi}(\mc{V})\isomto \bb{D}^{\prime,\tw{n}}_{\dR,\pi'}(\mc{V}).
    \end{equation*}
    The isomorphisms $\j_{\dR,\pi,\pi'}^{\tw{n},\mathsmaller{\mc{V}}}$ are functorial in $\mc{V}$ and satisfy the following cocycle condition:
    \begin{equation*} 
        \j^{\tw{n},\mathsmaller{\mc{V}}}_{\dR,\pi',\pi''}\circ \j_{\dR,\pi,\pi'}^{\tw{n},\mathsmaller{\mc{V}}}=\j^{\tw{n},\mathsmaller{\mc{V}}}_{\dR,\pi,\pi''}.
    \end{equation*}
    Thus, we obtain natural equivalences
    \begin{equation*}
        \j^{\tw{n}}_{\dR,\pi,\pi'}\colon \bb{D}_{\mr{dR},\pi}^{\prime,\tw{n}}\isomto \bb{D}_{\mr{dR},\pi'}^{\prime,\tw{n}}
    \end{equation*}
    in $\cat{Func}(\cat{Perf}((\mc{O}_K)_\smallprism),\cat{Perf}(\mc{O}_K))$ satisfying the cocycle condition.
\end{construction}

\begin{defn} 
    The \emph{$n\textrm{-twisted}$ de Rham realisation functor} is defined as
    \begin{equation*}
        \bb{D}_\mr{dR}^{\prime,\tw{n}} \defeq \varprojlim_\pi \bb{D}_{\mr{dR},\pi}^{\prime,\tw{n}} \colon \cat{Perf}((\mc{O}_K)_\smallprism) \longrightarrow \cat{Perf}(\mc{O}_K).
    \end{equation*}
\end{defn}

We end this section by comparing $\bb{D}^{\prime,\tw{n}}_\dR$ and $\bb{D}^{\tw{n}}_\dR$.
Via Corollary \ref{cor:twisted-dR-identification}, for any $\pi$, we have a canonical identification
\begin{equation*}
\psi_{\dR,\pi}^{\tw{n}}\colon \bb{D}^{\tw{n}}_\dR\isomto \bb{D}^{\prime,\tw{n}}_{\dR,\pi}.
\end{equation*}
Moreover, we have the following $\textrm{(}2\textrm{-)commutative}$ diagram:
\begin{equation*}
    \begin{tikzcd}[column sep= 6em, row sep=large]
    	& {\Spf(\mf{S})} \\
    	{\Spf(\mf{S}_{\pi,\pi'})} & {\Spf(\mc{O}_K)} & {\mc{O}_K^\smallprism,} \\
    	& {\Spf(\mf{S})}
    	\arrow["{\rho^{\tw{1}}_\pi}"{description}, from=1-2, to=2-3]
    	\arrow["{p_1}"{description}, from=2-1, to=1-2]
    	\arrow["{p_2}"{description}, from=2-1, to=3-2]
    	\arrow["{E_{\pi}=0}"{description}, from=2-2, to=1-2]
    	\arrow["{(E_{\pi}=E_{\pi'}=0)}"{description}, from=2-2, to=2-1]
    	\arrow["{\rho_\dR}", from=2-2, to=2-3]
    	\arrow["{E_{\pi'}=0}"{description}, from=2-2, to=3-2]
    	\arrow["{\rho^{\tw{1}}_{\pi'}}"{description}, from=3-2, to=2-3]
    \end{tikzcd}
\end{equation*}
and therefore, it is easy to deduce that $\j^{\tw{n}}_{\dR,\pi,\pi'}\circ \psi_{\dR,\pi}^{\tw{n}} = \psi_{\dR,\pi'}^{\tw{n}}$.
From this we obtain the following.
\begin{prop}\label{prop:prime-to-nonprime-deRham}
    The following morphism is an equivalence in $\cat{Func}(\cat{Perf}((\mc{O}_K)_\smallprism),\cat{Perf}(\mc{O}_K))$:
    \begin{equation*}
        \psi_\dR^{\tw{n}} \eqdef (\psi_{\dR,\pi}^{\tw{n}})\colon \bb{D}^{\tw{n}}_\dR \longrightarrow \varprojlim \bb{D}^{\prime,\tw{n}}_{\dR,\pi} = \bb{D}^{\prime,\tw{n}}_\dR.
    \end{equation*}
\end{prop}

\subsubsection{Non-stacky approach to the twisted crystalline-de Rham comparison}

We shall now give a non-stacky interpretation of Theorem \ref{thm:twisted-crys-dR-comparison}. We first require the following construction:

\begin{construction} 
    Fix a uniformiser $\pi$ and an integer $n\geqslant a$. From Equations \eqref{eq:BK-twist-identification} and \eqref{eq:rho-W-twist-identifications}, and Proposition \ref{prop:big-diagram-commutes} we obtain the following equivalences in $\cat{Perf}(\tilde{S})$:
    \begin{equation*}
        \phi_W^{\ast\tw{n+1}}\mc{V}(W, p)\otimes_{W} \tilde{S} \isomto \phi^{\ast n}_{\tilde{S}}\mc{V}(\tilde{S}, p) \isomfrom \phi_\mf{S}^{\ast\tw{n+1}}\mc{V}(\mf{S}, E_\pi) \otimes_{\mf{S}}\tilde{S}.
    \end{equation*}
    Base change of the preceding isomorphisms, along the map $\tilde{S}\to\mc{O}_K$ from Proposition \ref{prop:BO-prism-to-O_K}, furnishes us with an equivalence in $\cat{Perf}(\mc{O}_K)$:
    \begin{equation*}
        \phi^{\ast\tw{n+1}}_W\mc{V}(W, p) \otimes_W \mc{O}_K\isomto \phi_{\mf{S}}^{\ast\tw{n+1}}\mc{V}(\mf{S}, E_\pi)/E_\pi.
    \end{equation*}
    Finally, using the isomorphism (see Construction \ref{const:crys-transition-isom}),
    \begin{equation*}
        \phi_W^{\ast\tw{n+1}}\mc{V}(W, p) \isomto \phi_{\mf{S}}^{\ast \tw{n+1}}\mc{V}(\mf{S}, E_\pi)/u,
    \end{equation*}
     gives rise to an isomorphism
    \begin{equation*}
        \iota_\pi^{\prime,\tw{n},\mc{V}}\colon \bb{D}_{\mr{crys},\pi}^{\prime,\tw{n}}(\mc{V})\otimes_W \mc{O}_K\isomto \bb{D}_{\mr{dR},\pi}^{\prime,\tw{n}}(\mc{V}).
    \end{equation*}
\end{construction}

By construction, it is clear that the isomorphisms $\iota_\pi^{\prime,\tw{n},\mathsmaller{\mc{V}}}$ are functorial in $\mc{V}$, and thus give rise to an isomorphism $\iota_\pi^{\prime,\tw{n}}\colon \bb{D}_{\mr{crys},\pi}^{\prime,\tw{n}} \otimes_W \mc{O}_K \to \bb{D}_{\mr{dR},\pi}^{\prime,\tw{n}}$ in $\cat{Func}(\cat{Perf}((\mc{O}_K)_\smallprism),\cat{Perf}(\mc{O}_K))$.

We now compare this identification with that obtained from Theorem \ref{thm:twisted-crys-dR-comparison}. We begin with the following where it is useful to write $\rho^{\tw{n}}_\crys\otimes 1$ for the composition,
\begin{equation*}
    \Spf(\mc{O}_K)\to \Spf(W)\xrightarrow{\rho^{\tw{n}}_\crys} \mc{O}_K^\smallprism,
\end{equation*}
appearing in Theorem \ref{thm:twisted-crys-dR-comparison}.

\begin{prop}\label{prop:crys-dR-comparison-using-prisms} 
    The identification of \emph{Theorem \ref{thm:twisted-crys-dR-comparison}} naturally factors as,
    \begin{equation*}
        \rho_{\crys}^{\tw{n}}\otimes 1\isomto \tilde{\rho}^{\tw{n}}_{\pi}\circ i \isomfrom \rho_{\dR}^{\tw{n}},
    \end{equation*}
    where $i \colon \Spf(\mc{O}_K) \rightarrow \Spf(\tilde{S})$ is the map described in \emph{Proposition \ref{prop:BO-prism-to-O_K}}.
\end{prop}
\begin{proof}
    Assume that $n \geqslant a$ and consider the following diagram:
    \begin{equation*}
        \begin{tikzcd}[column sep=3.15em]
            &{\Spf(\mf{S})} &{\Spf(\mf S)} &{} 
            \\{\Spf(\mc{O}_K)} &{\Spf(\tilde{S})} &{} &{\mc{O}_K^\smallprism} 
            \\{} &{\Spf(W)} &{\Spf(W).} &{}    
        	\arrow["{(E_\pi=0)}"{description}, from=2-1, to=1-2]
        	\arrow["{i}"{description}, from=2-1, to=2-2]
                \arrow["{\textrm{nat}.}"{description}, from=2-1, to=3-2]
                \arrow["{\phi^{n+1}_\mf{S}}"{description}, from=1-2, to=1-3]
        	\arrow["{}", from=2-2, to=1-2]
                \arrow["{\tilde{\rho}^{\tw{n}}_{\pi}}"{description}, from=2-2, to=2-4]
                \arrow["{}", from=2-2, to=3-2]
                \arrow["{\phi^{n+1}_W}"{description}, from=3-2, to=3-3]
                \arrow["{\rho_{\pi}}"{description}, from=1-3, to=2-4]
                \arrow["{\rho_{W}}"{description}, from=3-3, to=2-4]
        \end{tikzcd}
    \end{equation*}
    where the left two triangles are commutative by Proposition \ref{prop:BO-prism-to-O_K}, and the right two trapeziums are $2\textrm{-commutative}$ with obvious identifications of the compositions. 
    
    Let $\rho^{\prime,\tw{n}}_{\dR}$ (resp.\ $\rho^{\prime,\tw{n}}_{\crys}\otimes 1$) denote the upper (resp.\ lower) composite maps $\Spf(\mc{O}_K)\to \mc{O}_K^\smallprism$. 
    By the commutativity of the above diagram, we obtain isomorphisms,
    \begin{equation*}
        \rho^{\prime,\tw{n}}_{\crys} \otimes 1 \isomto \tilde{\rho}^{\tw{n}}_{\pi}\circ i \isomfrom \rho^{\prime,\tw{n}}_{\dR}.
    \end{equation*}
    By Propositions \ref{prop:dR-twisted-mfS-identification} and \ref{prop:crys-rho-W-factorization} we have isomorphisms $\rho^{\prime,\tw{n}}_{\dR}\isomto\rho^{\tw{n}}_{\dR}$ and $\rho^{\prime,\tw{n}}_{\crys}\otimes 1\isomto\rho^{\tw{n}}_{\crys}\otimes 1$, respectively, thus proving both the claims.
%
\end{proof}

From this, and the construction of $\psi^{\tw{n}}_{\crys,\pi}$ and $\psi^{\tw{n}}_{\dR,\pi}$, we easily deduce the following compatibility between $\iota^{\tw{n}}$ and $\iota^{\prime,\tw{n}}_\pi$.

\begin{prop}\label{prop:identification-of-two-twisted-crys-dR-comps}
    For any $n\geqslant a$, the following diagram commutes in $\cat{Func}(\cat{Perf}((\mc{O}_K)_\smallprism),\cat{Perf}(\mc{O}_K))$:
        \begin{equation*}
        \begin{tikzcd}[sep=2.25em]
            {\bb{D}_{\mr{crys}}^{\tw{n}} \otimes_W \mc{O}_K} & {\bb{D}_{\mr{dR}}^{\tw{n}}} \\
            {\bb{D}_{\mr{crys},\pi}^{\prime \tw{n}} \otimes_W \mc{O}_K} & {\bb{D}^{\prime,\tw{n}}_{\mr{dR},\pi}}.
            \arrow["{\iota^{\tw{n}}}", from=1-1, to=1-2]
            \arrow["{\psi^{\tw{n}}_{\crys,\pi}\otimes 1}"', from=1-1, to=2-1]
            \arrow["{\psi^{\tw{n}}_{\dR,
            \pi}}", from=1-2, to=2-2]
            \arrow["{\iota^{\prime,\tw{n}}_\pi}"', from=2-1, to=2-2]
        \end{tikzcd}
    \end{equation*}
\end{prop}

Using Proposition \ref{prop:identification-of-two-twisted-crys-dR-comps}, it is simple to obtain the following compatibility between the various $\iota^{\prime,\tw{n}}_\pi$ as $\pi$ varies.

\begin{prop}\label{prop:crys-dR-comparison-pi-independence}
    For uniformisers $\pi$ and $\pi'$ of $\mc{O}_K$, and an integer $n\geqslant a$, the following diagram in $\cat{Func}(\cat{Perf}((\mc{O}_K)_\smallprism),\cat{Perf}(\mc{O}_K))$ naturally commutes:
    \begin{equation}\label{eq:uniformizer-comp-crys-dR}
        \begin{tikzcd}[sep=2.25em]
            {\bb{D}_{\mr{crys},\pi}^{\prime,\tw{n}} \otimes_W \mc{O}_K} & {\bb{D}_{\mr{dR},\pi}^{\prime,\tw{n}}} \\
            {\bb{D}_{\mr{crys},\pi'}^{\prime,\tw{n}} \otimes_W \mc{O}_K} & {\bb{D}^{\prime,\tw{n}}_{\mr{dR},\pi'}}.
            \arrow["{\iota^{\prime,\tw{n}}_\pi}", from=1-1, to=1-2]
            \arrow["{\j^{\tw{n}}_{\crys,\pi,\pi'}\otimes 1}"', from=1-1, to=2-1]
            \arrow["{\j^{\tw{n}}_{\dR,\pi,\pi'}}", from=1-2, to=2-2]
            \arrow["{\iota^{\prime,\tw{n}}_{\pi'}}"', from=2-1, to=2-2]
        \end{tikzcd}
    \end{equation}
\end{prop}

From the compatibility in $\pi$ from Proposition \ref{prop:crys-dR-comparison-pi-independence}, it is clear that we may pass to the limit to obtain an equivalence,
\begin{equation*}
    \iota^{\prime,\tw{n}}\colon \bb{D}_\crys^{\prime,\tw{n}}\otimes_W\mc{O}_K \isomto \bb{D}^{\prime,\tw{n}}_{\dR}
\end{equation*}
in $\cat{Func}(\cat{Perf}((\mc{O}_K)_\smallprism),\cat{Perf}(\mc{O}_K))$, for any $n\geqslant a$. Moreover, this satisfies the compatibility 
 \begin{equation*}
        \begin{tikzcd}[sep=2.25em]
            {\bb{D}_{\mr{crys}}^{\tw{n}} \otimes_W \mc{O}_K} & {\bb{D}_{\mr{dR}}^{\tw{n}}} \\
            {\bb{D}_{\mr{crys}}^{\prime \tw{n}} \otimes_W \mc{O}_K} & {\bb{D}^{\prime,\tw{n}}_{\mr{dR}}}.
            \arrow["{\iota^{\tw{n}}}", from=1-1, to=1-2]
            \arrow["{\psi^{\tw{n}}_\crys\otimes 1}"', from=1-1, to=2-1]
            \arrow["{\psi^{\tw{n}}_\dR}", from=1-2, to=2-2]
            \arrow["{\iota^{\prime,\tw{n}}}"', from=2-1, to=2-2]
        \end{tikzcd}
    \end{equation*}
as in the proof of Proposition \ref{prop:identification-of-two-twisted-crys-dR-comps}.

\section{Twisted crystalline and de Rham cohomology}

In this section, we shall apply the material of Section \ref{s:crys-de-rham-comp} to perfect prismatic crystals of the form $Rf^\smallprism_\ast\mc{O}^{\smallprism}$ and more generally.
In particular, we will deduce our extension of the Berthelot--Ogus isomorphism to the case of coefficients.
To formulate this result correctly, we introduce the notion of \emph{twisted crystalline and de Rham cohomology}, and study some of their properties.

\begin{nota}
    We continue to use the notation from \globalnotationref, Notation \ref{nota:Witt-vector-stuff} and Section \ref{subsec:breuil_kisin_prism}.
    We further assume that $f\colon X\to\Spf(\mc{O}_K)$ is a smooth proper morphism.
\end{nota}

\subsection{Twisted crystalline cohomology}

We begin with the easier of the two definitions, twisted crystalline cohomology, which ultimately reduces to classical notions.

\begin{defn}
    For $n\geqslant0$, the \emph{$n\textrm{-twisted}$ crystalline cohomology functor} is given by
    \begin{equation*}
        \mr{R}\Gamma^{\tw{n}}_\crys(-) = \mr{R}\Gamma^{\tw{n}}_\crys(X_k/W,-) \colon \cat{Perf}(X^\smallprism) \longrightarrow \cat{Perf}(W),\qquad \mc{V} \longmapsto \bb{D}^{\tw{n}}_\crys(\mr{R}f^\smallprism_\ast \mc{V}).
    \end{equation*}
    We denote the associated cohomology groups by $\mr{H}^{i,\tw{n}}_\crys(\mc{V}) = \mr{H}^{i,\tw{n}}_\crys(X_k/W, \mc{V})$, and when $\mc{V} = \mc{O}_\smallprism$ we omit $\mc{V}$ from the notation. 
\end{defn}

Recall that for a quasi-syntomic $k\textrm{-scheme}$ $Z$ there is an equivalence of $\infty\textrm{-categories}$,
\begin{equation*}
    (-)^\crys\colon \cat{Perf}(Z^\smallprism)\isomto \cat{Perf}((Z/W)_\crys), \qquad \mc{V} \longmapsto \mc{V}^\crys,
\end{equation*} 
 (e.g.\ see \cite[Theorem 6.4]{GuoReinecke}). We then define the functor,
\begin{equation*}
    (-)^\crys\colon \cat{Perf}(X^\smallprism) \longrightarrow \cat{Perf}((X_k/W)_\crys), \qquad \mc{V} \longmapsto (\mc{V}|_{(X_k)_\smallprism})^\crys.
\end{equation*}
These functors are compatible with Frobenius morphisms, i.e., $(F_X^\ast\mc{V})^\crys$ is naturally identified with $\phi_\crys^\ast\mc{V}^\crys$, where $\phi_\crys$ is the Frobenius endomorphism of the crystalline topos.

\begin{prop}\label{prop:twisted-crys-classical-comparison}
    For each $n\geqslant 0$, there are natural identifications
    \begin{equation*}
        \mr{R}\Gamma_\crys^{\tw{n}}(\mc{V})\simeq \phi_W^{\ast n}\mr{R}\Gamma((X_k/W)_\crys,\mc{V}^\crys)\simeq \RGamma((X_k^{\tw{n}}/W)_\crys,(\mc{V}^\crys)^{\tw{n}}),
    \end{equation*}
    where the twists in the last term are relative to $F_k$. If $\mc{V}$ is equipped with a Frobenius structure, then these isomorphisms are Frobenius-equivariant.
\end{prop}
\begin{proof}
    By Footnote \ref{footnote:pushforwards-agree} and Corollary \ref{cor:twisted-crys-identification}, we have natural identifications,
    \begin{equation*}
        \mr{R}\Gamma^{\tw{n}}_\crys(\mc{V}) \simeq \phi_W^{\ast n}(F_{\mathcal{O}_K}^\ast Rf^{\smallprism}_{\ast}\mc{V})(W, p) \simeq \phi_W^{\ast \tw{n+1}}\mr{R}\Gamma_\smallprism(X_k/W,\mc{V}),
    \end{equation*}
    and the last term is naturally isomorphic to $\phi_W^{\ast n}\mr{R}\Gamma((X_k/W)_\crys,\mc{V}^\crys)$ by \cite[Theorem 6.4]{GuoReinecke}.
    This shows the first natural isomorphism, and the second one follows from \cite[\href{https://stacks.math.columbia.edu/tag/07MS}{Tag 07MS}, \href{https://stacks.math.columbia.edu/tag/07MU}{Tag 07MU}, \href{https://stacks.math.columbia.edu/tag/07MY}{Tag 07MY}]{StacksProject}.
    The last claim follows by a similar argument and the same references.
\end{proof}

In particular, we see that there is a natural identification,
\begin{equation*}
    \mr{R}\Gamma^{\tw{0}}_\crys(\mc{V})\simeq \RGamma((X_k/W)_\crys,\mc{V}^\crys),
\end{equation*}
and we shall often implicitly make this identification. Finally, note that a Frobenius structure $\varphi_\mc{V}\colon F_X^\ast\mc{V}[\nicefrac{1}{\mc{I}_\smallprism}]\isomto \mc{V}[\nicefrac{1}{\mc{I}_\smallprism}]$ induces an isomorphism in $\cat{Perf}(K_0)$:
\begin{equation}\label{eq:varphi-bar-crys}
    \ov{\varphi}^{\tw{n}}_\mc{V}\colon \mr{R}\Gamma^{\tw{n}}_\crys(\mc{V})[\nicefrac{1}{p}]\isomto \mr{R}\Gamma^{\tw{0}}_\crys(\mc{V})[\nicefrac{1}{p}]=\RGamma((X_k/W)_\crys,\mc{V}^\crys)[\nicefrac{1}{p}],
\end{equation}
interpreted either prismatically, or in terms of the Frobenius structure on crystalline cohomology (e.g.\ see \cite[\href{https://stacks.math.columbia.edu/tag/07N5}{Tag 07N5}]{StacksProject}), which agree via Proposition \ref{prop:twisted-crys-classical-comparison}.

\subsection{Twisted de Rham cohomology}\label{ss:twisted-de-Rham}

We now come to the more subtle definition of twisted de Rham cohomology, which does not have a classical interpretation.

\begin{defn}\label{defn:n-twisted-deRham}
    For $n\geqslant0$, the \emph{$n\textrm{-twisted}$ de Rham cohomology functor} is given by
    \begin{equation*}
        \mr{R}\Gamma^{\tw{n}}_\dR(-) = \mr{R}\Gamma^{\tw{n}}_\dR(X/\mc{O}_K,-) \colon \cat{Perf}(X^\smallprism) \longrightarrow \cat{Perf}(\mc{O}_K),\qquad \mc{V} \longmapsto \bb{D}^{\tw{n}}_\dR(\mr{R}f^\smallprism_\ast \mc{V}).
    \end{equation*}
    We denote the associated cohomology groups by $\mr{H}^{i,\tw{n}}_\dR(\mc{V}) = \mr{H}^{i,\tw{n}}_\dR(X/\mc{O}_K,\mc{V})$, and when $\mc{V} = \mc{O}_\smallprism$ we omit $\mc{V}$ from the notation. 
\end{defn}

We now observe that twisted de Rham cohomology recovers the de Rham cohomology in the non-twisted (i.e., $n=0$) case.
To state this precisely, recall that we have the relative de Rham stack $(X/\mc{O}_K)^\mr{dR}$ as in \cite[Definition 2.5.3]{BhattNotes}.
A coefficient sheaf for de Rham cohomology is a perfect complex $\mc{W}$ on $(X/\mc{O}_K)^\dR$.
Additionally, if $\mc{W}$ is a vector bundle, then the data of $\mc{W}$ is equivalent to the data of a vector bundle with integrable connection $(\mc{W}_0, \nabla)$ on $X$, and the cohomology of $\mc{W}$ is the usual cohomology of the de Rham complex associated to $(\mc{W}_0, \nabla)$ (see \cite[Remark 2.5.8]{BhattNotes}).

Next, recall that we have the following morphism of stacks (see \cite[Constructions 5.3.13 \& 5.3.5 and Footnote 62]{BhattNotes}):
\begin{equation*}
    i_{\mr{dR},X/\mc{O}_K}\colon (X/\mc{O}_K)^\dR \longrightarrow X^\smallprism.
\end{equation*}
For an object $\mc{V}$ of $\cat{Perf}(X_\smallprism)$, let us write $\mc{V}^\dR$ for $i_{\dR,X/\mc{O}_K}^\ast(\mc{V})$, and we claim the following:
\begin{prop}\label{prop:de-Rham-base-change}
    There exists a natural identification in $\cat{Perf}(\mc{O}_K)$:
        \begin{equation*}
            \RGamma_\dR^{\tw{0}}(X/\mc{O}_K,\mc{V}) \simeq \RGamma_\dR(\mc{V}^\dR) .
        \end{equation*}
    \end{prop}
\begin{proof}
    Note that we have the following commutative diagram of formal stacks:
    \begin{equation}\label{eq:relative_deRham_stack}
        \begin{tikzcd}[column sep=large]
    	{(X/\mathcal{O}_K)^\dR} & {X^\dR} & {X^\smallprism} \\
    	{\Spf(\mathcal{O}_K)} & {\mathcal{O}_K^\dR} & {\mathcal{O}_K^\smallprism} \\
    	& {\Spf(\mathbb{Z}_p)} & {\mathbb{Z}_p^\smallprism},
    	\arrow[from=1-1, to=1-2]
    	\arrow["{f^\mr{dR}_{X/\mathcal{O}_K}}"', from=1-1, to=2-1]
    	\arrow["{i_{\mathrm{dR},X}}", from=1-2, to=1-3]
    	\arrow["{f^\mr{dR}}"', from=1-2, to=2-2]
    	\arrow["{f^\smallprism}", from=1-3, to=2-3]
    	\arrow[from=2-1, to=2-2]
    	\arrow["{i_{\mathrm{dR},\mc{O}_K}}", from=2-2, to=2-3]
    	\arrow[from=2-2, to=3-2]
    	\arrow[from=2-3, to=3-3]
    	\arrow["{i_{\mathrm{dR}}}", from=3-2, to=3-3]
        \end{tikzcd}
    \end{equation}
    where the composition of the top horizontal arrows defines $i_{\mr{dR},X/\mc{O}_K}$ (e.g.\ see \cite[Footnote 62]{BhattNotes}), and the bottom horizontal arrow $i_{\mathrm{dR}}$ is the restriction to $\mathbb{G}_m/\mathbb{G}_m$ of the map $i_{\dR} \colon \mathbb{A}^1/\mathbb{G}_m \rightarrow \mathbb{Z}_p^{\textup{N}}$ from \cite[Construction 5.3.4]{BhattNotes}, and this restriction identifies with the de Rham point of \cite[Example 3.2.6]{BhattLurieAbsolute}.
    Additionally, using the definitions, it is easy to verify that in \eqref{eq:relative_deRham_stack} the bottom right square and the outer square on the right are Cartesian, and so, it follows that the top right square is also Cartesian.
    Moreover, the top left square is Cartesian by definition and thus it follows that the outer square on the top is also Cartesian.
    
    Next, we claim that the composition of the middle horizontal arrows in \eqref{eq:relative_deRham_stack} identifies with the de Rham point $\rho_{\dR}$ of Definition \ref{defn:ntwisted_deRham_point}.
    Consider the following commutative diagram of stacks:
    \begin{equation*}
        \begin{tikzcd}
            \Spf(\mc{O}_K) = \Spf(\mf{S}/E_{\pi}) & \mc{O}_K^{\textup{HT}} & \Spf(\mc{O}_K) & \mc{O}_K^{\dR} \\
            \Spf(\mf{S}) & \mc{O}_K^{\smallprism} && \mc{O}_K^{\smallprism},
            \arrow["\rho_{(\mathfrak{S}, E_{\pi})}^{\textup{HT}}", from=1-1, to=1-2]
            \arrow["\mr{id}", curve={height=-24pt}, from=1-1, to=1-3]
            \arrow[hook, from=1-1, to=2-1]
            \arrow[from=1-2, to=1-3]
            \arrow[hook, from=1-2, to=2-2]
            \arrow[from=1-3, to=1-4]
            \arrow["i_{\mathrm{dR},\mc{O}_K}", from=1-4, to=2-4]
            \arrow["\rho_{(\mathfrak{S}, E_{\pi})}", from=2-1, to=2-2]
            \arrow["F_{\smallprism}", from=2-2, to=2-4]
        \end{tikzcd}
    \end{equation*}
    where the composition of the left vertical arrow with the bottom horizontal arrows identifies with the de Rham point $\rho_{\dR}$ by Proposition \ref{prop:dR-twisted-mfS-identification}, and it is clear that the composition of the top horizontal arrows with the right vertical arrow coincides with the composition of the middle horizontal arrows in \eqref{eq:relative_deRham_stack}.
    In the preceding diagram, the left square commutes by definition and the right square commutes by Lemma \ref{lem:HT_dR_stacks_commute} below, hence, the claim follows.

    Finally, to show the claimed identification on cohomologies, note that if we have derived base change for formal stacks for the top outer square, we have the following identifications:
    \begin{equation*}
        \mr{R}\Gamma_\mr{dR}(X/\mathcal{O}_K,\mc{V}^\mr{dR})\simeq R\big(f^\mr{dR}_{X/\mathcal{O}_K}\big)_\ast i_{\mr{dR},X/\mathcal{O}_K}^\ast \mc{V}\simeq \rho_\mr{dR}^{\ast}Rf^\smallprism_\ast\mc{V}\simeq R^{\tw{0}}_\mr{dR}(X/\mc{O}_K,\mc{V}),
    \end{equation*}
    and thus it suffices to verify that derived base change holds for this top outer square.

    Note that the top outer square in \eqref{eq:relative_deRham_stack} is Cartesian. As $\rho_\mr{dR}\colon \Spf(\mc{O}_K)\to \mc{O}_K^\smallprism$ factorises as the locally closed immersion $\Spf(\mc{O}_K)\to\Spf(\mf{S})$ followed by the flat cover $\rho_\mf{S}^{\tw{1}}\colon \Spf(\mf{S})\to\mc{O}_K^\smallprism$, we obtain the diagram
    \begin{equation*}\label{eq:relative_deRham_prismatisation_stack}
        \begin{tikzcd}[column sep=large]
        	{(X/\mathcal{O}_K)^\dR} & {\textup{WCart}_{X/\mathfrak{S}}} & {X^\smallprism} \\
        	{\Spf(\mathcal{O}_K)} & {\Spf(\mathfrak{S})} & {\mathcal{O}_K^\smallprism,}
        	\arrow[from=1-1, to=1-2]
        	\arrow[""{name=0, anchor=center, inner sep=0}, "{f^\mr{dR}_{X/\mathcal{O}_K}}"', from=1-1, to=2-1]
        	\arrow[from=1-2, to=1-3]
        	\arrow[""{name=1, anchor=center, inner sep=0}, from=1-2, to=2-2]
        	\arrow[""{name=2, anchor=center, inner sep=0}, "{f^\smallprism}", from=1-3, to=2-3]
        	\arrow[from=2-1, to=2-2]
        	\arrow["{\rho_{\mathfrak{S}}^{\tw{1}}}"', from=2-2, to=2-3]
            \arrow["{(2)}"{description}, draw=none, from=1, to=0]
            \arrow["{(1)}"{description}, draw=none, from=2, to=1]
        \end{tikzcd}
    \end{equation*}  
    where $\mr{WCart}_{X/\mf{S}}$ is as in \cite[Variant 5.1]{BhattLuriePrismatization} (although note that the $(A,I)$ in loc.\@ cit.\@ is actually $(\mf{S},(E^{\tw{1}}))$ here) and both diagrams (1) and (2) are Cartesian.  
    
    As $X^\smallprism$ is a colimit of flat simplicial $\mc{O}_K$-schemes (given by Breuil--Kisin prisms over a framed cover of $X$), and this is stable under pullback, the Cartesian diagram (2) satisfies base change by \cite[Proposition A.0.2]{Hauck}.
    Moreover, the Cartesian diagram (1) satisfies derived base change by Footnote \ref{footnote:pushforwards-agree}.
    This completes our proof.
\end{proof}

The following observation was used above:
\begin{lem}\label{lem:HT_dR_stacks_commute}
    Let $X$ be a bounded $p\textrm{-adic}$ formal scheme.
    Then, the following diagram of formal stacks $2$-commutes up to unique isomorphism:
    \begin{equation*}
        \begin{tikzcd}
            X^{\textup{HT}} & X & X^{\dR} \\
            X^{\smallprism} && X^{\smallprism}.
            \arrow[from=1-1, to=1-2]
            \arrow[hook, from=1-1, to=2-1]
            \arrow[from=1-2, to=1-3]
            \arrow["i_{\dR, X}", from=1-3, to=2-3]
            \arrow["F_{\smallprism}", from=2-1, to=2-3]
        \end{tikzcd}
    \end{equation*}
\end{lem}
\begin{proof}
    By an argument similar to the proof of \cite[Proposition 3.6.6]{BhattLurieAbsolute}, it is easy to see that the following diagram of ring stacks commutes up to unique isomorphism:
    \begin{equation*}
        \begin{tikzcd}
            \mathbb{G}_a^{\textup{HT}} & \mathbb{G}_a & \mathbb{G}_a^{\dR} \\
            \mathbb{G}_a^{\smallprism} && \mathbb{G}_a^{\smallprism}.
            \arrow[from=1-1, to=1-2]
            \arrow[hook, from=1-1, to=2-1]
            \arrow[from=1-2, to=1-3]
            \arrow["i_{\dR, \mathbb{G}_a}", from=1-3, to=2-3]
            \arrow["F_{\smallprism}", from=2-1, to=2-3]
        \end{tikzcd}
    \end{equation*}
    Then, the claim follows by transmutation (cf.\ \cite[Remark 2.3.8]{BhattNotes}).
\end{proof}

Moreover, we observe that in the case of trivial coefficients (i.e., $\mc{V}=\mc{O}_\smallprism$) one can interpret the twisted de Rham cohomology as a more classical object:

\begin{prop}\label{prop:twisted-dR-crys-comparison}
    There exists a functorial identification,
    \begin{equation*}
        \mr{R}\Gamma^{\tw{n}}_\dR(X/\mc{O}_K)\simeq \mr{R}\Gamma_\crys(X_{p=0}^{\tw{n}}/\mc{O}_K).
    \end{equation*}
\end{prop}
\begin{proof} 
    Choose a uniformiser $\pi$ of $\mc{O}_K$.
    Then, by Corollary \ref{cor:twisted-dR-identification} we have an identification,
    \begin{equation*}
        \mr{R}\Gamma^{\tw{n}}_\dR(X/\mc{O}_K)\simeq \phi_\mf{S}^{\ast\tw{n+1}}\mr{R}\Gamma_\mf{S}(X/\mc{O}_K)/E_\pi.
    \end{equation*}
    Applying the crystal property to the morphism $(\mf{S}, E_\pi^{\tw{1}}) \to (\tilde{S}, p)$ in $(\mc{O}_K)_\smallprism$ from Proposition \ref{prop:big-diagram-commutes} and using the morphism $\tilde{S}\to \mc{O}_K$ from Proposition \ref{prop:BO-prism-to-O_K}, gives us the following identification:
    \begin{equation*}
        \phi_\mf{S}^{\ast\tw{n+1}}\mr{R}\Gamma_\mf{S}(X/\mc{O}_K)/E_\pi \simeq \phi_{\tilde{S}}^{\ast n}\RGamma_{\tilde{S}}(X/\mc{O}_K)\otimes_{\tilde{S}}\mc{O}_K.
    \end{equation*}
    But, we have a further identification,
    \begin{equation*}
        \phi_{\tilde{S}}^{\ast n}\RGamma_{\tilde{S}}(X/\mc{O}_K)\otimes_{\tilde{S}}\mc{O}_K \simeq \RGamma_\crys(X_{p=0}^{\tw{n}}/\tilde{S})\otimes_{\tilde{S}}\mc{O}_K.
    \end{equation*}
    Therefore, we have the identifications
    \begin{equation*}
        \RGamma_\crys\big(X_{p=0}^{\tw{n}}/\tilde{S}\big) \otimes_{\tilde{S}} \mc{O}_K \simeq \RGamma_\crys\big(X_{\pi^{\tilde{e}}=0}^{\tw{n}}/\tilde{S}\big) \otimes_{\tilde{S}} \mc{O}_K \simeq \RGamma_\crys(X_{\pi^{\tilde{e}}=0}^{\tw{n}}/\mc{O}_K) \simeq \RGamma_\crys(X_{p=0}^{\tw{n}}/\mc{O}_K),
    \end{equation*}
    where the final isomorphism follows from \stacks{07MS} as the map $(\tilde{S}, p\tilde{S}+u^{\tilde{e}}\tilde{S}) \to (\mc{O}_K, \pi^{\tilde{e}})$ is a morphism of PD-rings.
    Finally, putting everything together yields the following identification:
    \begin{equation*}
        \RGamma^{\tw{n}}_\dR(X/\mc{O}_K)\simeq \RGamma_\crys(X_{p=0}^{\tw{n}}/\mc{O}_K),
    \end{equation*}
    and using an argument similar to that in the proof of Proposition \ref{prop:crys-dR-comparison-pi-independence} shows that the preceding identification is independent of the choice of uniformiser $\pi$.
    This allows us to conclude.
\end{proof}

\begin{rem}\label{rem:twisted-dR-cohomology}
    Despite Proposition \ref{prop:twisted-dR-crys-comparison} and in contrast to twisted crystalline cohomology, it is not possible to define twisted de Rham cohomology for general coefficients in a non-prismatic way.
    It may be tempting to say, in analogy with Proposition \ref{prop:twisted-dR-crys-comparison}, that $\mr{R}\Gamma^{\tw{n}}_\dR(\mc{V})$ should be isomorphic to $\mr{R}\Gamma_\crys(X_{p=0}^{\tw{n}}/\mc{O}_K,((\mc{V}_{p=0})^\mr{crys})^{\tw{n}})$, where $(\mc{V}_{p=0})^\mr{crys}$ has the obvious meaning. However, the natural map $X_{p=0}^{\tw{n}}\to X_{p=0}$ does \underline{not} upgrade to a morphism of topoi,
    \begin{equation*}
        \begin{tikzcd}[sep=large]
        	{(X_{p=0}^{\tw{n}}/\mathcal{O}_K)_\mathrm{crys}} & {(X_{p=0}/\mathcal{O}_K)_\mathrm{crys},}
        	\arrow["{\color{red}{\times}}"{description}, dotted, from=1-1, to=1-2]
        \end{tikzcd}
    \end{equation*}
    and so it does not seem possible to define $((\mc{V}_{p=0})^\mr{crys})^{\tw{n}}$ in a purely crystalline way.
\end{rem}
\subsection{Berthelot--Ogus comparison isomorphism with coefficients}

We now explain how to obtain an analogue of the Berthelot--Ogus comparison isomorphism (see \cite[Corollary 2.5]{BerthelotOgusFIsocrystals}) with coefficients in perfect prismatic $F\textrm{-crystals}$.

To begin, we observe the following cohomological consequence of Theorem \ref{thm:twisted-crys-dR-comparison}.

\begin{thm}\label{thm:integral-BO}
    Let $\mc{V}$ be a perfect prismatic crystal on $X$.
    For $n \geqslant a = \lceil \log_p(\tfrac{e}{p-1})\rceil$, there exists a natural \emph{integral generalised Berthelot--Ogus isomorphism},
    \begin{equation*}
        \mr{R}\Gamma^{\tw{n}}_\crys(\mc{V})\otimes_W \mc{O}_K\isomto \mr{R}\Gamma^{\tw{n}}_\dR(\mc{V}).
    \end{equation*}
\end{thm}

To obtain a generalisation of the classical Berthelot--Ogus isomorphism with coefficients, we need to \emph{rationally untwist} this isomorphism.
The procedure to do this begins with the following construction.

\begin{construction} 
    Let $(\mc{V},\varphi_\mc{V})$ be an object of $\cat{Perf}^\varphi((\mc{O}_K)_\smallprism)$ and $i\geqslant 0$.
    Evaluating $\mc{V}$ on the prism $(\mf{S}, E_\pi^{\tw{i}})$ and applying the Frobenius morphism $\varphi_\mc{V}$ gives us the following isomorphism:
    \begin{equation}\label{eq:Frobenius-structure-deep-twist}
        \phi_\mf{S}^\ast\mc{V}(\mf{S}, E_\pi^{\tw{i}})[\nicefrac{1}{E_\pi^{\tw{i}}}] \isomto \mc{V}(\mf{S}, E_\pi^{\tw{i}})[\nicefrac{1}{E_\pi^{\tw{i}}}].
    \end{equation}
    Using the morphism $\phi^i_\mf{S}\colon (\mf{S}, E_\pi) \to (\mf{S}, E_\pi^{\tw{i}})$ in $(\mc{O}_K)_\smallprism$, we obtain that $\phi_\mf{S}^{\ast i}\mc{V}(\mf{S}, E_\pi)\simeq \mc{V}(\mf{S},E_\pi^{\tw{i}})$.
    Therefore, we may interpret \eqref{eq:Frobenius-structure-deep-twist} as an isomorphism,
    \begin{equation*}
        \phi^{\ast(i+1)}_{\mf{S}}\mc{V}(\mf{S}, E_\pi)[\nicefrac{1}{E_\pi^{\tw{i}}}] \isomto \phi_{\mf{S}}^{\ast i}\mc{V}(\mf{S}, E_\pi)[\nicefrac{1}{E_\pi^{\tw{i}}}].
    \end{equation*}
    Thus, iterating this procedure for any $n \geqslant 0$, yields the following isomorphism:
    \begin{equation*}
        \varphi^{\tw{n}}_{\mc{V},\pi} \colon \phi_\mf{S}^{\ast\tw{n+1}}\mc{V}(\mf{S}, E_\pi)[\nicefrac{1}{E_\pi^{\tw{n}}\cdots E_\pi^{\tw{1}}}] \isomto \phi_{\mf{S}}^\ast\mc{V}(\mf{S}, E_\pi)[\nicefrac{1}{E_\pi^{\tw{n}}\cdots E_\pi^{\tw{1}}}].
    \end{equation*}
\end{construction}

Our goal is to apply the morphism $\varphi^{\tw{n}}_{\mc{V},\pi}$ to rationally untwist the twists appearing in the twisted crystalline and de Rham realisations.
For this, we make the following easy observation.
\begin{lem} 
    For every $i \geqslant 1$:
    \begin{equation*}
        E_\pi^{\tw{i}} = p\cdot(\mr{unit}) \textup{ mod } E_\pi,\qquad E_\pi^{\tw{i}} = p\cdot(\mr{unit}) \textup{ mod } u.
    \end{equation*}
\end{lem}

Now, consider the following isomorphisms:
\begin{equation*}
    \overline{\varphi}^{\tw{n}}_{\crys,\mc{V},\pi,} \colon \bb{D}^{\prime,\tw{n}}_{\crys,\pi}(\mc{V})[\nicefrac{1}{p}] \isomto \bb{D}^{\prime}_{\crys,\pi}(\mc{V})[\nicefrac{1}{p}],\qquad \big(\text{resp. } \overline{\varphi}^{\tw{n}}_{\dR,\mc{V},\pi} \colon \bb{D}^{\prime,\tw{n}}_{\dR,\pi}(\mc{V})[\nicefrac{1}{p}] \isomto \bb{D}^\prime_{\dR,\pi}(\mc{V})[\nicefrac{1}{p}]\big),
\end{equation*}
in $\cat{Perf}(K_0)$ (resp.\@ $\cat{Perf}(K)$) induced by the reduction of $\varphi^{\tw{n}}_{\mc{V},\pi}$ modulo $u$ (resp.\@ $E_\pi$).
The preceding isomorphisms are functorial in $\mc{V}$, and so we obtain ``rational untwisting'' maps.
\begin{prop}\label{prop:rational_untwist}
    For any $n \geqslant 0$, there exists a natural equivalence in $\cat{Func}(\cat{Perf}^{\varphi}((\mc{O}_K)_\smallprism), \cat{Perf}(K_0))$ \textup{(}resp.\@ $\cat{Func}(\cat{Perf}^{\varphi}((\mc{O}_K)_\smallprism), \cat{Perf}(K))$\textup{)},
    \begin{equation*}
        \overline{\varphi}^{\tw{n}}_{\crys,\pi} \colon \bb{D}^{\prime,\tw{n}}_{\crys,\pi}[\nicefrac{1}{p}] \isomto \bb{D}^{\prime}_{\crys,\pi}[\nicefrac{1}{p}],\qquad \big(\text{resp. } \overline{\varphi}^{\tw{n}}_{\dR,\pi} \colon \bb{D}^{\prime,\tw{n}}_{\dR,\pi}[\nicefrac{1}{p}] \isomto \bb{D}^{\prime}_{\dR,\pi}[\nicefrac{1}{p}]\big).
    \end{equation*}
\end{prop}

The following is proven similarly to Proposition \ref{prop:crys-dR-comparison-pi-independence}.
\begin{prop}\label{prop:rational_untwist_pi_indep}
    For uniformisers $\pi$ and $\pi'$ of $\mc{O}_K$ and any integer $n\geqslant 0$, we have the following identification:
    \begin{align*}
        \j^{\tw{0}}_{\crys,\pi,\pi'}[\nicefrac{1}{p}] \circ \overline{\varphi}^{\tw{n}}_{\crys,\pi} &= \overline{\varphi}^{\tw{n}}_{\crys,\pi'} \circ \j^{\tw{n}}_{\crys,\pi,\pi'}[\nicefrac{1}{p}]\\
        \big(\textrm{resp. } \j^{\tw{0}}_{\dR,\pi,\pi'}[\nicefrac{1}{p}] \circ \overline{\varphi}^\mr{\tw{n}}_{\dR,\pi} &= \overline{\varphi}^\mr{\tw{n}}_{\dR,\pi'} \circ \j^{\tw{n}}_{\dR,\pi,\pi'}[\nicefrac{1}{p}]\big).
    \end{align*}
\end{prop}

Using Propositions \ref{prop:rational_untwist} and \ref{prop:rational_untwist_pi_indep}, the next definition follows naturally.
\begin{defn} 
    For any $n \geqslant 0$, we define the following equivalence in $\cat{Func}(\cat{Perf}^{\varphi}((\mc{O}_K)_\smallprism), \cat{Perf}(K_0))$ \textup{(}resp.\@ $\cat{Func}(\cat{Perf}^{\varphi}((\mc{O}_K)_\smallprism), \cat{Perf}(K))$\textup{)}:
    \begin{equation*}
        \overline{\varphi}^{\tw{n}}_\crys \defeq (\overline{\varphi}^{\tw{n}}_{\crys,\pi}) \colon \bb{D}^{\prime,\tw{n}}_\crys[\nicefrac{1}{p}] \isomto \bb{D}^\prime_\crys[\nicefrac{1}{p}],\quad \big(\text{resp. } \overline{\varphi}^{\tw{n}}_\dR \defeq (\overline{\varphi}^{\tw{n}}_{\dR,\pi}) \colon \bb{D}^{\prime,\tw{n}}_\dR[\nicefrac{1}{p}] \isomto \bb{D}^\prime_\dR[\nicefrac{1}{p}]\big).
    \end{equation*}
    Moreover, using Propositions \ref{prop:prime-to-nonprime-crys} and \ref{prop:prime-to-nonprime-deRham} we obtain the following isomorphisms
    \begin{equation*}
        \overline{\varphi}^{\tw{n}}_\crys \defeq (\overline{\varphi}^{\tw{n}}_{\crys,\pi}) \colon \bb{D}^{\tw{n}}_\crys[\nicefrac{1}{p}] \isomto \bb{D}_\crys[\nicefrac{1}{p}],\quad \big(\text{resp. } \overline{\varphi}^{\tw{n}}_\dR \defeq (\overline{\varphi}^{\tw{n}}_{\dR,\pi}) \colon \bb{D}^{\tw{n}}_\dR[\nicefrac{1}{p}] \isomto \bb{D}_\dR[\nicefrac{1}{p}]\big).
    \end{equation*}
We furthermore have natural factorizations
\begin{equation*}
\ov{\varphi}^{\tw{n}}_\mr{crys}=\ov{\varphi}_{\crys,n}\circ \ov{\varphi}_{\crys,n-1}\circ\cdots\circ \ov{\varphi}_{\crys,1},\qquad \ov{\varphi}^{\tw{n}}_\mr{dR}=\ov{\varphi}_{\mr{dR},n}\circ \ov{\varphi}_{\mr{dR},n-1}\circ\cdots\circ \ov{\varphi}_{\mr{dR},1}
\end{equation*}
where for every $m$, one defines $\ov{\varphi}_{\crys,m}\colon \bb{D}^{\tw{m}}_\mr{crys}[\nicefrac{1}{p}]\to \bb{D}_\crys^{\tw{m-1}}[\nicefrac{1}{p}]$ and $\ov{\varphi}_{\mr{dR},m}\colon \bb{D}^{\tw{m}}_\mr{dR}[\nicefrac{1}{p}]\to \bb{D}_\mr{dR}^{\tw{m-1}}[\nicefrac{1}{p}]$ in the directly analagous way to the definitions of $\ov{\varphi}_\crys^{\tw{n}}$ and $\ov{\varphi}_\mr{dR}^{\tw{n}}$.
\end{defn}

Finally, we have the following compatibility-with-$n$ condition for the isomorphisms $\iota^{\tw{n}}$, which follows by unravelling the definitions:

\begin{lem}\label{lem:n-indepdend}
    Let $(\mc{V},\varphi_\mc{V})$ be an object of $\cat{Perf}^{\varphi}((\mc{O}_K)_\smallprism)$.
    For any $n\geqslant a$, the following diagram commutes:
    \begin{equation*}
        \begin{tikzcd}[sep=2.25em]
        	{\bb{D}^{\tw{n+1}}_\crys(\mc{V}) \otimes_{W} K} && {\mathbb{D}_\dR^{\tw{n+1}}(\mc{V}) \otimes_{\mathcal{O}_K} K} \\
        	{\bb{D}^{\tw{n}}_\crys(\mc{V}) \otimes_{W} K} && {\mathbb{D}_\dR^{\tw{n}}(\mc{V}) \otimes_{\mathcal{O}_K} K}.
        	\arrow["{\iota^{\tw{n+1}}[\nicefrac{1}{p}]}", from=1-1, to=1-3]
        	\arrow["{\overline{\varphi}_{\crys,n+1}}"', from=1-1, to=2-1]
        	\arrow["{\overline{\varphi}_{\dR,n+1}}", from=1-3, to=2-3]
        	\arrow["{\iota^{\tw{n}}[\nicefrac{1}{p}]}"', from=2-1, to=2-3]
        \end{tikzcd}
    \end{equation*}
\end{lem}

Thus, we obtain the following Berthelot--Ogus-type isomorphism.
\begin{thm}\label{thm:BO_equiv_coeff}
    Let $(\mc{V},\varphi_\mc{V})$ be an object of $\cat{Perf}^{\varphi}((\mc{O}_K)_\smallprism)$.
    Then, we have a \emph{generalised Berthelot--Ogus isomorphism}
    \begin{equation*}
        \iota_{\mr{BO}} \colon \bb{D}_\crys(\mc{V}) \otimes_{W} K \isomto \bb{D}_\dR(\mc{V}) \otimes_{\mc{O}_K} K,
    \end{equation*}
    defined as 
    \begin{equation*}
        \iota_\mr{BO}=\ov{\varphi}^{\tw{n}}_\mr{dR}\circ \iota^{\tw{n}}[\nicefrac{1}{p}] \circ(\ov{\varphi}^{\tw{n}}_\crys)^{-1},
    \end{equation*}
    for any $n\geqslant a$, which is independent of $n$ by \emph{Lemma \ref{lem:n-indepdend}}.
\end{thm}

As an application of Theorem \ref{thm:BO_equiv_coeff}, and using Corollaries \ref{cor:twisted-crys-identification} and \ref{cor:twisted-dR-identification}, we recover the following analogue of the classical Berthelot--Ogus isomorphism for cohomology with coefficients.
\begin{cor}\label{cor:BO-wit-coefficients}
    Let $(\mc{V},\varphi_\mc{V})$ be an object of $\cat{Perf}^\varphi((X)_\smallprism)$.
    There exists a canonical isomorphism
    \begin{equation*}
        \iota_\mr{BO}\colon \mr{R}\Gamma_\crys(\mc{V}^\crys)\otimes_W K\isomto \mr{R}\Gamma_\dR(\mc{V}^\dR)\otimes_{\mc{O}_K} K.
    \end{equation*}
\end{cor}

\begin{rem}
    The isomorphism $\iota_\mr{BO}$ from Corollary \ref{cor:BO-wit-coefficients} agrees with the isomorphism from \cite{BerthelotOgusFIsocrystals} when $\mc{V}=\mc{O}^\smallprism$.
    Indeed, this follows ultimately from two observations.
    The first is that the isomorphisms $\overline{\varphi}^{\tw{n}}_\crys$ and $\overline{\varphi}^{\tw{n}}_\dR$ agree with Frobenius maps on crystalline cohomology coming from the identifications in Propositions \ref{prop:twisted-crys-classical-comparison} and \ref{prop:twisted-dR-crys-comparison}, respectively.
    The second is the observation that the argument for equivalency of structure maps in the proof of Theorem \ref{thm:twisted-crys-dR-comparison} precisely corresponds to the Dwork-like isomorphism as in Equation \eqref{eq:BO-dwork} from the introduction.
\end{rem}

\subsection{A conjectural framework to study torsion in crystalline and de Rham cohomology}\label{ss:conjectural-framework}

In this final subsection we state a conjecture concerning the relationship between torsion in crystalline and de Rham cohomology, and explain how Theorem \ref{thm:twisted-crys-dR-comparison} gives a framework for its study.
In Appendix \ref{app:torsion-calculations}, we shall verify one of these conjectures in a small number of small cases.

\medskip

\paragraph*{The main conjecture}

To state our conjecture, it is helpful to first set the following notation.
\begin{nota}
    For any $n\geqslant 0$, we set 
    \begin{equation*}
        \ell^{i,\tw{n}}_\crys(X/\mc{O}_K) \defeq \ell_W\left(\mr{H}^{i,\tw{n}}_\crys(X_k/W)[p^\infty]\right), \qquad \ell^{i,\tw{n}}_\dR(X/\mc{O}_K) \defeq \ell_{\mc{O}_K}\big(\mr{H}^{i,\tw{n}}_\dR(X/\mc{O}_K)[p^\infty]\big).
    \end{equation*}
    As $X/\mc{O}_K$ is fixed, we shall always omit it from the notation (i.e., just writing $\ell^{i,\tw{n}}_\crys$ and $\ell^{i,\tw{n}}_\dR)$.
    Moreover, we shorten $\ell^{i,\tw{0}}_\crys$ and $\ell^{i,\tw{0}}_\dR$ to $\ell^i_\crys$ and $\ell^i_\dR$, respectively.
\end{nota}

\begin{conj}[{Main Conjecture}]\label{conj:main-inequality}
    For all $i\geqslant 0$, we have the following inequalities:
    \begin{equation*}
        \ell^i_\crys \leqslant \ell^i_\dR \leqslant e\cdot \ell^i_\crys.
    \end{equation*}
\end{conj}

For clarity, let us write
\begin{equation*}
    \mr{H}^i_\crys(X_k/W)[p^\infty] \simeq \textstyle\bigoplus_{i=1}^r W/p^{a_i}, \qquad \mr{H}^i_\mr{dR}(X/\mc{O}_K)[p^\infty] \simeq \textstyle\bigoplus_{j=1}^s\mc{O}_K/\pi^{b_j}.
\end{equation*}
Then, $r=s$ by \cite[Remark 7.8]{CesnaviciusKoshikawa}, and thus Conjecture \ref{conj:main-inequality} can be understood as giving a rough relationship between the size of $a_i$ and the size of $b_j$, as $\ell^i_\mr{crys} = \sum_{i=1}^r a_i$ and $\ell^i_\mr{dR}=\sum_{j=1}^sb_j$. 

To help understand this conjecture, we now examine several non-trivial examples.
In particular, we discuss an example where the second inequality in Conjecture \ref{conj:main-inequality} is strict for which it is unclear whether this has been previously recorded in the literature (cf.\@ \cite[Question 7.13]{CesnaviciusKoshikawa}).

\begin{eg}\label{eg:main-inequality-low-ramification}
    Suppose $X$ admits a descent to $\mc{O}_{K'}$ for some subextension $K'$ of $K/K_0$ of ramification index at most $p-1$. Then, $e\cdot\ell^i_\crys=\ell^i_\dR$.
    Indeed, by base change we are reduced to the case $K=K'$.
    We have $\lceil\log_p(\tfrac{e}{p-1})\rceil=0$, and so the claim follows from Theorem \ref{thm:integral-BO} (or, more simply, from \cite[Corollary 7.4]{BerthelotOgus}).
\end{eg}

\begin{eg}[Li--Petrov]\label{eg:Li--Petrov}
    There exists a finite extension $K/K_0$ with $e=p^4-p^2$ and an elliptic scheme $E/\mc{O}_K$ such that $E_k$ is supersingular, and such that $E[p^2](K)$ contains a point $x$ of order $p^2$.
    For this $E$, let $H$ denote the integral closure in $E$ of the subgroup generated by $x$.
    Set $X = BH$ to be the classifying stack for $H$, a smooth proper formal $\mc{O}_K$-stack.
    Then, Li and Petrov compute that,
    \begin{align*}
        \mr{H}^2_\crys(X_k/W) &= k^2,\qquad \mr{H}^3_\crys(X_k/W) = k,\\
        \mr{H}^2_\dR(X/\mc{O}_K) &= \mc{O}_K/\pi^{2e-p^3+p^2}\oplus \mc{O}_K/\pi^{p^3-p^2},\qquad \mr{H}^3_\mr{dR}(X/\mc{O}_K) = \mc{O}_K/\pi^{p^3-p^2}.
    \end{align*}
    From this, we observe that,
    \begin{equation*}
        e\cdot \ell^2_\crys=2e,\quad \ell^2_\dR=(2e-p^3+p^2)+(p^3-p^2)=2e,
    \end{equation*}
    so that $\ell^2_\crys < \ell^2_\dR = e \cdot \ell^2_\crys$, although it is clear that,
    \begin{equation*}
        \mr{H}^2_\crys(X_k/W)\otimes_W\mc{O}_K \not\simeq \mr{H}^2_\dR(X/\mc{O}_K).
    \end{equation*}
    Furthermore, we observe that,
    \begin{equation*}
    e\cdot \ell^3_\crys=e=p^4-p^2,\qquad \ell^3_\dR=p^3-p^2,
    \end{equation*}
    so that $1=\ell^3_\crys <\ell^3_\dR<e\cdot \ell^3_\crys$.
    One may further replace $X$ by an actual formal scheme by a standard approximation technique (cf.\ \cite[Construction 6.12 and Proposition 6.13]{li-liu-uinfty}).
\end{eg}

\begin{eg}\label{eg:main-inequality-true-for-p-torsion}
    Suppose that $\mr{H}^i_\dR(X/\mc{O}_K)[p^\infty]$ is $p\textrm{-torsion}$.
    In this case, we have that $1 \leqslant b_j \leqslant e$ for all $j$, and $a_i \geqslant 1$ for all $i$.
    Thus, we see that,
    \begin{equation*}
        \ell_\dR^i = \textstyle\sum_{j=1}^s b_j \leqslant e \cdot s = e \cdot r \leqslant e \cdot \textstyle\sum_{i=1}^r a_i = e \cdot \ell^i_\crys.
    \end{equation*}
    Additionally, if $\mr{H}^i_{\crys}(X_k/W)[p^\infty]$ is also $p\textrm{-torsion}$, then $a_i = 1$ for all $i$, and we see that,
    \begin{equation*}
        \ell^i_\crys = \textstyle\sum_{i=1}^r a_i = r = s \leqslant \textstyle\sum_{j=1}^s b_j = \ell_\dR^i.
    \end{equation*}
\end{eg}

\begin{rem}
    One can use Example \ref{eg:main-inequality-true-for-p-torsion} to illustrate the subtlety of Conjecture \ref{conj:main-inequality}.
    Suppose that $\mr{H}^i_\crys(X_k/W)[p^\infty]$ and $\mr{H}^i_\dR(X/\mc{O}_K)[p^\infty]$ are both $p^2$-torsion, so $a_i \in \{1,2\}$ for all $i$.
    Write $r_1$ for the number of $i$ with $a_i=1$ and define $r_2$ analogously.
    Thus, $\ell^i_\crys=r_1+2r_2$. Our assumptions imply only that $b_j\leqslant 2e$ for all $j$.
    Then, the naive approach taken in Example \ref{eg:main-inequality-true-for-p-torsion} only implies that $\ell^i_\dR\leqslant 2er$, whereas $e \cdot \ell^i_\crys = er_1+2er_2$, and thus this is not helpful if $r_1\ne 0$.
\end{rem}

\medskip

\paragraph*{Questions about torsion in Breuil--Kisin cohomology}

We now explain how Theorem \ref{thm:integral-BO} allows one to reduce Conjecture \ref{conj:main-inequality} to questions about torsion in Breuil--Kisin cohomology.

We begin by making the following simple but crucial observation, coming from Proposition \ref{prop:twisted-crys-classical-comparison} and the fact that $\phi_W \colon W \isomto W$ is an isomorphism of rings.

\begin{prop}\label{prop:lcrys-constant}
    The equality $\ell^{i,\tw{n}}_\crys=\ell^i_\crys$ holds for all $n\geqslant 0$.
\end{prop}

Now, from Theorem \ref{thm:twisted-crys-dR-comparison} it follows that we have the equality,
\begin{equation}\label{eq:ldR-equals-lcrys-for-large-n} 
    \ell^{i,\tw{n}}_\dR = e\cdot\ell^{i,\tw{n}}_\crys,\qquad \text{for all }n\geqslant a.
\end{equation}
This indicates a refinement of Conjecture \ref{conj:main-inequality}, namely that the following two statements hold:
\begin{enumerate}
    \item[\textbf{(L1)}] $\ell^{i,\tw{a}}_\dR \leqslant e\cdot\ell^i_\dR$,
    \item[\textbf{(L2)}] $\ell^{i,\tw{n}}_\dR$ is a non-decreasing quantity in $n$.
\end{enumerate}

\begin{rem}
    One interesting and non-obvious consequence of Proposition \ref{prop:lcrys-constant} and \eqref{eq:ldR-equals-lcrys-for-large-n} is that the quantity $\ell^{i,\tw{n}}_\dR$ is constant for $n \geqslant a$.
\end{rem}

To make the preceding statements more tractable, we would like to instead formulate them in terms of the highly structured $u^\infty\textrm{-torsion}$ in Breuil--Kisin cohomology.
In the following, for a Breuil--Kisin module $\mathfrak{M}$ (see \cite[Section 4]{BMSI}), we shall write $\mathfrak{M}[u^{\infty}]$ for its $u^{\infty}\textrm{-torsion}$ submodule and $\overline{\mathfrak{M}} \defeq \mathfrak{M}^{\vee\vee}/\mathfrak{M}_{\operatorname{tf}}$ (see Section \ref{ss:prelims-on-BK-modules} for details).
Our main conjecture concerns the preceding modules arising from Breuil--Kisin cohomology.
Additionally, for a (complex of) $\mf{S}\textrm{-module(s)}$ $\mf{M}$ we shorten the notation $\phi^{\ast n}_\mf{S}\mf{M}$ to $\mf{M}^{\tw{n}}$.

\begin{conj}\label{conj:beta}
    Set $\mf{M}^{i}\defeq \mr{H}^i_\mf{S}(X/\mc{O}_K)$ and
    let $f(n)$ denote either $\ell_{\mc{O}_K}(\mf{M}^i[u^\infty]^{\tw{n+1}}[E])$ or $\ell_{\mc{O}_K}(\ov{\mf{M}^i}^{\tw{n+1}}[E])$.
    Then, $f$ satisfies the following:
    \begin{enumerate}
        \item[\emph{(1)}] $f(a) \leqslant e\cdot f(0)$,
        \item[\emph{(2)}] $f(n) \leqslant f(n+1)$, for all $n$.
    \end{enumerate}
\end{conj}

\begin{rem}
    A neater conjecture would be that if we set $\mf{M}^i_m \defeq \mr{H}^i\big(\mr{R}\Gamma_\mf{S}(X/\mc{O}_K)/p^m)$ and $f_m(n) \defeq \ell_{\mc{O}_K}\big(\mf{M}^i_m[u^\infty]^{\tw{n+1}}[E])$, then $f_m(n) \leqslant f_m(n+1)$ for every $m$.
    The precise relationship between this statement and Conjecture \ref{conj:beta} is unclear, but it seems somewhat likely that the former is stronger than the latter.
    Proposition \ref{prop:len_Mbar} and its proof may be seen as evidence for this.
\end{rem}

Conjecture \ref{conj:beta} is not sufficient to obtain \textbf{(L1)} and \textbf{(L2)}, and so a further hypothesis is required. To state this, observe that as $\mr{H}^{i,\tw{n}}_\dR(X/\mc{O}_K)$ is defined as the $i^\text{th}$ cohomology of $\mr{R}\Gamma_\mf{S}(X/\mc{O}_K)^{\tw{n}}/E$, one may apply \stacks{061Z} to obtain the exact sequence,
\begin{equation}\label{eq:deRham_BK_ses}
    0 \longrightarrow \mr{H}^i_{\mathfrak{S}}(X/\mc{O}_K)^{\tw{n+1}}/E \longrightarrow \mr{H}^{i, \tw{n}}_{\textup{dR}}(X/\mc{O}_K) \longrightarrow\mr{H}^{i+1}_{\mathfrak{S}}(X/\mc{O}_K)^{\tw{n+1}}[E] \longrightarrow 0,
\end{equation}
which may fail to be exact on the right when one passes to $p^\infty\textrm{-torsion}$ submodules.

\begin{hypo}\label{conj:gamma}
    Set $\mf{Q}^{i,\tw{n}}$ to be the image of $\mr{H}^{i,\tw{n}}_\dR(X/\mc{O}_K)[p^\infty]$ in $\mr{H}^{i+1}_\mf{S}(X/\mc{O}_K)^{\tw{n+1}}[E]$.
    Then,
    \begin{equation*}
        e \cdot \ell_{\mc{O}_K}\big(\mf{Q}^{i,(0)}\big) \geqslant \ell_{\mc{O}_K}\big(\mf{Q}^{i,(a)}\big),\quad \text{and}\quad \ell_{\mc{O}_K}\big(\mf{Q}^{i,\tw{n+1}}\big) \geqslant \ell_{\mc{O}_K}\big(\mf{Q}^{i,\tw{n}}\big)\text{ for all }n\geqslant 0.   
    \end{equation*}
\end{hypo}

\begin{rem}
    We have chosen not to state Hypothesis \ref{conj:gamma} as a conjecture due to a lack of substantial evidence.
    That said, it is reasonable to believe that it holds in all cases.
    In fact, the authors suspect that $\mf{Q}^{i,\tw{n}}=\mr{H}^{i+1}_\mf{S}(X/\mc{O}_K)^{\tw{n+1}}[E]$ for all sufficiently large $n$, which should allow one to reduce the verification of Hypothesis \ref{conj:gamma} to the verification of Conjecture \ref{conj:beta}.
\end{rem}
 
The precise relationship between Conjecture \ref{conj:main-inequality} and Conjecture \ref{conj:beta} can now be made precise.
The proof of the following proposition makes use of several routine calculations concerning Breuil--Kisin modules done in Appendix \ref{app:torsion-calculations} and so we postpone its proof until Section \ref{sss:proof-of-implication}.

\begin{prop}\label{thm:beta-and-gamma-imply-alpha}
    Under \emph{Hypothesis \ref{conj:gamma}}, \emph{Conjecture \ref{conj:beta}} implies \emph{Conjecture \ref{conj:main-inequality}}.
\end{prop}

As remarked in the introduction and explained via Proposition \ref{thm:beta-and-gamma-imply-alpha}, we hope that Conjecture \ref{conj:beta} will be useful in proving Conjecture \ref{conj:main-inequality} as the $u^\infty\textrm{-torsion}$ in Breuil--Kisin cohomology has incredibly rich structure (e.g.\ see \cite{li-liu-uinfty} and \cite{GabberLi}).
To illustrate, we exploit this structure to prove Conjecture \ref{conj:beta} in Appendix \ref{app:torsion-calculations} in a small number of non-trivial cases (see Theorem \ref{thm:conj-beta-small-index-and-ramification}).

\begin{rem}
    When $\mc{O}_K = W[\zeta_{p^m}]$, it is also reasonable to use the $q\textrm{-de Rham}$ prism and its variants (see  \cite[Section 16]{BhattScholzePrisms}, \cite[Section 2.6]{BhattLurieAbsolute} and \cite[Section 2]{abhinandan-prismatic-wach}) instead of the Breuil--Kisin prism, in order to better understand torsion in de Rham and crystalline cohomologies.
    This would require a robust understanding of the structure of torsion in $q\textrm{-de Rham}$ cohomology, analogous to the $u^{\infty}\textrm{-torsion}$ in Breuil--Kisin cohomology.
    We intend to explore these ideas and their consequences in future work.
\end{rem}

We end by giving an interesting non-trivial example where Conjecture \ref{conj:beta} holds.

\begin{eg}\label{eg:BK}
    Let $G\to H$ be a map of finite flat $\mc{O}_K\textrm{-group}$ schemes which is generically an isomorphism, and let $A/\mc{O}_K$ be an abelian scheme with $H\subseteq A$ (see \cite[Th\'{e}or\`{e}me 3.1.1]{BBMDieuII}). 
    
    Write $\mf{M}(\bullet)$ for the Breuil--Kisin Dieudonn\'e functor of Kisin (see \cite[Theorem 2.3.5]{KisinFCrystal}). A forthcoming result of Kubrak--Li--Mondal implies that setting $X=[A/G]$, a smooth proper formal $\mc{O}_K\textrm{-stack}$, and $A'=A/H$, an abelian scheme over $\mc{O}_K$, then we have that
    \begin{align*}
        \mr{H}^i_\mf{S}(X/\mc{O}_K) &\simeq \mr{H}^i_\mf{S}(X/\mc{O}_K)[u^\infty] \oplus \mr{H}^i_\mf{S}(A'/\mc{O}_K), \quad\textrm{for } i \geqslant 0,\\
        \mr{H}^1_\mf{S}(X/\mc{O}_K)_\mr{tor} &= 0 \quad\text{and}\quad \mr{H}^2_\mf{S}(X/\mc{O}_K)_\mr{tor} \simeq \mr{coker}\left(\mf{M}(H)\to\mf{M }(G)\right) \eqdef \mf{M}.
    \end{align*}
    So, now assume that $K$ contains $K_0[\zeta_{p^m}]$, and consider the map $\underline{\mathbb{Z}/p^m} \to \mu_{p^m}$ sending $1$ to $\zeta_{p^m} \in \mu_{p^m}(\mc{O}_K)$, which clearly satisfies the  desired conditions.
    In this case, we have that $\mf{M}\simeq k$ for which Conjecture \ref{conj:beta} clearly holds.
\end{eg}

\appendix

\section{Some calculations involving Breuil--Kisin modules} \label{app:torsion-calculations}

We continue to use the notation from \globalnotationref, Notation \ref{nota:Witt-vector-stuff} and that discussed in Section \ref{subsec:breuil_kisin_prism}.
In addition, we often make use of the following pieces of notation.

\begin{nota}
    For a $\mf{S}\textrm{-module}$ $\mf{M}$ and integer $n\geqslant 0$, we write $\mf{M}^{\tw{n}}$ as a shorthand for $\mf{M}\otimes_{\mf{S},\phi^n_\mf{S}}\mf{S}$.
    For any subscript $?$, one should interpret $\mf{M}^{\tw{n}}_?$ as $(\mf{M}_?)^{\tw{n}}$. Similarly, for any superscript $?$ we shorten $(\mf{M}^?)^{\tw{n}}$ to $\mf{M}^{?,\tw{n}}$.
    Finally, as we use it so frequently, we also note the notation from Proposition \ref{prop:BKmod_struct} below.
\end{nota}

\subsection{Preliminaries on Breuil--Kisin modules}\label{ss:prelims-on-BK-modules}

In this section, we will collect some basic results on Breuil--Kisin modules studied in \cite{BMSI, Li-Liu-derived-deRham}.

\subsubsection{Generalised Breuil--Kisin modules}

Recall that an \textit{effective} $\varphi\textit{-module}$ over $\mathfrak{S}$ is a module $\mathfrak{M}$ equipped with an $\mf{S}\textrm{-linear}$ map $\varphi_{\mathfrak{M}} \colon \mathfrak{M}^{\tw{1}} \rightarrow \mathfrak{M}$.

\begin{defn}
    A \textit{generalised Breuil--Kisin module of height} $i \geqslant 0$ is a finite type effective $\varphi\textrm{-module}$ over $\mathfrak{S}$ such that there exists an $\mathfrak{S}\textrm{-linear}$ map $\psi_{\mathfrak{M}} : \mathfrak{M} \rightarrow \mathfrak{M}^{\tw{1}}$ satisfying the following:
    \begin{enumerate}
        \item[(1)] $\psi_{\mathfrak{M}} \circ\varphi_{\mathfrak{M}} = E^i \cdot \textup{id}_{\mf{M}^{\tw{1}}}$,

        \item[(2)] $\varphi_{\mathfrak{M}} \circ \psi_{\mathfrak{M}} = E^i \cdot \textup{id}_{\mathfrak{M}}$.
    \end{enumerate}
    Morphisms between generalised Breuil--Kisin modules are $\mf{S}$-linear, Frobenius equivariant maps.
\end{defn}

\begin{prop}[{\cite[Proposition 4.3]{BMSI}}]\label{prop:BKmod_struct}
    Let $\mathfrak{M}$ be a generalised Breuil--Kisin module.
    Then, there exist canonical short exact sequences of generalised Breuil--Kisin modules:
    \begin{align}
        0 &\longrightarrow \mathfrak{M}_{\operatorname{tor}} \longrightarrow \mathfrak{M} \longrightarrow \mathfrak{M}_{\operatorname{tf}} \longrightarrow 0 \label{eq:ses_Mtortf}\\
        0 &\longrightarrow \mathfrak{M}_{\operatorname{tf}} \longrightarrow \mathfrak{M}_{\textup{free}} \longrightarrow \overline{\mathfrak{M}} \longrightarrow 0,\label{eq:ses_Mtffreebar}
    \end{align}
    where $\mathfrak{M}_{\operatorname{tor}} \subset \mathfrak{M}$ is the torsion submodule and is killed by a power of $p$, the module $\mathfrak{M}_{\operatorname{tf}}$ is torsion-free, the module $\mathfrak{M}_{\textup{free}}$ is free over $\mathfrak{S}$ and given as the reflexive hull of $\mathfrak{M}_{\operatorname{tf}}$ over $\mathfrak{S}$, and the module $\overline{\mathfrak{M}}$ is a torsion $\mathfrak{S}\textrm{-module}$ and killed by a power of $(p, u)$.
\end{prop}

Let $\mathfrak{M}[u^{\infty}]$ denote the $u\textrm{-power}$ torsion submodule of $\mathfrak{M}$, a generalised Breuil--Kisin module of height $i$, and observe that we have,
\begin{equation*}
    \varphi_{\mathfrak{M}}\left(\mathfrak{M}[u^{\infty}]^{\tw{1}}\right) \subset \mathfrak{M}[u^{\infty}], \quad\mr{and}\quad \psi_{\mathfrak{M}}(\mathfrak{M}[u^{\infty}]) \subset \mathfrak{M}[u^{\infty}]^{\tw{1}}.
\end{equation*}

\begin{lem}[{\cite[Lemma 6.2]{Li-Liu-derived-deRham}}]\label{lem:ses_uinfty_etale}
    The following is a short exact sequence in $\textup{Mod}_{\mathfrak{S}}^{\varphi, i}$:
    \begin{equation}\label{eq:ses_uinfty_etale}
        0 \longrightarrow \mathfrak{M}[u^{\infty}] \longrightarrow \mathfrak{M} \longrightarrow \mathfrak{M}/\mathfrak{M}[u^{\infty}] \longrightarrow 0.
    \end{equation}
\end{lem}

\begin{lem}[{\cite[Lemma 6.3]{Li-Liu-derived-deRham}}]\label{lem:toretale_BKmod}
    Let $\mathfrak{M}$ be a generalised Breuil--Kisin module of height $i$ and assume that it is killed by $p^n$ for some $n \geqslant 0$.
    Then, the following are equivalent:
    \begin{enumerate}
        \item[\textup{(1)}] The module $\mathfrak{M}$ is $u\textrm{-torsion}$ free,

        \item[\textup{(2)}] $\mathfrak{M} = \mathfrak{N}/\mathfrak{N}'$, where $\mathfrak{N}' \subset \mathfrak{N}$ and both are finite free Breuil--Kisin modules of height $i$,

        \item[\textup{(3)}] the module $\mathfrak{M}$ may be written as a successive extension of finite free $k\llbracket u \rrbracket\textrm{-modules}$ $\mathfrak{M}_j$ such that each $\mathfrak{M}_j$ is a generalised Breuil--Kisin module  of height $i$.
    \end{enumerate}
\end{lem}

Finally, we note a structural result on generalised Breuil--Kisin modules modulo powers of $p$.
\begin{lem}[{\cite[Lemma 4.6]{GabberLi}}]\label{lem:Mbar_MmodpN}
    Let $\mathfrak{M}$ be a generalised Breuil--Kisin module of height $i$ and $n \geqslant 0$ be an integer such that $p^n$ kills both $\mathfrak{M}_{\operatorname{tor}}$ and $\overline{\mathfrak{M}}$.
    Then, for $N \geqslant 2n$, the sequence
    \begin{equation}
        0 \longrightarrow \mathfrak{M}_{\operatorname{tor}} \oplus \overline{\mathfrak{M}} \longrightarrow \mathfrak{M}/p^N\mathfrak{M} \longrightarrow \mathfrak{M}_{\textup{free}}/p^N\mathfrak{M}_{\textup{free}} \longrightarrow \overline{\mathfrak{M}} \longrightarrow 0,
    \end{equation}
   is exact. In particular, we have an isomorphism of $\mathfrak{S}\textrm{-modules}$ $(\mathfrak{M}/p^N\mathfrak{M})[u^{\infty}] \isomorphic \mathfrak{M}[u^{\infty}] \oplus \overline{\mathfrak{M}}$.
\end{lem}

\subsubsection{Structure of the \texorpdfstring{$u^{\infty}\textrm{-torsion}$ in Breuil--Kisin cohomology}{-}}

\begin{defn}[{\cite[Definition 3.1]{li-liu-uinfty}}]\label{defi:quasi_filtered_BK_mod}
   A \emph{quasi-filtered Breuil--Kisin module of height $i \geqslant 1$} consists of the following set of data:
   \begin{itemize}
       \item Two $\mathfrak{S}\textrm{-modules}$ $\mathfrak{M}$ and $\mathfrak{N}$;

       \item Four $\mathfrak{S}\textrm{-linear}$ maps $f \colon \mathfrak{M}^{\tw{1}} \rightarrow \mathfrak{N}$, $g \colon \mathfrak{N} \rightarrow \mathfrak{M}^{\tw{1}}$, $h \colon \mathfrak{N} \rightarrow \mathfrak{M}$ and $h' \colon \mathfrak{M} \rightarrow \mathfrak{N}$;
   \end{itemize}
   satisfying the following conditions:
   \begin{enumerate}
       \item The map $h$ is injective.
       
       \item We have $g \circ f = E^{i-1} \cdot \textup{id}_{\mathfrak{M}^{\tw{1}}}$, $f \circ g = E^{i-1} \cdot \textup{id}_{\mathfrak{N}}$, $h' \circ h = E \cdot \textup{id}_{\mathfrak{N}}$ and $h \circ h' = E \cdot \textup{id}_{\mathfrak{M}}$.
   \end{enumerate}
\end{defn}

Note that in Definition \ref{defi:quasi_filtered_BK_mod}, $\mf{M}$ is a Breuil--Kisin module of height $i$ with $\varphi_\mf{M}\defeq h\circ f$ and $\psi_\mf{M}=g\circ h'$.
When the extra structure of a quasi-filtered Breuil--Kisin module beyond this is unimportant, we abuse notation and just refer to $\mf{M}$ as the quasi-filtered Breuil--Kisin module.

\begin{prop}[{\cite[Proposition 3.3 \& Corollary 3.4]{li-liu-uinfty}}]\label{prop:ann_quasifilBKmod}
    Let $\mathfrak{M}$ be a quasi-filtered Breuil--Kisin module of height $i$.
    Then, the annihilator ideal of $\mathfrak{M}$ admits the following restriction:
    \begin{equation}
        E^{i-1} \cdot \textup{Ann}(\mathfrak{M}) \subset \textup{Ann}\left(\mathfrak{M}^{\tw{1}}\right) =\textup{Ann}(\mathfrak{M})^{\tw{1}}.
    \end{equation}
    Moreover, if there exists an integer $\alpha \geqslant 0$ such that $\textup{Ann}(\mathfrak{M}) + (p) = (p, u^{\alpha})$, then $\alpha \leqslant \lfloor\tfrac{e(i-1)}{p-1}\rfloor$.
\end{prop}

In what follows, we denote the quantity $\alpha$ appearing in Proposition \ref{prop:ann_quasifilBKmod} by $\alpha(\mf{M})$.

\begin{thm}[{\cite[Theorem 3.6]{li-liu-uinfty} and \cite[Theorem 4.3]{GabberLi}}]\label{thm:ann_uinfty_submod}
    Let $\mf{M}$ be a $u^\infty$-torsion quasi-filtered Breuil--Kisin module.
    Then, we have the following:
    \begin{enumerate}
        \item[\textup{(1)}] if $e(i-1) < p-1$, then $\mathfrak{M} = 0$,

        \item[\textup{(2)}] if $e(i-1) = p-1$, then $\textup{Ann}(\mathfrak{M}) \supseteq (p,u)$,

        \item[\textup{(3)}] if $e(i-1) < 2(p-1)$, then $\textup{Ann}(\mathfrak{M}) + (u) \supseteq (p^{i-1}, u)$,

        \item[\textup{(4)}] if $i \leqslant 2$ and $e < p(p-1)$, then $\textup{Ann}(\mathfrak{M}) + (u) \supseteq (p,u)$.
    \end{enumerate}
\end{thm}

Using this, we are able to show that for small height and ramification, the annihilator of $\mf{M}$ must contain an element of a particularly simple form.

\begin{prop}\label{prop:ann_uinfty_submod}
    Let $\mf{M}$ be a nonzero $u^\infty$-torsion quasi-filtered Breuil--Kisin module of height $i \leqslant 2$, and assume $e < p(p-1)$ and write $\alpha=\alpha(\mf{M})$.
    Then, $1 \leqslant \alpha\leqslant \lfloor\tfrac{e}{p-1}\rfloor <p$, and either
    \begin{itemize}
        \item $\mathfrak{M}$ is $p\textrm{-torsion}$,

        \item or $\textup{Ann}(\mathfrak{M})$ contains an element of the form $u^{\alpha} + px$ with $x$ in $\mathfrak{S}^{\times}$.
    \end{itemize}
\end{prop}
\begin{proof}
    For $i = 0, 1$, Theorem \ref{thm:ann_uinfty_submod} (1) implies that $\mathfrak{M}$ must be zero.
    So, we may assume that $i = 2$.
    Moreover, the inequality concerning $\alpha$ follows from Proposition \ref{prop:ann_quasifilBKmod}.
    So, we focus only on the last claim.
    
    From Theorem \ref{thm:ann_uinfty_submod} and the definition of $\alpha$ we have that,
    \begin{equation*}
        \textup{Ann}(\mathfrak{M}) + (p) = (p, u^{\alpha}),\quad \text{and}\quad \textup{Ann}(\mathfrak{M}) + (u) = (p, u).
    \end{equation*}
    So, let us write $f = u^{\alpha} + pb$ and $g = p + uc$, for some $f$ and $g$ in $\textup{Ann}(\mathfrak{M}[u^{\infty}])$ with $b$ and $c$ in $\mathfrak{S}$.
    Moreover, we may write $c = \sum _{i \geqslant 0} c_i u^i$, with each $c_i$ in $W$.
    If for each $i \geqslant 0$, we have that $c_i = pd_i$ with $d_i$ in $W$, then observe that $g = p (1+ u\sum_{i \geqslant 0} d_i u^i)$, in particular, $p$ is in $\textup{Ann}(\mathfrak{M})$.
    Otherwise, we may assume that there exists a minimal $r \geqslant 1$ such that $c_{r-1}$ is in $W^{\times}$.
    
    Let us write $c_i = pd_i$ with $d_i$ in $W$ and $0 \leqslant i \leqslant r-2$, and observe that 
    \begin{equation}\label{eq:elem_in_ann}
        g = p(1+d_0 u + \cdots + d_{r-2}u^{r-1}) + u^r\textstyle\sum _{i \geqslant 0} c_{i+r-1} u^i \in \textup{Ann}(\mathfrak{M}).
    \end{equation}
    If $r \geqslant \alpha+1$, then substituting $u^{\alpha} = f-pb$ in \eqref{eq:elem_in_ann}, we obtain that
    \begin{equation*}
        p(1+d_0 u + \cdots + d_{r-2}u^{r-1}) - pbu^{r-{\alpha}}\textstyle\sum _{i \geqslant 0} c_{i+r-1} u^i \in \textup{Ann}(\mathfrak{M}).
    \end{equation*}
    In particular, if $r \geqslant \alpha+1$, then it follows that $p$ is in $\textup{Ann}(\mathfrak{M})$.
    On the other hand, if $1 \leqslant r \leqslant \alpha$, then using \eqref{eq:elem_in_ann}, it easily follows that $g = pd+u^ra'$, with $d$ and $a'$ in $\mathfrak{S}^{\times}$, is in $\textup{Ann}(\mathfrak{M})$.
    As $\alpha$ is minimal, it follows that $1 \leqslant r = \alpha$.
    This allows us to conclude.
\end{proof}

We end this section with the following simple observation.
\begin{lem}\label{lem:EnotinAnn}
    Let $\mathfrak{N}$ be a finite $\mathfrak{S}\textrm{-module}$.
    Assume that there exists $f = u^r + px$ in $\textup{Ann}(\mathfrak{N})$ with $x$ in $\mathfrak{S}^{\times}$.
    If $E$ is in $\textup{Ann}(\mathfrak{N})$ and $e \neq r$, then $\mathfrak{N}$ is killed by $p$.
\end{lem}

\subsubsection{\texorpdfstring{$E\textrm{-torsion}$}{-} and reduction modulo \texorpdfstring{$E$}{-} of Frobenius twists}

If $\mathfrak{M}$ is a generalised Breuil--Kisin module and $n \geqslant 0$, then \eqref{eq:ses_Mtortf} and \eqref{eq:ses_Mtffreebar}, give exact sequences:
\begin{align}
    0 &\longrightarrow \mathfrak{M}_{\operatorname{tor}}^{\tw{n}} \longrightarrow \mathfrak{M}^{\tw{n}} \longrightarrow \mathfrak{M}_{\operatorname{tf}}^{\tw{n}} \longrightarrow 0 \label{eq:ses_phiMtortf}\\
    0 &\longrightarrow \mathfrak{M}_{\operatorname{tf}}^{\tw{n}} \longrightarrow \mathfrak{M}_{\textup{free}}^{\tw{n}} \longrightarrow  \overline{\mathfrak{M}}^{\tw{n}} \longrightarrow 0,\label{eq:ses_phiMtffreebar}
\end{align}
Reducing \eqref{eq:ses_phiMtortf} modulo $E$, yields the following long exact sequence:
\begin{equation}\label{eq:les_phiMtortf}
    \begin{aligned}
        0 &\longrightarrow \mathfrak{M}_{\operatorname{tor}}^{\tw{n}}[E] \longrightarrow \mathfrak{M}^{\tw{n}}[E] \longrightarrow \mathfrak{M}_{\operatorname{tf}}^{\tw{n}}[E] \longrightarrow \\
        &\qquad \longrightarrow \mathfrak{M}_{\operatorname{tor}}^{\tw{n}}/E \longrightarrow \mathfrak{M}^{\tw{n}}/E \longrightarrow \mathfrak{M}_{\operatorname{tf}}^{\tw{n}}/E \longrightarrow 0.
    \end{aligned}
\end{equation}

\begin{lem}\label{lem:Mtors_Etors}
    Equation \eqref{eq:les_phiMtortf} yields the following isomorphism of finite torsion $\mc{O}_K\textrm{-modules}$:
    \begin{equation}\label{eq:Mtors_Etors}
        \mathfrak{M}_{\operatorname{tor}}^{\tw{n}}[E] \isomorphic \mathfrak{M}^{\tw{n}}[E] = (\mathfrak{M}^{\tw{n}}[E])[p^{\infty}].
    \end{equation}
\end{lem}
\begin{proof}
    In the long exact sequence \eqref{eq:les_phiMtortf}, let us note that $\mathfrak{M}_{\operatorname{tf}}^{\tw{n}}[E] = 0$.
    As $\mathfrak{M}_{\operatorname{tor}}^{\tw{n}}$ is a $p\textrm{-power}$ torsion $\mathfrak{S}\textrm{-module}$ (because $\mf{M}_\mr{tor}$ is) the isomorphism in \eqref{eq:Mtors_Etors} follows.
\end{proof}

\begin{lem}\label{lem:ses_phiMtortf_modE}
    The following sequence of finite $\mc{O}_K\textrm{-modules}$ is exact:
    \begin{equation}\label{eq:ses_phiMtortf_modE}
        0 \longrightarrow \mathfrak{M}_{\operatorname{tor}}^{\tw{n}}/E \longrightarrow \mathfrak{M}^{\tw{n}}/E \longrightarrow \mathfrak{M}_{\operatorname{tf}}^{\tw{n}}/E \longrightarrow 0.
    \end{equation}
    Additionally, we have the following short exact sequence of finite torsion $\mc{O}_K\textrm{-modules}$:
    \begin{equation}\label{eq:ses_phiMtortf_modE_pinfty}
        0 \longrightarrow \mathfrak{M}_{\operatorname{tor}}^{\tw{n}}/E\longrightarrow (\mathfrak{M}^{\tw{n}}/E)[p^{\infty}] \longrightarrow (\mathfrak{M}_{\operatorname{tf}}^{\tw{n}}/E)[p^{\infty}] \longrightarrow 0.
    \end{equation}
\end{lem}
\begin{proof}
    By combining \eqref{eq:les_phiMtortf} and Lemma \ref{lem:Mtors_Etors}, we obtain the short exact sequence in \eqref{eq:ses_phiMtortf_modE}.
    Next, in the sequence \eqref{eq:ses_phiMtortf_modE_pinfty}, the only non-obvious part is the surjectivity of the following map:
    \begin{equation*}
        (\mathfrak{M}^{\tw{n}}/E)[p^{\infty}] \longrightarrow (\mathfrak{M}_{\operatorname{tf}}^{\tw{n}}/E)[p^{\infty}].
    \end{equation*}
    But, this easily follows from the fact that $\mf{M}^{\tw{n}}_\mr{tor}$, and thus $\mf{M}^{\tw{n}}_\mr{tor}/E$, is $p^\infty$-torsion.
\end{proof}

Next, let us note that we have $\mathfrak{M}[u^{\infty}] = \mathfrak{M}_{\operatorname{tor}}[u^{\infty}]$ and consider the following short exact sequence of $p^\infty$-torsion $\mathfrak{S}\textrm{-modules}$:
\begin{equation}\label{eq:ses_uinfty_toretale}
        0 \longrightarrow \mathfrak{M}[u^{\infty}] \longrightarrow \mathfrak{M}_{\operatorname{tor}} \longrightarrow \mathfrak{M}_{\operatorname{tor}}/\mathfrak{M}[u^{\infty}] \longrightarrow 0.
\end{equation} 
Then, $\mathfrak{M}_{\operatorname{tor,} u\operatorname{-tf}} := \mathfrak{M}_{\operatorname{tor}}/\mathfrak{M}[u^{\infty}]$ is $u\textrm{-torsion}$ free.

\begin{lem}\label{lem:toretale_Etors}
    For any $n \geqslant 0$, we have that $\mathfrak{M}_{\operatorname{tor,} u\operatorname{-tf}}^{\tw{n}}[E] = 0$. 
\end{lem}
\begin{proof}
    Note that $\mathfrak{M}_{\operatorname{tor,} u\operatorname{-tf}}$ is a $p^\infty$-torsion and $u\textrm{-torsion}$-free Breuil--Kisin module.
    So, from Lemma \ref{lem:toretale_BKmod}, there exist finite free Breuil--Kisin modules $\mathfrak{N}$ and $\mathfrak{N}'$ with $\mathfrak{N}' \subset \mathfrak{N}$ and $\mf{M}_{\mr{tor},u\mr{-tf}} \isomorphic \mf{N}/\mf{N}'$.
    As $\phi_\mf{S}$ is flat, we  may twist by $\phi_\mf{S}$ and reduce modulo $E$ to obtain the exact sequence,
    \begin{equation*}
        0 \longrightarrow \mathfrak{M}_{\operatorname{tor,} u\operatorname{-tf}}^{\tw{n}}[E] \longrightarrow \mathfrak{N}'^{\tw{n}}/E \longrightarrow  \mathfrak{N}^{\tw{n}}/E.
    \end{equation*}
    As $\mathfrak{N}'^{\tw{n}}/E$ is a free $\mc{O}_K\textrm{-module}$, the claim follows.
\end{proof}

Now, Frobenius twisting the short exact sequence \eqref{eq:ses_uinfty_toretale} and reducing modulo $E$ yields the following long exact sequence of $\mathfrak{S}\textrm{-modules}$:
\begin{equation}\label{eq:les_phiMuniftytoret}
    \begin{aligned}
        0 &\longrightarrow \mathfrak{M}[u^{\infty}]^{\tw{n}}[E] \longrightarrow \mathfrak{M}_{\operatorname{tor}}^{\tw{n}}[E] \longrightarrow \mathfrak{M}_{\operatorname{tor,} u\operatorname{-tf}}^{\tw{n}}[E] \longrightarrow \\
        &\qquad \longrightarrow \mathfrak{M}[u^{\infty}]^{\tw{n}}/E \longrightarrow \mathfrak{M}_{\operatorname{tor}}^{\tw{n}}/E\longrightarrow \mathfrak{M}_{\operatorname{tor,} u\operatorname{-tf}}^{\tw{n}}/E \longrightarrow 0.
    \end{aligned}
\end{equation}

By combining Equation \eqref{eq:les_phiMuniftytoret} and Lemma \ref{lem:toretale_Etors}, we obtain the following:
\begin{lem}\label{lem:Munifty_Etors}
    For a generalised Breuil--Kisin module $\mf{M}$, there is an isomorphism of $\mc{O}_K\textrm{-modules}$:
    \begin{equation}\label{eq:Munifty_Etors}
        \mathfrak{M}[u^{\infty}]^{\tw{n}}[E] \isomorphic \mathfrak{M}_{\operatorname{tor}}^{\tw{n}}[E].
    \end{equation}
    Moreover, the following sequence of finite torsion $\mc{O}_K\textrm{-modules}$ is exact:
    \begin{equation}\label{eq:ses_phiMuinftytor_modE}
        0 \longrightarrow \mathfrak{M}[u^{\infty}]^{\tw{n}}/E \longrightarrow \mathfrak{M}_{\operatorname{tor}}^{\tw{n}}/E \longrightarrow \mathfrak{M}_{\operatorname{tor,} u\operatorname{-tf}}^{\tw{n}}/E \longrightarrow 0.
    \end{equation}
\end{lem}

Finally, note that by reducing the short exact sequence \eqref{eq:ses_phiMtffreebar} modulo $E$ one obtains the following long exact sequence:
\begin{equation}\label{eq:les_phiMtffreebar}
    \begin{aligned}
        0 &\longrightarrow \mathfrak{M}_{\operatorname{tf}}^{\tw{n}}[E] \longrightarrow \mathfrak{M}_{\textup{free}}^{\tw{n}}[E] \longrightarrow \overline{\mathfrak{M}}^{\tw{n}}[E] \longrightarrow \\
        &\qquad \longrightarrow \mathfrak{M}_{\operatorname{tf}}^{\tw{n}}/E \longrightarrow \mathfrak{M}_{\textup{free}}^{\tw{n}}/E \longrightarrow \overline{\mathfrak{M}}^{\tw{n}}/E \longrightarrow 0.
    \end{aligned}
\end{equation}

Using this, we obtain the following.

\begin{lem}\label{lem:Mbar_Etor}
    For any $n\geqslant 0$, there is an isomorphism of $\mc{O}_K\textrm{-modules}$:
    \begin{equation}\label{eq:Mbar_Etor}
        (\overline{\mathfrak{M}}^{\tw{n}}[E])[p^{\infty}] = \overline{\mathfrak{M}}^{\tw{n}}[E] \isomorphic (\mathfrak{M}_{\operatorname{tf}}^{\tw{n}}/E)[p^{\infty}].
    \end{equation}
\end{lem}

\subsection{Length calculations for generalised Breuil--Kisin modules}\label{ss:length-calculations} In this section, we collect some results about lengths of Breuil--Kisin modules, their Frobenius twists, and their $E\textrm{-torsion}$ and reduction modulo $E$.

\begin{lem}\label{prop:len_genBK}
    Let $\mathfrak{M}$ be a generalised Breuil--Kisin module.
    Then, for any $n \geqslant 0$:
    \begin{align}
        \phantom{\ell_{\mc{O}_K}\big(\big((\phi^{*n}\mathfrak{M}) \otimes_{\mathfrak{S}} \mc{O}_K\big)[p^{\infty}]\big)}
        &\begin{aligned}\label{eq:len_genBK_Etor}
            \mathllap{\ell_{\mc{O}_K}\big(\mathfrak{M}^{\tw{n}}[E] \big)} &= \ell_{\mc{O}_K}\big(\mathfrak{M}[u^{\infty}]^{\tw{n}}[E]\big),
        \end{aligned}\\
        &\begin{aligned}\label{eq:len_genBK_modE}
            \mathllap{\ell_{\mc{O}_K}\big(\big(\mathfrak{M}^{\tw{n}}/E\big)[p^{\infty}]\big)} &= \ell_{\mc{O}_K}\big(\mathfrak{M}_{\operatorname{tor,} u\operatorname{-tf}}^{\tw{n}}/E\big) + \ell_{\mc{O}_K}\big(\mathfrak{M}[u^{\infty}]^{\tw{n}}/E\big) + \ell_{\mc{O}_K}\big(\overline{\mathfrak{M}}^{\tw{n}}[E]\big).
        \end{aligned}
    \end{align}
\end{lem}
\begin{proof}
    The claim in \eqref{eq:len_genBK_Etor} follows by combining the isomorphism of $\mc{O}_K\textrm{-modules}$ in \eqref{eq:Mtors_Etors} of Lemma \ref{lem:Mtors_Etors} and \eqref{eq:Munifty_Etors} of Lemma \ref{lem:Munifty_Etors}.
    Next, for \eqref{eq:len_genBK_modE}, let us first note that we have the following equality using \eqref{eq:ses_phiMtortf_modE_pinfty} of Lemma \ref{lem:ses_phiMtortf_modE}:
    \begin{equation*}
        \ell_{\mc{O}_K}\big(\big(\mathfrak{M}^{\tw{n}}/E\big)[p^{\infty}]\big) = \ell_{\mc{O}_K}\big(\mathfrak{M}_{\operatorname{tor}}^{\tw{n}}/E\big) + \ell_{\mc{O}_K}\big(\big(\mathfrak{M}_{\operatorname{tf}}^{\tw{n}}/E\big)[p^{\infty}]\big).
    \end{equation*}
    Moreover, from \eqref{eq:ses_phiMuinftytor_modE} in Lemma \ref{lem:Munifty_Etors}, we have that,
    \begin{equation*}
        \ell_{\mc{O}_K}\big(\mathfrak{M}_{\operatorname{tor}}^{\tw{n}}/E\big) = \ell_{\mc{O}_K}\big(\mathfrak{M}_{\operatorname{tor,} u\operatorname{-tf}}^{\tw{n}}/E\big) + \ell_{\mc{O}_K}\big(\mathfrak{M}[u^{\infty}]^{\tw{n}}/E\big),
    \end{equation*}
    and from \eqref{eq:Mbar_Etor} in Lemma \ref{lem:Mbar_Etor}, we have that $\overline{\mathfrak{M}}^{\tw{n}}[E] \isomorphic (\mathfrak{M}_{\operatorname{tf}}^{\tw{n}}/E)[p^{\infty}]$ as $\mc{O}_K\textrm{-modules}$.
    Hence, the equality in \eqref{eq:len_genBK_modE} follows.
\end{proof}

\begin{lem}\label{lem:len_toret}
    Let $\mathfrak{M}$ be a generalised Breuil--Kisin module of height $i$. Then, for any $n \geqslant 0$:
    \begin{equation}\label{eq:len_toret_modE}
        \ell_{\mc{O}_K}\big(\mathfrak{M}_{\operatorname{tor,} u\operatorname{-tf}}^{\tw{n+1}}/E\big) = \ell_{\mc{O}_K}\big(\mathfrak{M}_{\operatorname{tor,} u\operatorname{-tf}}^{\tw{n}}/E\big).
    \end{equation}
\end{lem}
\begin{proof}
    From Lemma \ref{lem:toretale_BKmod} (3), note that there exists some $N \geqslant 0$ such that we may write $\mathfrak{M}_{\operatorname{tor,} u\operatorname{-tf}}$ as a successive extension of finite free $k\llbracket u \rrbracket\textrm{-modules}$.
    Using this, one quickly reduces to the case when $\mf{M} = \mf{M}_{\mr{tor},u\mr{-tf}}$ is a finite free $k\llbracket u\rrbracket\textrm{-module}$.
    But, writing $\mathfrak{M} \isomorphic k\llbracket u \rrbracket^{\oplus r}$, one computes that $\mathfrak{M}^{\tw{n}}/E \isomorphic (k[u]/u^ek[u])^{\oplus r}$ as $\mc{O}_K\textrm{-modules}$ and for any $n \geqslant 0$, so the claim follows.
\end{proof}

\begin{lem}\label{lem:uinfty_padic_filtration}
    For all $n\geqslant 0$, we have that $\ell_{\mc{O}_K}\big(\mathfrak{M}[u^{\infty}]^{\tw{n}}[E]\big) = \ell_{\mc{O}_K}\big(\mathfrak{M}[u^{\infty}]^{\tw{n}}/E\big)$.
\end{lem}
\begin{proof} This follows from the exact sequence
\begin{equation*}
0\to \mf{M}[u^\infty]^{\tw{n}}[E]\to \mf{M}[u^\infty]^\tw{n}\xrightarrow{E}\mf{M}[u^\infty]^\tw{n}\to \mf{M}[u^\infty]^\tw{n}/E\to 0
\end{equation*}
as length is additive in exact sequences.
\end{proof}

\subsection{Applications to cohomology}

We now apply the abstract statements on Breuil--Kisin modules to prove some results concerning Breuil--Kisin cohomology of smooth formal $\mc{O}_K\textrm{-schemes}$.

\subsubsection{Proof of Proposition \ref{thm:beta-and-gamma-imply-alpha}}\label{sss:proof-of-implication} 

We shall in fact show the stronger claim that under Hypothesis \ref{conj:gamma}, Conjecture \ref{conj:beta} implies \textbf{(L1)} and \textbf{(L2)} from Section \ref{ss:conjectural-framework} (see the discussion after Proposition \ref{prop:lcrys-constant}), i.e., that $\ell^{i,\tw{a}}_\dR\leqslant e\cdot \ell^i_\dR$ and that $\ell^{i,\tw{n}}_\dR$ is non-decreasing in $n$.
By \eqref{eq:deRham_BK_ses}, we have the following equality:
\begin{equation}\label{eq:deRham_BK_len}
    \ell^{i,\tw{n}}_{\textrm{dR}} = \ell_{\mc{O}_K}\left(\big(\mathfrak{M}^{i,\tw{n+1}}/E\big)[p^{\infty}]\right) + \ell_{\mc{O}_K}(\mathfrak{Q}^{i, \tw{n}}).
\end{equation}
Using \eqref{eq:len_genBK_modE} in Lemma \ref{prop:len_genBK}, and Lemmas \ref{lem:len_toret} and \ref{lem:uinfty_padic_filtration}, we see that statement (2) in Conjecture \ref{conj:beta} implies that the first term on the right hand side of \eqref{eq:deRham_BK_len} is non-decreasing in $n$.
Moreover, the second inequality in Hypothesis \ref{conj:gamma} implies that the second term on the right hand side of \eqref{eq:deRham_BK_len} is non-decreasing in $n$. A similar analysis, using statement (1) in Conjecture \ref{conj:beta} and the first inequality in Hypothesis \ref{conj:gamma}, implies the inequality $\ell^i_{\crys} \leqslant \ell^i_{\dR}$ in Conjecture \ref{conj:main-inequality}.

\subsubsection{Verification of Conjecture \ref{conj:beta} in some small cases}\label{sss:verification}

We finally aim to use the results in Sections \ref{ss:prelims-on-BK-modules} and \ref{ss:length-calculations} to prove the following. Below we use $\mf{M}^i$ for $\mr{H}^i\big(\RGamma_{\mathfrak{S}}(X/\mc{O}_K)\big)$ and $\mf{M}^i_N$ as shorthand for $\mr{H}^i\big(\RGamma_{\mathfrak{S}}(X/\mc{O}_K)/p^N\big)$.

\begin{thm}\label{thm:conj-beta-small-index-and-ramification} 
Let $f_1$ and $f_2$ be the functions as in \emph{Conjecture \ref{conj:beta}}, but for the modules $\mf{M}^i[u^\infty]^{\tw{n+1}}[E]$ and $\ov{\mf{M}^i}^{\tw{n+1}}[E]$, respectively. 
Assume that $i \in \{1, 2\}$\footnote{For $i=0$, Conjecture \ref{conj:beta} is obvious using Theorem \ref{thm:ann_uinfty_submod}.} and $e< p(p-1)$.
\begin{enumerate}
\item If $p\cdot \alpha(\mf{M}^i[u^\infty]) \ne e$, then the inequalities for $f_1$ as in \emph{Conjecture \ref{conj:beta}} hold.
\item If $p\cdot\alpha(\mf{M}^i_N[u^\infty])\ne e$ for any $N\gg 0$, then the inequalities for $f_2$ as in \emph{Conjecture \ref{conj:beta}} hold.
\end{enumerate}
Consequently, when the conditions in {(1)} and {(2)} above hold, \emph{Conjecture \ref{conj:beta}} holds.
\end{thm}

The key to proving this result is the following length estimates.
\begin{prop}\label{prop:len_unifty}
    Let $\mathfrak{M}$ be a $u^{\infty}\textrm{-torsion}$ quasi-filtered Breuil--Kisin module of height $i\in\{1,2\}$.
    Assume that $p\alpha(\mf{M}) \ne e$ and $e < p(p-1)$.
    Then, the following two inequalities hold:
    \begin{align}
        \ell_{\mc{O}_K}\big(\mathfrak{M}^{\tw{n+1}}[E]\big) &\geqslant \ell_{\mc{O}_K}\big(\mathfrak{M}^{\tw{n}}[E]\big),\label{eq:len_uinfty_Etor}\\
        e \cdot \ell_{\mc{O}_K}\big(\mathfrak{M}^{\tw{1}}[E]\big) &\geqslant \ell_{\mc{O}_K}\big(\mathfrak{M}^{\tw{a+1}}[E]\big).\label{eq:len_uinfty_Etor_a}
    \end{align}
\end{prop}

\begin{proof}
    Using Theorem \ref{thm:ann_uinfty_submod}, the claim is trivial for $i = 1$.
    For $i = 2$, the claim follows by using Proposition \ref{prop:ann_uinfty_submod} and Lemmas \ref{lem:len_uinfty_ptor} and \ref{lem:len_uinfty_ftor} below.
\end{proof}

\begin{lem}\label{lem:len_uinfty_ptor}
    Assume that $\mathfrak{M}$ is $p\textrm{-torsion}$, then the inequalities \eqref{eq:len_uinfty_Etor} and \eqref{eq:len_uinfty_Etor_a} hold.
\end{lem}
\begin{proof}
    As $\mathfrak{M}$ is a finitely generated $u^{\infty}\textrm{-torsion}$ module over $\mathfrak{S}/p = k\llbracket u \rrbracket$, therefore, from the proof of Lemma \ref{lem:uinfty_padic_filtration} recall that we have,
    \begin{equation*}
        \ell_{\mc{O}_K}\big(\mathfrak{M}^{\tw{n}}[E]\big) = \ell_{\mc{O}_K}\big(\mathfrak{M}^{\tw{n}}/E) = \textstyle\sum_{j=1}^s \min\{e, p^nr_j\},
    \end{equation*}
    for any $n \geqslant 0$, which implies the claim.
\end{proof}

\begin{lem}\label{lem:len_uinfty_ftor}
    Write $\alpha=\alpha(\mf{M})$. Assume that $u^{\alpha} + px$ is in $\textup{Ann}(\mathfrak{M})$ with $1 \leqslant \alpha \leqslant \lfloor \tfrac{e}{p-1} \rfloor < p$ and $x$ in $\mathfrak{S}^{\times}$, and $e \neq p\alpha$.
    Then, the inequalities \eqref{eq:len_uinfty_Etor} and \eqref{eq:len_uinfty_Etor_a} hold.
\end{lem}
\begin{proof}
    For $e \leqslant p-1$, note that $\mathfrak{M}$ is $p\textrm{-torsion}$ by Theorem \ref{thm:ann_uinfty_submod}, and the claim follows by Lemma \ref{lem:len_uinfty_ptor}.
    So, we may assume that $p \leqslant e < p(p-1)$.
    Consider the element $f \defeq u^{\alpha} + px$ assumed to be in $\textup{Ann}(\mathfrak{M})$.
    Note that if $p \in \textup{Ann}(\mathfrak{M})$, for example, when $E \in \textup{Ann}(\mathfrak{M})$ or $E \in \textup{Ann}(\mathfrak{M}^{\tw{1}})$ (see Lemma \ref{lem:EnotinAnn}), then the claim follows by Lemma \ref{lem:len_uinfty_ptor}.
    So we may assume that $p \not\in \textup{Ann}(\mathfrak{M})$, in particular, $E \not\in \textup{Ann}(\mathfrak{M})$ and $E \not\in \textup{Ann}(\mathfrak{M}^{\tw{1}})$.


    
    Let $n \geqslant 0$ and observe that $f^{(n)}$ belongs to $(p,u) \mysetminus (p,u)^2$, and thus $\mf{S}/f^{\tw{n}}$ is a regular local ring of dimension one, and so a DVR with uniformiser $u$.
    As $\mf{M}$ is a finite, $u^\infty\textrm{-torsion}$ $\mf{S}/f\textrm{-module}$, we may write 
    \begin{equation*}
        \mathfrak{M} \isomorphic \textstyle\bigoplus_{j=1}^s \mathfrak{S}/(f, u^{r_j}),
    \end{equation*}
    for some $r_j$ and $s$ in $\mathbb{N}$, and for all $n \geqslant 0$, we see that 
    \begin{equation*}
        \mathfrak{M}^{\tw{n}} \isomorphic \textstyle\bigoplus_{j=1}^s \mathfrak{S}/(f^{\tw{n}}, u^{p^nr_j}).
    \end{equation*}
    Recall that $u$ is a uniformiser of $\mathfrak{S}/f^{\tw{n}}$ and let $\upsilon_n(-)$ denote the $u\textrm{-adic}$ valuation on $\mathfrak{S}/f^{\tw{n}}$.
    Then, we compute that $\upsilon_n(p) = p^n\alpha$, and therefore,
    \begin{equation}
        \upsilon_n(E) = \left\{
            \begin{array}{lll}
                \alpha & \mbox{if } n = 0, \\
                \min\{e, p\alpha\} & \mbox{if } n = 1, \\
                e & \mbox{if } n \geqslant 2.
            \end{array}
	   \right.
    \end{equation}
    Observe that we have,
    \begin{equation*}
        \ell_{\mc{O}_K}\big(\mathfrak{M}^{\tw{n}}[E]\big) = \ell_{\mc{O}_K}\big(\mathfrak{M}^{\tw{n}}/E\big) = \textstyle\sum_{j=1}^s \min\{\upsilon_n(E), p^nr_j\}.
    \end{equation*}
    Hence, it follows that the inequalities \eqref{eq:len_uinfty_Etor} and \eqref{eq:len_uinfty_Etor_a} also hold when $p \not\in \textup{Ann}(\mathfrak{M})$ and $f \in \textup{Ann}(\mathfrak{M})$.
    This allows us to conclude.
\end{proof}

These methods also imply the following, which actually finishes the proof of Theorem \ref{thm:conj-beta-small-index-and-ramification}, as $\mr{H}^i\big(\mr{R}\Gamma_\mf{S}(X/\mc{O}_K)/p^m)$ and its $u^\infty\textrm{-torsion}$ submodule $\mr{H}^i\big(\mr{R}\Gamma_\mf{S}(X/\mc{O}_K)/p^m)[u^\infty]$ are quasi-filtered Breuil--Kisin modules of height $i$ by \cite[Proposition 3.2]{li-liu-uinfty}.

\begin{prop}\label{prop:len_Mbar}
    With the assumptions in \emph{Theorem \ref{thm:conj-beta-small-index-and-ramification}}, the following inequalities hold:
    \begin{equation}\label{eq:len_Mbar_Etor}
        \begin{aligned}
            \ell_{\mc{O}_K}\big(\overline{\mathfrak{M}^i}^{\tw{n+1}}[E]\big) &\geqslant \ell_{\mc{O}_K}\big(\overline{\mathfrak{M}^i}^{\tw{n}}[E]\big),\\
            e \cdot \ell_{\mc{O}_K}\big(\overline{\mathfrak{M}^i}^{\tw{1}}[E]\big) &\geqslant \ell_{\mc{O}_K}\big(\overline{\mathfrak{M}^i}^{\tw{a+1}}[E]\big).
        \end{aligned}
    \end{equation}
\end{prop}
\begin{proof}
    From Lemma \ref{lem:Mbar_MmodpN} and \stacks{061Z}, we have natural inclusions of $\mathfrak{S}\textrm{-modules}$
    \begin{equation*}
        \overline{\mathfrak{M}^i} \hookrightarrow (\mathfrak{M}^i/p^N)[u^{\infty}] \hookrightarrow \mathfrak{M}^i_N[u^{\infty}],
    \end{equation*}
    for some $N \gg 0$.
    In particular, from Proposition \ref{prop:ann_quasifilBKmod} and Proposition \ref{prop:ann_uinfty_submod}, it follows that one of the following holds: 
    \begin{enumerate}
        \item[(1)] $p$ kills $\overline{\mathfrak{M}^i}$.
        \item[(2)] $u^{\alpha} + px$ kills $\overline{\mathfrak{M}^i}$, with $e \neq p\alpha$, $1 \leqslant \alpha \leqslant \lfloor \tfrac{e}{p-1} \rfloor < p$ and $x$ in $\mathfrak{S}^{\times}$.
    \end{enumerate}
    In case (1), observe that the claim follows from Lemma \ref{lem:len_uinfty_ptor}.
    In case (2), for $e \leqslant p-1$, note that $\mathfrak{M}^i_N[u^{\infty}]$ is $p\textrm{-torsion}$ by Theorem \ref{thm:ann_uinfty_submod}, so $\overline{\mathfrak{M}^i}$ is $p\textrm{-torsion}$ and the claim again follows from Lemma \ref{lem:len_uinfty_ptor}.
    Finally, in case (2) for $p \leqslant e \leqslant p(p-1)$, the claim follows by arguing as in the proof of Lemma \ref{lem:len_uinfty_ftor}.
    This concludes our proof.
\end{proof}

\bibliographystyle{test2}
\bibliography{reference}

\end{document}